\documentclass{article}
\usepackage{amsmath}
\usepackage{amsthm}
\usepackage[T1]{fontenc}
\usepackage{amsfonts}
\usepackage{amssymb}
\usepackage{float}
\usepackage{url}
 \usepackage[all]{xy}
 \usepackage{textgreek}
\usepackage{pxfonts}
\usepackage{bbm}
\usepackage{hyperref}
\usepackage{tikz}
\usepackage[fancy]{tikz-inet}
\usetikzlibrary{automata}
\usetikzlibrary{arrows}
\usetikzlibrary{shapes}
\usetikzlibrary{decorations.pathmorphing}
\usetikzlibrary{fit}
\usepackage{ifthen}
\usepackage{stmaryrd} 

\tikzstyle{every picture}=[
  >=stealth', 
shorten >=1pt, 
node distance=1.44cm,
bend angle=45,
initial text=,
  every state/.style={inner sep=0.75mm, minimum size=1mm},
font=\scriptsize,
]

\newtheorem{definition}{Definition}
\newtheorem{lemma}{Lemma}
\newtheorem{example}{Example}
\newtheorem{proposition}{Proposition}
\newtheorem{theorem}{Theorem}
\newtheorem{claim}{Claim}

\newtheorem{corollary}{Corollary}
\newtheorem{conj}{Conjecture}

\def\N{{\mathbb N}}

\def\bigboxplus{{\boxed{+}}}

\newcount\ran
\def\pro#1#2#3#4#5#6#7#8{
\ran=#5
 \pgfmathsetcount{\ran}{\ran+1}
\draw[fill=black] (#1,#2) rectangle (#1+#3,#2+#4);
\node[] (bla) at (#1+#3/2,#2+#4/2) {\color{white} #7};
\ifthenelse{#5>0}{\foreach \x in {1,...,#5} {
 \node[] (uu#8\x) at (#1+\x*#3/\ran,#2+#4-0.11){};
 \node[] (u#8\x) at (#1+\x*#3/\ran,#2+#4*1.7-0.11){};
}}{}
\pgfmathsetcount{\ran}{#6+1}
\ifthenelse{#6>0}{\foreach \x in {1,...,#6} {
 \node[] (dd#8\x) at (#1+\x*#3/\ran,#2+0.11){};
 \node[] (d#8\x) at (#1+\x*#3/\ran,#2-#4*0.7+0.11){};
}}{}

}

\def\Apro#1#2#3#4#5#6#7#8#9{
\ran=#5
 \pgfmathsetcount{\ran}{\ran+1}
\draw[fill=#8] (#1,#2) rectangle (#1+#3,#2+#4);
\node[] (bla) at (#1+#3/2,#2+#4/2) {\color{black} #7};
\ifthenelse{#5>0}{\foreach \x in {1,...,#5} {
 \node[] (uu#9\x) at (#1+\x*#3/\ran,#2+#4-0.11){};
 \node[] (u#9\x) at (#1+\x*#3/\ran,#2+#4*1.7-0.11){};
}}{}
\pgfmathsetcount{\ran}{#6+1}
\ifthenelse{#6>0}{\foreach \x in {1,...,#6} {
 \node[] (dd#9\x) at (#1+\x*#3/\ran,#2+0.11){};
 \node[] (d#9\x) at (#1+\x*#3/\ran,#2-#4*0.7+0.11){};
}}{}

}

\def\socket#1#2#3#4#5{
\ran=#5
 \pgfmathsetcount{\ran}{\ran+1}
\ifthenelse{#5>0}{\foreach \x in {1,...,#5} {
 \draw[fill=black] (#1+\x*#3/\ran,#2+#4*1.7-0.3) circle [radius=0.03];
}}{}
 
}

\def\plug#1#2#3#4#5{
\ran=#5
 \pgfmathsetcount{\ran}{\ran+1}
\ifthenelse{#5>0}{\foreach \x in {1,...,#5} {
 \draw[fill=white,color=white] (#1+\x*#3/\ran,#2) circle [radius=0.03];
}}{}
 
}

\def\chip #1#2#3#4#5#6#7{
\ran=#5
 \pgfmathsetcount{\ran}{\ran+1}

\draw[fill=black] (#1,#2) rectangle (#1+#3,#2+#4);

\node[] (bla) at (#1+#3/2,#2+#4/2) {\color{white} #7};
\ifthenelse{#5>0}{\foreach \x in {1,...,#5} {
 \node[] (uu\x) at (#1+\x*#3/\ran,#2+#4-0.11){};
 \node[] (u\x) at (#1+\x*#3/\ran,#2+#4*1.7-0.11){};
 \draw (uu\x) -- (u\x);
}}{}

\socket {#1}{#2}{#3}{#4}{#5}
\plug {#1}{#2}{#3}{#4}{#6}
}

\def\tikdanseq#1{
\begin{array}{c}
\begin{tikzpicture}
#1
\end{tikzpicture}
\end{array}
}

\def\hmat#1#2#3{{\displaystyle\mathop{\mathfrak #1}^{#2}_{#3}}}
\def\TI{{\mathtt I}}
\def\TJ{{\mathtt J}}
\def\TK{{\mathtt K}}
\def\obj{\mathrm{Obj}}

\begin{document} 

\title{Multilinear representations of Free PROs}
\author{\'E.~Laugerotte, J.-G.~Luque, L.~Mignot, F.~Nicart}
\date{\today}

 \maketitle{}
 \begin{abstract}

We describe a structure of PRO on hypermatrices. This structure allows us to define multilinear representations of PROs and in particular of free Pros. As an example of applications, we investigate the relations of the representations of Pros with the theory of automata.
 \end{abstract}
\section{Introduction}
 The \emph{PRO} theory is a way to embed an abstract theory of 
 operators into the formalism of categories. Informally, a PRO is a set of 
 operators having several inputs and outputs and which can be composed 
 by branching ones to the others or by juxtaposing them. The name PRO 
 means \emph{PROduct category} and comes from the interpretation of the 
 juxtaposition as a tensor product.
 The first occurrence of the notion of PRO dates back from 
 the early works of Mac Lane  \cite{MacLane1,MacLane} and was used in a special case
 by Boardman and Vogt \cite{BV1,BV} to model homotopy. The PRO theory 
 has many connexions with several fields; sometimes 
 explicitly as in algebraic topology \cite{BV} or algebraic combinatorics 
 \cite{BG}. Sometimes the connexion was not explicitly identified but the 
 underlying algebraic structure involves naturally a PRO; we give a few examples in 
 the end of the paper (Section \ref{automata}, appendices \ref{QI} and \ref{TL}).
 
  As for many algebraic structures, the notion of 
 freeness is well defined for PROs. This means that there exist objects, 
 called free PROs, having the universal property. This property allows 
 to construct onto morphisms from a free PRO to any PRO having the same 
 set of generators. At the other end, we define  PROs
 on hypermatrices. The goal of the paper is to investigate the morphisms 
 from free PROs to  PROs on hypermatrices. These morphisms are called 
 multilinear representations.
 
 The paper is organized as follows. Section \ref{PRO} is devoted to the 
 general study of the notion of PRO. More precisely, we first recall 
 (Section \ref{what}) the original definition of PRO in category 
 theory. In Section \ref{CombiPro}, we recall the alternative 
 constructive definition. We investigate also several generalizations 
 and variations 
 like colored PRO (Section \ref{colorPro}) and ModPro (Section \ref{ModPro}). 
 This last notion allows us to consider some PROs as modules over a 
 semiring. The concept of subPros and quotient are also recalled (see 
 Section \ref{Quot}).\\
  Free Pros are defined and studied in Section \ref{Free}.
 In particular in  Section \ref{CombiPro}, we 
 investigate the case where all the generators have neither empty 
 input nor empty output. Such a free PRO is called \emph{Circuit 
 PROs}. This is a particularly important example 
 because the elements can be nicely represented by some kind of electronic 
 circuits. We discuss also on the difficulties to find a combinatorial 
 representation for general free PROs (Section \ref{Others}) and we 
 describe a colored PRO on generalized paths on hypergraphs in Section 
 \ref{Path}.\\
 The PRO structure on hypermatrices is defined in Section \ref{PROHyp} 
 and
  we investigate the properties of the composition and juxtaposition with respect to the Kronecker product (Section 
  \ref{Kronecker}) and the sum (Section \ref{Sum}).\\
  This allows us to define multilinear representations of PROs (Section 
  \ref{RepFreePros}). We investigate the behavior with respect to the 
  sum on hypermatrices  (Section \ref{RepQuasi}) and the interpretation 
  in terms of generalized paths (Section \ref{RepPath}).\\
  As an example of applications, in the last section, 
  we investigate the links with the 
  theory of automata  \ref{automata}.
  
\section{Pros and their generalizations\label{PRO}}
This section is devoted to the definitions and the properties of several structures related to the notion of Pro.
Pros are bigraded sets of objects (according to the number of inputs and the number of outputs). In Section \ref{what} we compare two definitions
of Pros which can be found in literature. The first one comes directly 
from the category theory and the second is more combinatorial.
 \subsection{What  PROs are\label{what}}

In modern algebra, \emph{PRO}s are defined in category theory as strict monoidal categories whose objects 
are the natural numbers (including zero), and whose tensor product is the addition. Before giving a
 more combinatorial way to define PROs, let us explain what this first definition means. First we recall that the aim of category 
 theory is to provide tools to describe in an abstract way classes of mathematical objects (sets, monoids, algebras, \emph{etc}.).
  A class is encoded by a graph whose vertices are the objects and whose arrows encode the morphisms.
   More precisely,  a \emph{category} $\mathcal C$ is constituted with three entities:
\begin{itemize}
\item The class $\obj(\mathcal C)$ of its \emph{objects},
\item The class $Hom(\mathcal C)$ of its \emph{arrows} (also called \emph{morphisms}). Each arrow has a source and a target
 which belong to $\obj(\mathcal C)$. The set of the arrows whose source is $a$ and whose target is $b$ is denoted
  by $Hom_C(a,b)$ (or $Hom(a,b)$ when there is no ambiguity).
\item  A binary operation $\circ$, called \emph{composition} such that for any $a,b,c\in \obj(\mathcal C)$,
 $\circ:Hom(b,c)\times Hom(a,b)\rightarrow Hom(a,c)$. When there is no ambiguity, we omit to write $\circ$. 
 The composition must satisfy two properties,
\begin{itemize}
\item {\it Associativity.} $f(gh)=(f g) h$
\item {\it Identity.} For any object $a\in \obj(\mathcal C)$, there exists a unique morphism $1_a\in Hom(a,a)$ called identity on $a$. 
These arrows satisfy $1_b f=f 1_a=f$ for any $f\in Hom(a,b)$. When there is no ambiguity the identity is simply denoted by $1$.
\end{itemize}

\end{itemize}
\emph{Functors} are morphisms between categories which encode maps preserving identity and compositions. 
Let $\mathcal C$ and $\mathcal D$ be two categories. A functor $F$ from $\mathcal C$ to $\mathcal D$ 
is a map that associates to each object $a\in\mathcal C$ an object $F(a)\in\mathcal  D$ 
and to each arrow $f\in Hom(a,b)$ an arrow $F(f)\in Hom(F(a),F(b))$ such that the image of the $1_a$ is $1_{F(a)}$ and 
$F(f\circ g)=F(f)\circ F(g)$.
If $F$ and $G$ are functors between two categories $\mathcal C$ and $\mathcal D$, a \emph{natural transformation} $\eta$ is a family
 of morphisms such that for each object $a\in\mathcal C$ there exists a morphism $\eta_a:F(a)\rightarrow G(a)$ in $\mathcal D$ called \emph{component}
  of $\eta$ at 	$a$ satisfying $\eta_b\circ F(f)=G(f)\circ \eta_a$ for every morphism $f:a\rightarrow b$. In the aim to simplify the notation, we use $\eta$ instead of $\eta_a$ when there is no ambiguity.

The \emph{product category} $\mathcal C\times\mathcal  D$ is the category whose objects are the pairs $(a,b)$ with $a\in\mathcal C$
 and $b\in\mathcal  D$ and whose arrows are the pairs of morphisms $(f,g)$ such that $f\in Hom(a,a')$ and $g\in Hom(b,b')$; the composition is the component-wise composition. More precisely, 
\begin{equation}\label{prodcat}
(f,g)\circ (f',g')=(f\circ f',g\circ g')\mbox{ and }1_{(a,b)}=(1_a,1_b).
\end{equation}
A \emph{monoidal category} is a category $\mathcal M$ equiped with a 
(bi)functor $\otimes: \mathcal M\times \mathcal M\rightarrow \mathcal M$, an 
object $I$ called the \emph{unit} object or the \emph{identity} object, three natural isomorphisms
\begin{enumerate}
\item the \emph{associator} $\alpha$ with components $\alpha_{a,b,c}:(a\otimes b)\otimes c\simeq a\otimes (b\otimes c)$;
\item the \emph{left unitor} $\lambda$ with components $\lambda_a:I\otimes a\simeq a$,
\item the \emph{right unitor} $\rho$ with components $\rho_a:a\otimes I\simeq a$
\end{enumerate}
satisfying $ (1\otimes \alpha)\alpha(\alpha\otimes 1)=\alpha\alpha:((a\otimes b)\otimes c)\otimes d)\rightarrow a\otimes(b\otimes (c\otimes d))$ and
$(1\otimes\lambda)\alpha=\rho\otimes 1:(a\otimes 1)\otimes b\rightarrow a\otimes b$ for any objects $a, b, c, d$ in $\mathcal M$.\\
A  monoidal category $(\mathcal M,\alpha,\lambda,\rho,I)$ is said \emph{strict} 
 if the natural isomorphisms $\alpha$, $\lambda$ and $\rho$ are identities. In other words 
 \[a\otimes(b\otimes c)=(a\otimes b)\otimes c\] for any $a,b,c\in \obj(\mathcal M)$,
   \[f\otimes (g\otimes h)=(f\otimes g)\otimes h\] for any $f,g,h\in Hom(\mathcal M)$, 
 $I\otimes a=a=a\otimes I$ for any $a\in \obj(\mathcal M)$, and $1\otimes f=f\otimes I=f$ for any $f\in Hom(\mathcal M)$ (\emph{i.e.} the left and right multiplication by $I$ are  the identity functor). Notice that, for any $f\in Hom(a,a')$ and $g\in Hom(b,b')$, $f\otimes g$ is an arrow whose source is $a\otimes b$ and target is $a'\otimes b'$. Furthermore from (\ref{prodcat}), we have
 an additional  identity : 
\begin{equation}\label{mass2}
(f\circ g)\otimes (f'\circ g')=(f\otimes f')\circ (g\otimes g')\end{equation}
 for any $f\in Hom(a,b), g\in Hom(b,c),  f'\in Hom(a',b')$, and $g'\in Hom(b',c')$
 

A \emph{PRO} $\mathcal P$ is a strict monoidal category whose object are  the natural numbers and the tensor product sends 
$(m, n)$ to $m+n$. Hence, if $f\in Hom(m,n)$ and $g\in Hom(m',n')$ then we have $f\otimes g\in Hom(m+m',n+n')$. 
Obviously, the unit object of the category is $0$.

\begin{example}
\rm
 Suppose that $Hom(m,n)=\{f_{m,n}\}$. In other words, the  morphisms are the edges of  
the complete graph whose vertices are the integers.
In addition, we define a composition $\circ$ satisfying $f_{m,n}\circ f_{n,p}=f_{m,p}$ 
 for any integers $n,m,p\in\N$.
Obviously, this defines a category. Now, if we set $f_{m,n}\otimes f_{m',n'}=f_{m+m',n+n'}$, then the category is endowed with a structure of PRO.
\end{example}

\begin{example}
\rm The category FinSet, whose objects are all finite sets and
 whose morphisms are all functions between them, is a PRO.  Each integer $n$ is identified with a unique set $\{0,\dots,n-1\}$ 
 and a morphism from $m$ to $n$ is a $m$-tuple $(\alpha_0,\dots,\alpha_{m-1})$ such that $0\leq\alpha_i\leq n-1$ for each $0\leq i\leq m-1$.
\end{example}
This definition being very formal, 
some properties, like formula (\ref{mass2}), are implicit. Furthermore we have 
 \begin{equation}\label{1_n}1_n=\overbrace{1_1\otimes\dots\otimes1_1}^{n\mbox{ \footnotesize times}}\end{equation} for $n>0$ as a consequence 
 of $1_{m}\otimes 1_{n}=1_{m+n}$ and formula (\ref{prodcat}), and
$
1_0
$ is the identity on $0$.
The last equality is a consequence of formula (\ref{mass2}).

  In the next section, we give an  alternative combinatorial definition and make a parallel with the first one. 
  Nevertheless, the algebraicity can be used to propose structures naturally derived from the notion of PRO. 
  In our paper, we will use two kinds of processes. 
  The first one consists in endowing each $Hom(a,b)$ with an additional algebraic structure and some compatibility conditions. 
  The second one consists in changing the objects in the strict monoidal category. For instance, we will investigate the notion of
   colored PROs which are strict monoidal categories whose objects are vectors of colours and whose tensor 
   product is the catenation of the vectors. 
\subsection{A combinatorial definition for PROs\label{CombiPro}}
In this section, we give an alternative  definition for PROs  used in the context of algebraic combinatorics (see \emph{e.g.} \cite{BG}).

A \emph{PRO}\ 
  is a bi-graded set 
$\mathcal P=\bigcup_{m,n\in \mathbb N}\mathcal P_{m,n}$ endowed with two binary 
operations $\leftrightarrow:\mathcal P_{m,n}\times \mathcal P_{m',n'}\rightarrow \mathcal P_{m+m',n+n'}$
 (\emph{horizontal composition}) and $\updownarrow:\mathcal P_{m,n}\times \mathcal P_{n,p}\rightarrow \mathcal P_{m,p}$
  (\emph{vertical composition}) which satisfy the following rules:
\begin{itemize}
\item\emph{Horizontal associativity.} For any $\mathfrak p,\mathfrak p',\mathfrak p''\in \mathcal P$, we have: 
\begin{equation}\label{hass}
\mathfrak p\leftrightarrow (\mathfrak p'\leftrightarrow \mathfrak p'')=(\mathfrak p\leftrightarrow \mathfrak p')\leftrightarrow \mathfrak p''.
\end{equation}
\item\emph{Vertical associativity.} For each $\mathfrak p\in \mathcal P_{m,n}$, $\mathfrak q\in \mathcal P_{n,p}$ 
and $\mathfrak r\in\mathcal P_{p,q}$ we have
\begin{equation}\label{vass}
\begin{array}{c}\mathfrak p\\\updownarrow\\ \left(\begin{array}{c}\mathfrak q\\\updownarrow \\\mathfrak r\end{array}\right)\end{array}=\begin{array}{c}\left(\begin{array}{c}\mathfrak p\\\updownarrow\\ \mathfrak q\end{array}\right)\\\updownarrow \\\mathfrak r.\end{array}
\end{equation}
\item \emph{Interchange law.} For each $\mathfrak p\in \mathcal P_{m,n}$, $\mathfrak q\in \mathcal P_{n,p}$,  
$\mathfrak p'\in \mathcal P_{m',n'}$ and $\mathfrak q'\in \mathcal P_{n',p'}$ we have
\begin{equation}\label{mass}
\left(\begin{array}{c}\mathfrak p\\\updownarrow \\\mathfrak q\end{array}\right)\leftrightarrow \left(\begin{array}{c}\mathfrak p'\\\updownarrow\\\mathfrak q'\end{array}\right)=\begin{array}{c}(\mathfrak p\leftrightarrow\mathfrak p')\\\updownarrow \\(\mathfrak q\leftrightarrow\mathfrak q').\end{array}
\end{equation} 
\item \emph{Graded vertical unit.} For each $n\in\mathbb N$ there exists a unique graded unit $Id_n$. More precisely,
 one has 
\begin{equation}\label{grneut} \begin{array}{c}Id_m\\\updownarrow\\\mathfrak p\end{array}=
\begin{array}{c}\mathfrak p\\\updownarrow\\ Id_n\end{array}=\mathfrak p\end{equation} 
for each $\mathfrak p\in \mathcal P_{m,n}$. Moreover we must have \begin{equation}\label{grneut3}Id_n=I
d_1^{\leftrightarrow n}=\overbrace{Id_{1}\leftrightarrow\cdots\leftrightarrow Id_{1}}^{n\mbox{ \footnotesize times}}\end{equation} for $n\geq 1$. 
\item \emph{Horizontal unit.} There exists an element  $\varepsilon$ such that 
for each $\mathfrak p\in\mathcal P$ we have \begin{equation}\label{hneut}\varepsilon\leftrightarrow\mathfrak p=\mathfrak p\leftrightarrow \varepsilon=\mathfrak p.\end{equation}
\end{itemize}

For simplicity, when there is no ambiguity, we  denote $Id_1=|$, $Id_n=\overbrace{|\cdots|}^n$ and $Id_0=\oslash$.
Observe that $(\mathcal P,\leftrightarrow)$ is a monoid whose unit is $\oslash$ and each $(\mathcal P_{n,n},
\updownarrow)$ is a monoid whose  unit is $\overbrace{|\cdots|}^n$.
Each element $p\in \mathcal P_{m,n}$ with $m,n>0$ will be called \emph{pluggable}.

Naturally,  \emph{morphisms}\footnote{In the algebraic definition, these morphisms are monoidal functors.} of PROs are defined as  
maps $\phi:\mathcal P\rightarrow \mathcal Q$ that carry the algebraic structures. More precisely, a morphism $\phi$ satisfies the following properties:
\begin{itemize}
\item \emph{Graded map.} $$\phi(\mathcal P_{m,n})\subset \mathcal Q_{m,n}.$$
\item \emph{Morphism of monoid.}
 $$\phi(\mathfrak p\leftrightarrow \mathfrak q)=\phi(\mathfrak p)\leftrightarrow \phi(\mathfrak q).$$
\item \emph{Compatibility with the vertical composition.} $$\phi\left(\begin{array}{c}\mathfrak p\\\updownarrow\\\mathfrak q\end{array}\right)=\begin{array}{c}\phi(\mathfrak p)\\\updownarrow\\\phi(\mathfrak q)\end{array}.$$
\item \emph{The image of a unit is a unit.} 
$$\phi\left(\overbrace{|\cdots|}^n\right)=\overbrace{|\cdots|}^n,$$ for each $n\in\mathbb N$.
\end{itemize}
Let us compare the two definitions:
\begin{center}
\begin{tabular}{|c|c|}
\hline
Combinatorial definition& Categorial definition\\
\hline 
$\mathcal P_{m,n}$&$Hom(m,n)$\\
Vertical composition& Composition of the morphisms\\
Horizontal composition& Tensor product\\
$Id_n$&Identity on $n$\\
$\varepsilon$&Identity on the unit object\\
Morphism& Monoidal functor\\\hline
\end{tabular}
\end{center}
Notice that the interchange law and the equality $Id_n=Id_1^{\leftrightarrow n}$ are now axioms whilst they are deduced from the strict monoidal structure in the first definition. Also we have
\begin{equation}Id_0=\begin{array}{c}Id_0\leftrightarrow\varepsilon\\
\updownarrow\\
\varepsilon\leftrightarrow Id_0
\end{array}
=
\begin{array}{c}Id_0\\
\updownarrow\\
\varepsilon
\end{array}
\leftrightarrow
\begin{array}{c}\varepsilon\\
\updownarrow\\
Id_0
\end{array}
=
\varepsilon\leftrightarrow\varepsilon=\varepsilon\end{equation}
Hence, the horizontal unit is unique: $\varepsilon=Id_0$ as for the algebraic definition. 
\subsection{Colored PROs\label{colorPro}}
The notion of colored PRO extends the notion of PRO to more general graded sets with finer graduation. Consider a set $\mathbf C$ of colors. A \emph{colored PRO} is a graded set $$\mathcal P=\bigcup_{n,m\geq 0}\bigcup_{I\subset \mathbf C^{n}\atop J\subset\mathbf C^{m}}\mathcal P_{I,J}$$ endowed with two laws $\leftrightarrow: \mathcal P_{I,J}\times \mathcal P_{I',J'}\longrightarrow\mathcal P_{II',JJ'}$, where $II'$ is the catenation of the lists $I$ and $I'$ and 
$\updownarrow: \mathcal P_{I,J}\times \mathcal P_{J,K}\longrightarrow\mathcal P_{I,K}$ satisfying 
the horizontal associativity, the vertical associativity,  and the interchange law and such that
\begin{itemize}
\item \emph{Graded vertical unit}: For each $n\geq 0$ and $I\in\mathbf C^n$ there exists a unique element $Id_I\subset\mathcal P_{I,I}$ verifying 
$$\begin{array}{c}Id_I\\\updownarrow\\\mathfrak p\end{array}=\mathfrak p$$ for each $\mathfrak p\in \mathcal P_{I,J}$ and $$\begin{array}{c}\mathfrak q\\\updownarrow\\ Id_I\end{array}=\mathfrak q$$ for each $\mathfrak q\in \mathcal P_{K,I}$. Moreover we must have $Id_I=Id_{c_1}\leftrightarrow \cdots\leftrightarrow Id_{c_n}$ for $I=[c_1\dots,c_n]$. 
\item \emph{Horizontal unit}: there exists an element  $\varepsilon\in \mathcal P_{[],[]}$ such that 
for each $\mathfrak p\in\mathcal P$ we have $$\varepsilon\leftrightarrow\mathfrak p=\mathfrak p\leftrightarrow \varepsilon=\mathfrak p.$$ 
\end{itemize}
As in the case of PRO we have
\begin{lemma}
The horizontal unit is unique: $\varepsilon=Id_{[]}$.
\end{lemma}
In terms of category, a colored PRO is a strict monoidal category such that the objects are vectors of colors and the tensor product sends $([c_{1},\dots,c_{m}],[d_{1},\dots,d_{n}])$ to the
catenation $[c_{1},\dots,c_{m},d_{1},\dots,d_{n}]$ of the two vectors. The unit object of the category is the empty vector $[]$. 
\subsection{ModPro\label{ModPro}}

\def\aa{\mathtt a}
\def\bb{\mathtt b}
\def\cc{\mathtt c}
\def\dd{\mathtt d}
A $\mathbb K$-\emph{ModPro} $\mathtt M=\bigcup_{n,m\geq 0}\mathtt M_{n,m}$ is a PRO equipped with two additional 
operations $+$ and $\cdot$ which confer to each $\mathtt M_{n,m}$ a structure of $\mathbb K$-module satisfying
the following rules:
\begin{itemize}
\item {\it Left and right distributivities}.
\begin{equation}\label{lrdist}
\mathtt p\leftrightarrow (\mathtt q+\mathtt r)=(\mathtt p\leftrightarrow\mathtt q)+(\mathtt p\leftrightarrow\mathtt r),\  
 (\mathtt q+\mathtt r)\leftrightarrow\mathtt p=(\mathtt 
 q\leftrightarrow\mathtt p)+(\mathtt r\leftrightarrow\mathtt p),
\end{equation}
\item {\it Up and down distributivities}.
\begin{equation}\label{uddist}
\begin{array}{c}
(\mathtt p_1+\mathtt p_2)\\
\updownarrow\\
\mathtt q
\end{array}=\left(\begin{array}{c}
\mathtt p_1\\
\updownarrow\\
\mathtt q
\end{array}\right)+
\left(\begin{array}{c}
\mathtt p_2\\
\updownarrow\\
\mathtt q
\end{array}\right),\,
\begin{array}{c}
\mathtt q\\
\updownarrow\\
(\mathtt p_1+\mathtt p_2)
\end{array}=\left(\begin{array}{c}
\mathtt q\\
\updownarrow\\
\mathtt p_1
\end{array}\right)+
\left(\begin{array}{c}
\mathtt q\\
\updownarrow\\
\mathtt p_2
\end{array}\right),
\end{equation}
\item $a\cdot (\mathtt p\leftrightarrow \mathtt q)=(a\cdot \mathtt p)\leftrightarrow \mathtt q=\mathtt p\leftrightarrow{}
(a\cdot \mathtt q)$.
\item $a\cdot\left(\begin{array}{c}\mathtt p\\\updownarrow\\\mathtt  q \end{array}\right)={}
\left(\begin{array}{c}a\cdot \mathtt p\\\updownarrow\\\mathtt  q \end{array}\right)={}
\left(\begin{array}{c}\mathtt p\\\updownarrow\\a\cdot \mathtt  q \end{array}\right).
$
\end{itemize}

Notice that $(\mathtt M_{0,0},+,\leftrightarrow)$ is a semiring.
ModPros can alternatively be defined in terms of enriched categories 
(see Appendix \ref{AppModPro}). 



\subsection{SubPROs and Quotients\label{Quot}}
A \emph{subPRO} of a PRO $\mathcal P$ is a subset $\mathcal P'\subset 
\mathcal P$ containing any $Id_{n}$ and which is stable for the 
compositions $\leftrightarrow$ and $\updownarrow$.\\
As for many algebraic structures, there is a notion of quotient of PRO. A \emph{congruence} $\equiv$ is an equivalence relation which is compatible
with the graduation and the two compositions. This means 
\begin{enumerate}
	\item $\mathfrak p\equiv \mathfrak q$ implies $\mathfrak p,\mathfrak q\in\mathcal P_{m,n}$ for some $m,n\in\mathbb N$,
	\item $\mathfrak p\equiv\mathfrak p'\in\mathcal P_{m,n}$, $\mathfrak q\equiv\mathfrak q'\in\mathcal P_{m',n'}$ implies 
	$\mathfrak p\leftrightarrow\mathfrak q\equiv \mathfrak 
	p'\leftrightarrow\mathfrak q'$, and
	\item  $\mathfrak p\equiv\mathfrak p'\in\mathcal P_{m,n}$, $\mathfrak q\equiv\mathfrak q'\in\mathcal P_{n,p}$ implies 
	$\begin{array}{c}\mathfrak p\\\updownarrow\\\mathfrak q\end{array}\equiv \begin{array}{c}\mathfrak p'\\\updownarrow\\\mathfrak q'\end{array}$.
	\end{enumerate}
As a consequence,  the quotient set $\mathcal 
P/_{\equiv}=\bigcup_{m,n}\mathcal P_{m,n}/_\equiv$ inherits  a structure of PRO  
 from $\mathcal P$. 
The PRO $\mathcal 
P/_{\equiv}$ is called
the \emph{quotient} of $\mathcal P$ by $\equiv$.
This also defines a morphism of PRO $\phi_{\equiv}:\mathcal P\rightarrow\mathcal P/_{\equiv}$ sending each element 
on its class.\\
Conversely, if $\phi:\mathcal P\rightarrow\mathcal Q$ is a morphism of PRO then the equivalence $\equiv_{\phi}$, defined by 
$\mathfrak p\equiv\mathfrak q$ if and only if $\phi(\mathfrak p)=\phi(\mathfrak q)$, is a congruence and $\mathcal P/_{\equiv_{\phi}}$ is
isomorphic to $\phi(\mathcal P)\subset \mathcal Q$, the subPRO of $\mathcal Q$ which is the image of $\mathcal P$.\\
Similarly, one defines \emph{quotient of ModPro} as a quotient of PRO such that the restriction to each graded component is a quotient of $\mathbb K$-module.

\section{Freeness\label{Free}}

Freeness is a notion generalizing the concept of basis in a vector 
space to other structures. In category theory, free objects satisfies,
when they exist, the universal property. 
Consider the forgetful (faithful) functor $\mathbf F$ sending each PRO to its underlying graded set.
 Let $\mathcal X=\bigcup_{m,n}\mathcal X_{m,n}$ be a bigraded set. 
 A PRO $\mathcal P$ is the free PRO on $\mathcal X$ if there exists a canonical bigraded injection 
 $i:\mathcal X\rightarrow \mathbf F(\mathcal P)$ satisfying the following universal property: 
 for any PRO $\mathcal P'$ and any bigraded map $f:\mathcal X\rightarrow \mathbf F(\mathcal P')$ 
 there exists a unique morphism of PRO $\overline f:\mathcal P\rightarrow \mathcal P'$ such that $f=\mathbf F(g)\circ i$. 
 When it exists, such a PRO is denoted by $\mathcal F(\mathcal X)$. In this case,
  if a PRO $\mathcal P$ is generated by $\mathcal X$ then
 there exists an onto morphism from $\psi:\mathcal F(\mathcal X)\rightarrow \mathcal P$ sending each element $\mathfrak x$ to itself. 
 In other words, $\mathcal P$ is isomorphic to the quotient $\mathcal F(\mathcal X)/_{\equiv_{\psi}}$.
\subsection{Circuit PROs\label{CircPro}}
We consider a graded set $\mathcal C=\bigcup_{n,m\geq 1}\mathcal C_{n,m}$
 such that each $\mathcal C_{n,m}$ is finite. The elements of $\mathcal C$ are the \emph{chips}. 
 Each chip $\mathfrak c\in \mathcal C_{n,m}$ is graphically represented by a labeled box with $n$ \emph{outputs},
  drawn from the top and $m$ \emph{inputs} drawn on the bottom of the box (See in Fig \ref{ChipsA} for an example of a $(3,4)$-chip).

\medskip

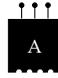
\begin{figure}[h]
\begin{center}
\begin{tikzpicture}
\chip 00{0.7}{0.7}34{A}
\end{tikzpicture}

\end{center}

\caption{A chip with $3$ outputs and $4$ inputs\label{ChipsA}}
\end{figure}
A \emph{circuit} is an object with outputs and inputs constructed inductively as follow.
\begin{definition}
A circuit is
\begin{enumerate}
\item either the \emph{empty circuit} denoted by $\oslash$ with $0$ output and $0$ input,
\item or a \emph{a wire} denoted by $\multimapdotbothBvert$ with $1$ output and $1$ intput,
\item or a chip,
\item or the juxtaposition $\mathfrak p\ \mathfrak p'$ of two circuits $\mathfrak p$ and $\mathfrak p'$ 
with the rule $\oslash  \mathfrak p=\mathfrak p\oslash=\mathfrak p$. 
The outputs (resp. the inputs) of $\mathfrak p\ \mathfrak p'$ is obtained by catening the outputs
(resp. the inputs) of $\mathfrak p$ and the outputs (resp. the inputs) of $\mathfrak p'$. We denote by $\leftrightarrow$ the operation of juxtaposition.
\item or the connection of a $(m,n)$-circuit $\mathfrak p$ and a $(n,p)$-circuit $\mathfrak q$ 
obtained by connecting the $i$-th input of $\mathfrak p$ to the $i$-th output of $\mathfrak q$ with the following two rules:
\begin{itemize}\item connecting $\multimapdotbothBvert\dots \multimapdotbothBvert$ to the outputs or the inputs of a circuit let 
the circuit unchanged,
\item connecting $\oslash$ with itself is still $\oslash$.
\end{itemize}
 The outputs of the connection of $\mathfrak p$ and $\mathfrak q$ are the outputs of $\mathfrak p$ and its inputs are the inputs of $\mathfrak q$. We  denote by $\updownarrow$ the operation of connection.
\end{enumerate}
Chips and  wires are called \emph{elementary circuits}.\\
We denote by $\mathcal Circ(\mathcal C)$ the set of the circuits obtained from the chips of $\mathcal C$.
\end{definition}
Notice that the only circuit having $0$ input or $0$ output is $\oslash$ and that any non empty circuit 
can be obtained by juxtaposing and connecting elementary circuits (see 
Figure \ref{decompcirc} for an example).
\begin{figure}[h]
\begin{center}
\begin{tikzpicture}
\pro 02{0.5}{0.5}34A{A}
\pro {0.18}1{0.5}{0.5}34A{AA}
\pro {0.5}0{0.5}{0.5}21B{B}
\socket 02{0.5}{0.5}3
\plug {0.5}{0}{0.5}{0.5}1
\draw[color=white,thick] (ddA3)--(dA3);
\draw[color=white,thick] (ddA2)--(dA2);
\draw[color=white,thick] (ddA1)--(dA1);
\draw[color=white,thick] (ddA4)--(dA4);
\draw[color=white,thick] (uuAA1)--(uAA1);
\draw[color=white,thick] (uuAA2)--(uAA2);
\draw[color=white,thick] (uuAA3)--(uAA3);
\draw[color=white,thick] (ddAA3)--(dAA3);
\draw[color=white,thick] (ddAA2)--(dAA2);
\draw[color=white,thick] (ddAA4)--(dAA4);
\draw[color=white,thick] (ddAA1)--(dAA1);
\draw[color=white,thick] (uuB1)--(uB1);
\draw[color=white,thick] (uuB2)--(uB2);
\draw[color=white,thick] (ddB1)--(dB1);
\draw (ddA3)--(uuAA1);
\draw (ddA4)--(uuAA2);
\draw (0.55,2.6)--(uuAA3);
\draw (ddA2)--(0,0);
\draw (ddA1)--(-0.2,0);
\draw (ddAA4)--(uuB1);
\draw (ddAA1)--(0.1,0);
\draw (ddAA2)--(0.2,0);
\draw (ddAA3)--(0.3,0);
\draw (uuB2)--(0.8,2.6);
\draw[fill=black] (0.55,2.6) circle [radius=0.03];
\draw[fill=black] (0.8,2.6) circle [radius=0.03];
\draw[fill=white] (0,0) circle [radius=0.03];
\draw[fill=white] (-0.2,0) circle [radius=0.03];
\draw[fill=white] (0.1,0) circle [radius=0.03];
\draw[fill=white] (0.2,0) circle [radius=0.03];
\draw[fill=white] (0.3,0) circle [radius=0.03];
\end{tikzpicture}
\ \ \ \ \ \ \ \ \ \ \ \ 
\begin{tikzpicture}
\pro 02{0.5}{0.5}34A{A}
\pro {0.18}1{0.5}{0.5}34A{AA}
\pro {0.5}0{0.5}{0.5}21B{B}

\draw[color=white,thick] (ddA3)--(dA3);
\draw[color=white,thick] (ddA2)--(dA2);
\draw[color=white,thick] (ddA1)--(dA1);
\draw[color=white,thick] (ddA4)--(dA4);
\draw[color=white,thick] (uuAA1)--(uAA1);
\draw[color=white,thick] (uuAA2)--(uAA2);
\draw[color=white,thick] (uuAA3)--(uAA3);
\draw[color=white,thick] (ddAA3)--(dAA3);
\draw[color=white,thick] (ddAA2)--(dAA2);
\draw[color=white,thick] (ddAA4)--(dAA4);
\draw[color=white,thick] (ddAA1)--(dAA1);
\draw[color=white,thick] (uuB1)--(uB1);
\draw[color=white,thick] (uuB2)--(uB2);
\draw[color=white,thick] (ddB1)--(dB1);
\draw[densely dotted](ddA1)--(-0.1,1.6);
\draw[densely dotted](ddA2)--(0.1,1.6);
\draw[densely dotted](ddA3)--(uuAA1);
\draw[densely dotted](ddA4)--(uuAA2);
\draw[densely dotted](0.7,2)--(uuAA3);
\draw[densely dotted](0.9,2)--(0.9,1.6);
\draw[densely dotted](-0.1,1)--(-0.1,0.6);
\draw[densely dotted](0.1,1)--(0,0.6);
\draw[densely dotted](ddAA1)--(0.1,0.6);
\draw[densely dotted](ddAA2)--(0.2,0.6);
\draw[densely dotted](ddAA3)--(0.3,0.6);
\draw[densely dotted](ddAA4)--(uuB1);
\draw[densely dotted](0.9,1)--(uuB2);

\chip 02{0.5}{0.5}34A
\chip {0.18}1{0.5}{0.5}34A
\chip {0.5}0{0.5}{0.5}21B
\draw (0.7,2.6)--(0.7,2);
\draw (0.9,2.6)--(0.9,2);
\draw (0.9,1.6)--(0.9,1);
\draw[fill=white] (0.7,2) circle [radius=0.03];
\draw[fill=white] (0.9,2) circle [radius=0.03];
\draw[fill=white] (0.9,1) circle [radius=0.03];

\draw[fill=black] (0.7,2.6) circle [radius=0.03];
\draw[fill=black] (0.9,2.6) circle [radius=0.03];
\draw[fill=black] (0.9,1.6) circle [radius=0.03];
\draw (0.1,1.6)--(0.1,1);
\draw (-0.1,1.6)--(-0.1,1);
\draw[fill=white] (0.1,1) circle [radius=0.03];
\draw[fill=white] (-0.1,1) circle [radius=0.03];
\draw[fill=black] (0.1,1.6) circle [radius=0.03];
\draw[fill=black] (-0.1,1.6) circle [radius=0.03];

\draw (-0.1,0.6)--(-0.1,0);
\draw (0,0.6)--(0,0);
\draw (0.1,0.6)--(0.1,0);
\draw (0.2,0.6)--(0.2,0);
\draw (0.3,0.6)--(0.3,0);
\draw[fill=white] (-0.1,0) circle [radius=0.03];
\draw[fill=white] (0,0) circle [radius=0.03];
\draw[fill=white] (0.1,0) circle [radius=0.03];
\draw[fill=white] (0.2,0) circle [radius=0.03];
\draw[fill=white] (0.3,0) circle [radius=0.03];
\draw[fill=black] (-0.1,0.6) circle [radius=0.03];
\draw[fill=black] (0,0.6) circle [radius=0.03];
\draw[fill=black] (0.1,0.6) circle [radius=0.03];
\draw[fill=black] (0.2,0.6) circle [radius=0.03];
\draw[fill=black] (0.3,0.6) circle [radius=0.03];

\end{tikzpicture}
\caption{A circuit and a decomposition into elementary circuits\label{decompcirc}}
\end{center}
\end{figure}
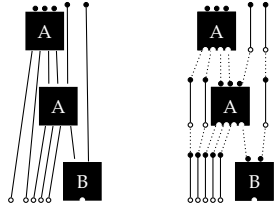
\ \\Straightforwardly from the definition we obtain the following property.
\begin{proposition}\label{propCirc}
The operations $\leftrightarrow$ and $\updownarrow$ endow $\mathcal Circ(\mathcal C)$ with a structure of PRO.
\end{proposition}

This Pro is known to be isomorphic to $\mathcal F(\mathcal C)$ (see e.g. \cite{BG}).

A circuit $\mathfrak p$ is \emph{deconnected} if it is the juxtaposition of two non empty circuits otherwise it is \emph{connected}. The \emph{connected components} of $\mathfrak p$ is the set of the non-empty connected circuits $\mathfrak p_1,\dots, \mathfrak p_k$ such that $$\mathfrak p=\mathfrak p_1\leftrightarrow \mathfrak p_2\leftrightarrow \cdots\leftrightarrow \mathfrak p_k.$$


\begin{proposition}\label{Pconnected}
Let $\mathfrak p$ be a non empty circuit. The following assertions are equivalent:
\begin{enumerate}
\item $\mathfrak p$ is connected.
\item $\mathfrak p$ has only one connected component.
\item $\mathfrak p$ is 
\begin{itemize}
\item either a chip
\item or a wire
\item or $\mathfrak p$ is a non elementary circuit satisfying
\[
\mathfrak p=\begin{array}{c}(\mathfrak p_1\leftrightarrow\cdots\leftrightarrow\mathfrak p_k)\\\updownarrow\\
(\mathfrak q_1\leftrightarrow\cdots\leftrightarrow\mathfrak q_\ell)\end{array}
\]
where $\mathfrak p_i\in \mathcal Circ(\mathcal C)_{m_i,n_i}$,
 $\mathfrak q_i\in \mathcal Circ(\mathcal C)_{p_i,q_i}$ are connected circuits such that   $n_1+\cdots+n_i= p_1+\cdots+p_j$ implies $i=k$ and $j=\ell$.
\end{itemize}
\end{enumerate}
\end{proposition}
\begin{proof}
Chips and wires are connected circuits. Suppose that $\mathfrak p$ is a non elementary circuit which does not satisfy the property of 3: there exists 
$i\not =k$ and $j\not =\ell$ such that $n_1+\cdots+n_i= p_1+\cdots+p_j$. 
Then  
\[
\mathfrak p=\begin{array}{c}(\mathfrak p_1\leftrightarrow\cdots\leftrightarrow\mathfrak p_i)\\\updownarrow\\
(\mathfrak q_1\leftrightarrow\cdots\leftrightarrow\mathfrak q_j)\end{array}
\leftrightarrow
\begin{array}{c}(\mathfrak p_{i+1}\leftrightarrow\cdots\leftrightarrow\mathfrak p_k)\\\updownarrow\\
(\mathfrak q_{j+1}\leftrightarrow\cdots\leftrightarrow\mathfrak q_\ell)
\end{array}
\]
\end{proof}
\subsection{Other free PROs\label{Others}}
The aim of this section is to discuss about how to represent the elements of a free PRO in the general case.
Let $\mathcal X=\bigcup_{m,n\geq 0}\mathcal X_{m,n}$ be a bigraded set. 
In the previous section, we treat the case where $\mathcal X_{n,0}=\mathcal X_{0,m}=\emptyset$ which 
seems to be one of the simplest cases.
In general, we have more relations in $\mathcal F(\mathcal X)$ which are consequences of the interchange law.
Indeed, we have the following proposition.
\begin{proposition}\label{commuteP0} Let $\mathcal P=\bigcup_{m,n}\mathcal P_{m,n}$ be a PRO.
Let $\mathfrak p\in\mathcal P_{0,n}$ and $\mathfrak q\in\mathcal P_{m,0}$. We have
\begin{equation}
\mathfrak p\leftrightarrow\mathfrak q=\begin{array}{c}\mathfrak q\\\updownarrow\\\mathfrak p\end{array}=\mathfrak q\leftrightarrow \mathfrak p.
\end{equation}
Moreover if $m=n=0$ then
 we have 
 $\begin{array}{c}\mathfrak p\\\updownarrow\\\mathfrak q\end{array}=\begin{array}{c}\mathfrak
  q\\\updownarrow\\\mathfrak p\end{array}$. 
\end{proposition}  
\begin{proof}
This comes from
\begin{equation}
\begin{array}{rcl}
\mathfrak p\leftrightarrow\mathfrak q&=&
\left(\begin{array}{c}\oslash\\\updownarrow\\\mathfrak p\end{array}\right)\leftrightarrow
\left(\begin{array}{c}\mathfrak q\\\updownarrow\\\oslash\end{array}\right)
=\begin{array}{c}(\oslash\leftrightarrow\mathfrak q)\\\updownarrow\\ (\mathfrak p\leftrightarrow\oslash)\end{array}=\begin{array}{c}(\mathfrak q\leftrightarrow\oslash)\\\updownarrow\\ (\oslash\leftrightarrow\mathfrak p)\end{array}\\
&=&
\left(\begin{array}{c}\mathfrak q\\\updownarrow\\\oslash\end{array}\right)\leftrightarrow \left(\begin{array}{c}\oslash\\\updownarrow\\\mathfrak p\end{array}\right)
=\mathfrak q\leftrightarrow\mathfrak p
\end{array}
\end{equation}
\end{proof}
The special case $m=n=0$ is known as the Eckmann-Hilton argument\footnote{The  Eckmann-Hilton argument states that 
if a set $X$ is endowed with two unital binary operations $\cdot$ and $\star$ 
satisfying $(a\star b)\cdot(c\star d)=(a\cdot c)\star (b\cdot d)$, then $\star$ 
and $\cdot$ are the same commutative and associative operation.} \cite{EH}.
Proposition \ref{commuteP0} implies that it is rather difficult to describe in a combinatorial way the element of a generic free Pro.
 For instance, the drawings 
\begin{tikzpicture}
\chip 00{0.5}{0.5}02A
\chip {0.7}0{0.3}{0.5}00B
\end{tikzpicture}
 and
\begin{tikzpicture}
\chip {0.5}0{0.5}{0.5}02A
\chip {0}0{0.3}{0.5}00B
\end{tikzpicture}
must denote the same element which differs from the element whose drawing is
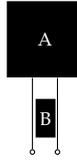
\begin{center}
\begin{tikzpicture}
\pro 011102A{A}
\draw (ddA1)--(1/3,0);
 \draw[fill=white] (1/3,0) circle [radius=0.03];
\draw (ddA2)--(2/3,0);
 \draw[fill=white] (2/3,0) circle [radius=0.03];
\pro {1/3+0.05}{0.2}{1/4}{0.5}00B{B}
\end{tikzpicture}
\end{center}

This suggests that one has to include in the description of the circuits a topological notion of zone. For instance, we can associate  the element 
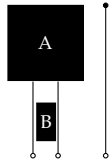
\begin{center}
\begin{center}
\begin{tikzpicture}
\pro 011102A{A}
\draw (ddA1)--(1/3,0);
 \draw[fill=white] (1/3,0) circle [radius=0.03];
\draw (ddA2)--(2/3,0);
 \draw[fill=white] (2/3,0) circle [radius=0.03];
\pro {1/3+0.05}{0.2}{1/4}{0.5}00B{B}
\draw (1.3,2)-- (1.3,0);
\draw[fill=white] (1.3,0) circle [radius=0.03];
\draw[fill=black] (1.3,2) circle [radius=0.03];
\end{tikzpicture}
\end{center}

\end{center}
with the three zones below
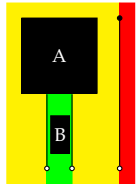
\begin{center}
\begin{tikzpicture}
\draw[fill=blue,color=yellow] (-0.2,2.2) rectangle (1.3,-0.2);
\draw[fill=green,color=red] (1.3,2.2) rectangle (1.5,-0.2);
\draw[fill=green,color=green] (1/3,1) rectangle (2/3,-0.2);
\pro 011102A{A}
\draw (ddA1)--(1/3,0);
 \draw[fill=white] (1/3,0) circle [radius=0.03];
\draw (ddA2)--(2/3,0);
 \draw[fill=white] (2/3,0) circle [radius=0.03];
\pro {1/3+0.05}{0.2}{1/4}{0.5}00B{B}
\draw (1.3,2)-- (1.3,0);
\draw[fill=white] (1.3,0) circle [radius=0.03];
\draw[fill=black] (1.3,2) circle [radius=0.03];
\end{tikzpicture}
\end{center}
But,  the description of the vertical composition in terms of zones seems to be complicated  as suggested by the following example. Consider the  element above 
with three zones together with the element represented by
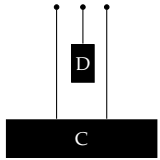
\begin{center}
\begin{tikzpicture}
\pro 002{0.5}20C{C}
\pro {2/3+0.2}{1}{0.3}{0.5}10D{D}
\draw (uuC1)--(2/3,2);
\draw (uuC2)--(4/3,2);
\draw (uuD1)--(2/3+0.2+0.15,2);
\draw[fill=black] (2/3,2) circle [radius=0.03];
\draw[fill=black] (4/3,2) circle [radius=0.03];
\draw[fill=black] (2/3+0.2+0.15,2) circle [radius=0.03];

\end{tikzpicture}
\end{center}
 which has two zones. Their vertical composition has only one zone and admits the following  graphical representations
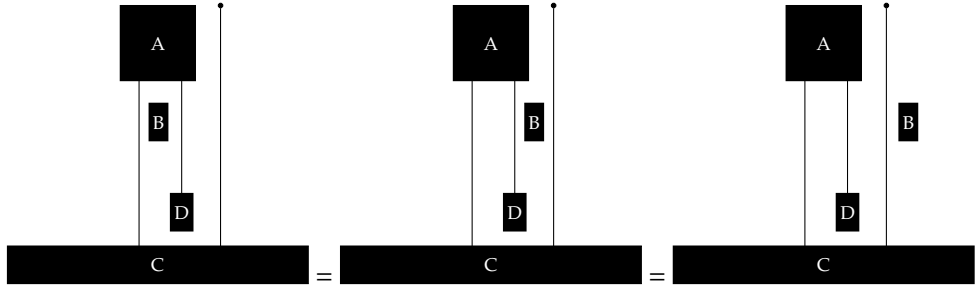
\begin{center}
\begin{tikzpicture}
\pro 021100A{A}
\pro {1/3+0.05}{1.2}{1/4}{0.5}00B{B}
\pro {-1.5}{-0.7}4{0.5}00C{C}
\pro {2/3+0}{0}{0.3}{0.5}00D{D}
\draw (2/3+0.15,0.5)--(2/3+0.15,2.2);
\draw (0.25,-0.2)--(0.25,2.2);
\draw (4/3,-0.2)--(4/3,3);
\draw[fill=black] (4/3,3) circle [radius=0.03];

\end{tikzpicture}
$=$
\begin{tikzpicture}
\pro 021100A{A}
\pro {0.95}{1.2}{1/4}{0.5}00B{B}
\pro {-1.5}{-0.7}4{0.5}00C{C}
\pro {2/3+0}{0}{0.3}{0.5}00D{D}
\draw (2/3+0.15,0.5)--(2/3+0.15,2.2);
\draw (0.25,-0.2)--(0.25,2.2);
\draw (4/3,-0.2)--(4/3,3);
\draw[fill=black] (4/3,3) circle [radius=0.03];

\end{tikzpicture}
$=$
\begin{tikzpicture}
\pro 021100A{A}
\pro {1.5}{1.2}{1/4}{0.5}00B{B}
\pro {-1.5}{-0.7}4{0.5}00C{C}
\pro {2/3+0}{0}{0.3}{0.5}00D{D}
\draw (2/3+0.15,0.5)--(2/3+0.15,2.2);
\draw (0.25,-0.2)--(0.25,2.2);
\draw (4/3,-0.2)--(4/3,3);
\draw[fill=black] (4/3,3) circle [radius=0.03];

\end{tikzpicture}

\end{center}

Notice that certain cases are simpler, as suggested by the following remark.

\subsection{Path colored PROs\label{Path}}
Let $\mathcal C=\bigcup_{n,m\geq 1}\mathcal C_{n,m}$ be a graded set of chips and $N>0$. We construct the set $\mathcal Steps(\mathcal C,N)$ of the triplets $\left(\begin{array}{c}I\\\mathfrak c\\J\end{array}\right)$ with $\mathfrak c\in \mathcal C$ and $I\in[N]^n$ and $J\in[N]^m$ when $\mathfrak c\in \mathcal C_{n,m}$ (equivalently we label the inputs and the outputs of each chips by numbers in $[N]$). The set $\mathcal Paths(\mathcal C,N)$ of the \emph{paths} is obtained recursively as follows: a path is
\begin{itemize}
\item either the empty path denoted by $\oslash$ with $0$ input and $0$ output,
\item or a labelled wire $\displaystyle\mathop\multimapdotbothBvert^i_i$ with $1$ input and $1$ output,
\item or a labelled chips $\mathfrak c\in \mathcal Steps(\mathcal C,N)$,
\item or the juxtaposition $\mathfrak p\mathfrak p'$ of two paths with the rule $\oslash\mathfrak p=\mathfrak p\oslash=\mathfrak p$. The inputs (resp. the outputs) of $\mathfrak p\mathfrak p'$ are obtained by catening the inputs (resp. the outputs) of $\mathfrak p$ and those of $\mathfrak p'$. This operation will be denoted by $\leftrightarrow$.
\begin{figure}[h]
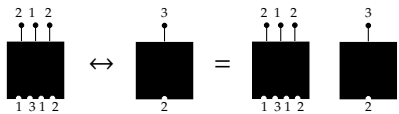

\begin{center}
$
\tikdanseq{
\chip 00{0.75}{0.75}34{}{a}
\node(n1) at (0.4,-0.1) {\tiny 1\ \ 3\ 1\ \ 2};
\node(n1) at (0.35,1.15) {\tiny 2\ \ 1\ \ \ 2};
}
\leftrightarrow
\tikdanseq{
\chip 00{0.75}{0.75}11{}{b}
\node(n1) at (0.75/2,-0.1){\tiny 2};
\node(n1) at (0.75/2,1.15){\tiny 3};
}
=
\tikdanseq{
\chip 00{0.75}{0.75}34{}{a}
\node(n1) at (0.4,-0.1) {\tiny 1\ \ 3\ 1\ \ 2};
\node(n1) at (0.35,1.15) {\tiny 2\ \ 1\ \ \ 2};
}
\tikdanseq{
\chip 00{0.75}{0.75}11{}{b}
\node(n1) at (0.75/2,-0.1){\tiny 2};
\node(n1) at (0.75/2,1.15){\tiny 3};
}
$
\end{center}
\caption{An example of horizontal composition of two paths.\label{fhcomppath}}
\end{figure}
\item or the connexion of a $(I,J)$-path $\mathfrak p$ to a $(J,K)$-path $\mathfrak q$ obtained by connecting each input labelled by $i$ in $\mathfrak p$ to the output labelled by $i$ in $\mathfrak q$ by a wire labelled by $i$ with the rule: connecting a $(I,J)$-path to a list of labelled wires let the path unchanged. The outputs of the connections are the output of $\mathfrak p$ and the inputs are those of $\mathfrak q$. We  denote by $\updownarrow$ this operation.
\begin{figure}[h]
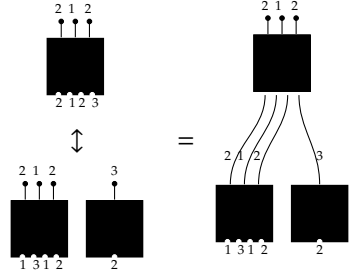

\begin{center}
\[\begin{array}{c}
\tikdanseq{
\chip 00{0.75}{0.75}34{}{a}
\node(n1) at (0.4,-0.1) {\tiny 2\ \ 1\ 2\ \ 3};
\node(n1) at (0.35,1.15) {\tiny 2\ \ 1\ \ \ 2};
}\\
\updownarrow\\
\tikdanseq{
\chip 00{0.75}{0.75}34{}{a}
\node(n1) at (0.4,-0.1) {\tiny 1\ \ 3\ 1\ \ 2};
\node(n1) at (0.35,1.15) {\tiny 2\ \ 1\ \ \ 2};
\chip 10{0.75}{0.75}11{}{b}
\node(n1) at (1+0.75/2,-0.1){\tiny 2};
\node(n1) at (1+0.75/2,1.15){\tiny 3};
}
\end{array}
=
\tikdanseq{
\pro 00{0.75}{0.75}34{}{a}
\node(n1) at (0.4,-0.1) {\tiny 1\ \ 3\ 1\ \ 2};
\node(n1) at (0.35,1.15) {\tiny 2\ \ 1\ \ \ 2};
\pro 10{0.75}{0.75}11{}{b}
\node(n1) at (1+0.75/2,-0.1){\tiny 2};
\node(n1) at (1+0.75/2,1.15){\tiny 3};
\plug 00{0.75}{0.75}4{}
\plug 10{0.75}{0.75}1{}
\chip {0.5}2{0.75}{0.75}34{}{c}
\pro {0.5}2{0.75}{0.75}34{}{c}
\node(n1) at (0.5+0.35,3.15) {\tiny 2\ \ 1\ \ \ 2};
\draw (uua1) to [out=90,in=270] (ddc1);
\draw (uua2) to [out=90,in=270] (ddc2);
\draw (uua3) to [out=90,in=270] (ddc3);
\draw (uub1) to [out=90,in=270] (ddc4);
}
\]
\end{center}
\caption{An example of vertical composition of two paths.\label{fvcomppath}}
\end{figure}

\end{itemize}
The set $\mathcal Paths(\mathcal C,N)$ is graded by the colors of the inputs and the outputs of the paths. For simplicity we  denote by $\mathtt{In}(\mathfrak q)$ (resp. $\mathtt{Out}(\mathfrak q)$) the vector of the colors of the input (resp. output) of $\mathfrak q$.
Straightforwardly from the definition:
\begin{proposition}
$\mathcal Paths(\mathcal C,N)$ is a colored PRO.
\end{proposition}

We define $u:\mathcal Paths(\mathcal C,N)\longrightarrow\mathcal F(\mathcal C)$ the unlabeling map that associate to each paths $\mathfrak p$ the circuit $u(\mathfrak p)$ obtained by removing the labels from the inputs, the outputs and the connections of $\mathfrak p$ (see \ref{fp2c} for an example).
\begin{figure}[h]
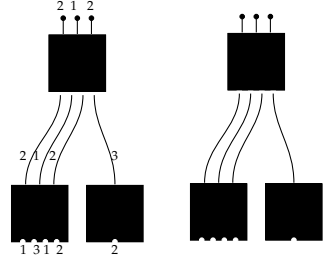

\begin{center}
$\tikdanseq{
\pro 00{0.75}{0.75}34{}{a}
\node(n1) at (0.4,-0.1) {\tiny 1\ \ 3\ 1\ \ 2};
\node(n1) at (0.35,1.15) {\tiny 2\ \ 1\ \ \ 2};
\pro 10{0.75}{0.75}11{}{b}
\node(n1) at (1+0.75/2,-0.1){\tiny 2};
\node(n1) at (1+0.75/2,1.15){\tiny 3};
\plug 00{0.75}{0.75}4{}
\plug 10{0.75}{0.75}1{}
\chip {0.5}2{0.75}{0.75}34{}{c}
\pro {0.5}2{0.75}{0.75}34{}{c}
\node(n1) at (0.5+0.35,3.15) {\tiny 2\ \ 1\ \ \ 2};
\draw (uua1) to [out=90,in=270] (ddc1);
\draw (uua2) to [out=90,in=270] (ddc2);
\draw (uua3) to [out=90,in=270] (ddc3);
\draw (uub1) to [out=90,in=270] (ddc4);
}
$\ \ \ 
$\tikdanseq{
\pro 00{0.75}{0.75}34{}{a}
\pro 10{0.75}{0.75}11{}{b}
\plug 00{0.75}{0.75}4{}
\plug 10{0.75}{0.75}1{}
\chip {0.5}2{0.75}{0.75}34{}{c}
\pro {0.5}2{0.75}{0.75}34{}{c}
\draw (uua1) to [out=90,in=270] (ddc1);
\draw (uua2) to [out=90,in=270] (ddc2);
\draw (uua3) to [out=90,in=270] (ddc3);
\draw (uub1) to [out=90,in=270] (ddc4);
}
$
\end{center}
\caption{A path and its associated circuit.\label{fp2c}}
\end{figure}

The following result are easy to obtain:
\begin{lemma}\label{Path2Circ}
\begin{enumerate}
\item 
The sets
\[
\{\mathfrak q_1\leftrightarrow \mathfrak q_2:u(\mathfrak q_1)=\mathfrak p_1, u(\mathfrak q_2)=\mathfrak p_2, \mathtt{In}(\mathfrak q_1)=I_1,
\mathtt{In}(\mathfrak q_2)=I_2,  \mathtt{Out}(\mathfrak q_1)=J_1, \mbox{ and } 
\mathtt{Out}(\mathfrak q_2)=J_2 \} 
\]
and
\[
\{\mathfrak q:u(\mathfrak q)=\mathfrak p_1\leftrightarrow\mathfrak p_2, \mathtt{Out}(\mathfrak q)=I_1I_2,\mbox{ and }
\mathtt{In}(\mathfrak q)=J_1J_2\}
\]
are equal.
\item The sets
\[
\left\{\begin{array}{c}\mathfrak q'\\\updownarrow\\ \mathfrak q''\end{array}:u(\mathfrak q')=\mathfrak p', u(\mathfrak q'')=\mathfrak p'', \mathtt{In}(\mathfrak q'')=J,
\mathtt{In}(\mathfrak q')=\mathtt{Out}(\mathfrak q''),  \mathtt{Out}(\mathfrak q')=I\right\} 
\]
and
\[
\left\{
\mathfrak q:u(\mathfrak q)=\begin{array}{c}\mathfrak p'\\\updownarrow\\ \mathfrak p''\end{array}, \mathtt{Out}(\mathfrak q)=I, \mbox{ and }\mathtt{In}(\mathfrak q)=J
\right\}
\]
are equal.
\end{enumerate}
\end{lemma}

\section{A PRO structure on hypermatrices }
\subsection{Hypermatrices\label{PROHyp}}

An hypermatrix is a family of elements indexed by a list of integers 
$i_1,\dots,i_k$ where $i_j\in[N_j]:=\{1,\dots,N_j\}$. The hypermatrices here 
are such that all the $N_j$ are the same. Furthermore we split the list of the indices into two parts in the aim to coincide with the structure of PRO.\\
 \begin{definition}
 Let $\mathbb K$ be a commutative semiring and $N,p,q\in\mathbb N$, $N>0$. We define 
 $$\mathbb K(N,m, n):=\left\{\left(\hmat a\TI\TJ\right)_{\TI\in[N]^m, \TJ\in [N]^n}:\hmat a\TI\TJ\in\mathbb K\right\}$$ 
 with the special case $\mathbb K(N,0,0):=\mathbb K$. 
 We endow $\mathbb K(N):=\bigcup_{m,n}\mathcal M(N,m,n)$ with  the three operators:
\begin{itemize}
\item We define $\leftrightarrow:\mathbb K(N,m,n)\times \mathbb K(N,m',n')\rightarrow \mathbb K(N,m+m',n+n')$ as follows. 
Consider  $\mathbf A=\left(\hmat a\TI\TJ\right)_{\TI\in[N]^m\atop \TJ\in[N]^n}$ and $\mathbf B=\left(\hmat b\TI\TJ\right)_{\TI\in[N]^{m'}
\atop \TJ\in[N]^{'}}$. We set $\mathbf A\leftrightarrow \mathbf B=\mathbf C$ with 
$\mathbf C=\left(\hmat c\TI\TJ\right)_{\TI\in[N]^{m+m'}\atop \TJ\in[N]^{n+n'}}$ and
$\hmat c\TI\TJ=\hmat a{\TI[1\dots m]}{\TJ[1\dots n]}\hmat b{\TI[m+1\dots m+m']}{\TJ[n+1\dots n+n']}$
where $\TI[a\dots b]$ denotes the sequence $[i_a,\dots,i_b]$ for $\TI=[i_1,\dots,i_s]$ and $1\leq a\leq b\leq s$.
\item We define $\updownarrow:\mathbb K(N,m,n)\times \mathbb K(N,n,p)\rightarrow \mathbb K(N,m,p)$ as follows.
 Consider  $\mathbf A=\left(\hmat a\TI\TJ\right)_{\TI\in[N]^m\atop \TJ\in[N]^n}$ and 
 $\mathbf B=\left(\hmat b\TI\TJ\right)_{\TI\in[N]^{n}\atop \TJ\in[N]^{p}}$.
 We set $\begin{array}{c}\mathbf A\\\updownarrow\\\mathbf B\end{array}=
 \mathbf C$ with $\mathbf C=\left(\hmat c\TI\TJ\right)_{\TI\in[N]^m\atop \TJ\in[N]^p}$ and  
$\hmat c\TI\TJ=\sum_{\TK\in [N]^n}\hmat a\TI\TK\hmat b\TK{\TJ}.$
\item Also each $\mathbb K(N,m,n)$ is naturally endowed with a structure of $\mathbb K$-module.
\end{itemize}
\end{definition}
For the sake of simplicity, when computing, we  write an hypermatrix of $\mathbb K(N,m,n)$ as a rectangular $N^m\times N^n$-matrix.
 With this notation, the operation $\updownarrow$ is assimilated with the classical product of matrices and $\leftrightarrow$ with the Kronecker product.

\begin{example}\rm
Let $\mathbf A=\left(\hmat a\TI\TJ\right)_{\TI\in[2]^2,\TJ\in[2]^2}$ and $\mathbf B=\left(\hmat b\TI\TJ\right)_{\TI\in[2]^2,\TJ\in[2]^2}$.
 These matrices have the following $2$-dimensional representation
\[
\mathbf A=\left[
\begin{array}{c|c|c|c}
\hmat a{11}{11}&\hmat a{11}{12}&\hmat a{11}{21}&\hmat a{11}{22}\\\hline
\hmat a{12}{11}&\hmat a{12}{12}&\hmat a{12}{21}&\hmat a{12}{22}\\\hline
\hmat a{21}{11}&\hmat a{21}{12}&\hmat a{21}{21}&\hmat a{21}{22}\\\hline
\hmat a{22}{11}&\hmat a{22}{12}&\hmat a{22}{21}&\hmat a{22}{22}
\end{array}
\right]
\mbox{ and }
\mathbf B=\left[
\begin{array}{c|c|c|c}
\hmat b{11}{11}&\hmat b{11}{12}&\hmat b{11}{21}&\hmat b{11}{22}\\\hline
\hmat b{12}{11}&\hmat b{12}{12}&\hmat b{12}{21}&\hmat b{12}{22}\\\hline
\hmat b{21}{11}&\hmat b{21}{12}&\hmat b{21}{21}&\hmat b{21}{22}\\\hline
\hmat b{22}{11}&\hmat b{22}{12}&\hmat b{22}{21}&\hmat b{22}{22}
\end{array}
\right].
\]
We have
\[
\begin{array}{c}\mathbf A\\\updownarrow\\\mathbf  B\end{array}=\left[
\begin{array}{c|c|c}
\hmat a{11}{11}\hmat b{11}{11}+
\hmat a{11}{12}\hmat b{12}{11}+\hmat a{11}{21}\hmat b{21}{11}+\hmat a{11}{22}\hmat b{22}{11}
&
\cdots
&
\hmat a{11}{11}\hmat b{11}{22}+
\hmat a{11}{12}\hmat b{12}{22}+\hmat a{11}{21}\hmat b{21}{22}+\hmat a{11}{22}\hmat b{22}{22}
\\\hline
\vdots&\ddots
&
\vdots
\\\hline
\hmat a{22}{11}\hmat b{11}{11}+ \hmat a{22}{12}\hmat b{12}{11}+\hmat a{22}{21}\hmat b{21}{11}+
\hmat a{22}{22}\hmat b{22}{11}
&
\cdots&
\hmat a{22}{11}\hmat b{11}{22}+
\hmat a{22}{12}\hmat b{12}{22}+
\hmat a{22}{21}\hmat b{21}{22}+\hmat a{22}{22}\hmat b{22}{22}\end{array}
\right]
\]
and
\[
\mathbf A\leftrightarrow\mathbf  B=\left[
\begin{array}{c|c|c|ccc|c|c|c}
\hmat a{11}{11}  \hmat b{11}{11}&
\hmat a{11}{11}  \hmat b{11}{12}&
\hmat a{11}{11}  \hmat b{11}{21}&
\hmat a{11}{11}  \hmat b{11}{22}&\cdots&
\hmat a{11}{22}  \hmat b{11}{11}&
\hmat a{11}{22}  \hmat b{11}{12}&
\hmat a{11}{22}  \hmat b{11}{21}&
\hmat a{11}{22}  \hmat b{11}{22}\\\hline
\hmat a{11}{11}  \hmat b{12}{11}&
\hmat a{11}{11}  \hmat b{12}{12}&
\hmat a{11}{11}  \hmat b{12}{21}&
\hmat a{11}{11}  \hmat b{12}{22}&\cdots&
\hmat a{11}{22}  \hmat b{12}{11}&
\hmat a{11}{22}  \hmat b{12}{12}&
\hmat a{11}{22}  \hmat b{12}{21}&
\hmat a{11}{22}  \hmat b{12}{22}\\\hline
\hmat a{11}{11}  \hmat b{21}{11}&
\hmat a{11}{11}  \hmat b{21}{12}&
\hmat a{11}{11}  \hmat b{21}{21}&
\hmat a{11}{11}  \hmat b{21}{22}&\cdots&
\hmat a{11}{22}  \hmat b{21}{11}&
\hmat a{11}{22}  \hmat b{21}{12}&
\hmat a{11}{22}  \hmat b{21}{21}&
\hmat a{11}{22}  \hmat b{21}{22}\\\hline
\hmat a{11}{11}  \hmat b{22}{11}&
\hmat a{11}{11}  \hmat b{22}{12}&
\hmat a{11}{11}  \hmat b{22}{21}&
\hmat a{11}{11}  \hmat b{22}{22}&\cdots&
\hmat a{11}{22}  \hmat b{22}{11}&
\hmat a{11}{22}  \hmat b{22}{12}&
\hmat a{11}{22}  \hmat b{22}{21}&
\hmat a{11}{22}  \hmat b{22}{22}\\
\multicolumn{4}{c}{\vdots}&\ddots&\multicolumn{4}{c}{\vdots}\\

\hmat a{22}{11}  \hmat b{11}{11}&
\hmat a{22}{11}  \hmat b{12}{11}&
\hmat a{22}{11}  \hmat b{21}{11}&
\hmat a{22}{11}  \hmat b{22}{11}&\cdots&
\hmat a{22}{22}  \hmat b{11}{11}&
\hmat a{22}{22}  \hmat b{12}{11}&
\hmat a{22}{22}  \hmat b{21}{11}&
\hmat a{22}{22}  \hmat b{22}{11}\\\hline
\hmat a{22}{11}  \hmat b{11}{12}&
\hmat a{22}{11}  \hmat b{12}{12}&
\hmat a{22}{11}  \hmat b{21}{12}&
\hmat a{22}{11}  \hmat b{22}{12}&\cdots&
\hmat a{22}{22}  \hmat b{11}{12}&
\hmat a{22}{22}  \hmat b{12}{12}&
\hmat a{22}{22}  \hmat b{21}{12}&
\hmat a{22}{22}  \hmat b{22}{12}\\\hline
\hmat a{22}{11}  \hmat b{11}{21}&
\hmat a{22}{11}  \hmat b{12}{21}&
\hmat a{22}{11}  \hmat b{21}{21}&
\hmat a{22}{11}  \hmat b{22}{21}&\cdots&
\hmat a{22}{22}  \hmat b{11}{21}&
\hmat a{22}{22}  \hmat b{12}{21}&
\hmat a{22}{22}  \hmat b{21}{21}&
\hmat a{22}{22}  \hmat b{22}{21}\\\hline
\hmat a{22}{11}  \hmat b{11}{22}&
\hmat a{22}{11}  \hmat b{12}{22}&
\hmat a{22}{11}  \hmat b{21}{22}&
\hmat a{22}{11}  \hmat b{22}{22}&\cdots&
\hmat a{22}{22}  \hmat b{11}{22}&
\hmat a{22}{22}  \hmat b{12}{22}&
\hmat a{22}{22}  \hmat b{21}{22}&
\hmat a{22}{22}  \hmat b{22}{22}
\end{array}
\right]
\]
\end{example}

\begin{proposition}\label{isModPro}
$\mathbb K(N)$ is a $\mathbb K$-ModPro.
\end{proposition}
\begin{proof}
It is easy to check that $\mathbb K(N)$ is a PRO where the graded $(1,1)$-vertical unity, 
$\mathbf I_{1}(N)$, is the $N\times N$ identity matrix and the horizontal unity is 
$\mathbf I_{0}(N)=1$. The left, right, up and down distributivity follows from the bilinearity of the product and Kronecker product of matrices.
\end{proof}
We introduce for $K\in [N]^p, L\in [N]^q$ the hypermatrices
 $E(N,p,q;K,L)=\left(\delta_{IK}\delta_{JL}\right)_{I\in [N]^p,J\in[N]^q}$
  where $\delta_{IK}=0$ if $I\neq K$ and $1$ otherwise is the Kronecker delta.. 
  As in the case of matrices, these elements play the role of basic elements allowing to decompose the hypermatrices.
\begin{claim}\label{M2E}
For each $\mathbf A=\left(\hmat aIJ\right)$, we have
\[
\mathbf A=\mathop\bigboxplus_{I,J}\hmat aIJ.E(N,p,q;I,J).
\]
Furthermore this decomposition is unique.
\end{claim}
Furthermore, we have for any $I\in[N]^{p}$, $J\in[N]^{q}$,
$I'\in[N]^{p'}$, $J'\in[N]^{q'}$
\begin{equation}\label{EHE}
	\mathbf E(N,p,q;I,J)\leftrightarrow \mathbf E(N,p',q';I',J')=
	\mathbf E(N,p+p',q+q';II',
	JJ'),
\end{equation}
where $[i_{1},\dots,i_{k}][j_{1},\dots,j_{\ell}]=[i_{1},\dots,i_{k},j_{1},\dots,j_{\ell}]$ denotes the catenation of the sequences,
and for any $I\in[N]^{p}$, $J,J'\in[N]^{q}$, and $K\in [N^{r}]$,
\begin{equation}\label{EVE}
	\begin{array}{c}\mathbf E(N,p,q;I,J)\\\updownarrow\\\mathbf E(N,p,r;J',K)\end{array}=\delta_{J,J'}\mathbf E(N,p,r;I,K),
\end{equation}
where $\delta_{J,J'}=1$ if $J=J'$ and $0$ otherwise.
\subsection{Kronecker product\label{Kronecker}}
If $I=[i_1,\dots,i_p]\in (\N-\{0\})^p$ we
 will write $I\% N=[r_1,\dots,r_p]$ and $I/N=[q_1,\dots,q_p]$ where $i_k=(q_k-1)N+r_k$ and $1\leq r_{k-1}\leq N$.\\
The \emph{Kronecker product} $\mathbf A\odot\mathbf B$ of two hypermatrices
 $\mathbf A=\left(\hmat aIJ\right)\in\mathbb K(M,p,q),{}
  \mathbf B=\left(\hmat bIJ\right)\in\mathbb K(N,p,q)$ is the hypermatrix $\mathbf C=\left(\hmat cIJ\right)\in\mathbb K(MN,p,q)$ 
  defined by
\begin{equation}\label{defdot}
\hmat c IJ=\hmat a{I\% M}{J\%M}\hmat b{I/ M}{J/M}.
\end{equation}
For instance we have the following identity
\begin{equation}\label{dot2E}
\begin{array}{l}
\mathbf E(M,p,q;[r_1,\dots,r_p],[s_1,\dots,s_q])\odot
\mathbf E(N,p,q;[k_1,\dots,k_p],[\ell_1,\dots,\ell_q])=\\
\mathbf E(MN,p,q;[(k_1-1)M+r_1,\dots,(k_p-1)M+r_p],
[(\ell_1-1)M+s_1,\dots,(\ell_q-1)M+s_q)].\end{array}
\end{equation}
Or equivalently,
\begin{equation}\label{E2dot}
\mathbf E(MN,p,q;I,J)=\mathbf E(M,p,q;I\%M,J\%M)\odot \mathbf E(N,p,q;I/N,J/N)
\end{equation} 
with the same notation as in Equation (\ref{defdot}).\\
Let us examine the first properties of this construction.
\begin{lemma}\label{distMat}
We have 
$(\mathbf A+ \mathbf B)\odot \mathbf C=(\mathbf A\odot \mathbf C)+(\mathbf B\odot \mathbf C)$ and  
$\mathbf A\odot (\mathbf B+ \mathbf C)=(\mathbf A\odot \mathbf B)+(\mathbf A\odot \mathbf C)$.
\end{lemma}
\begin{proof}
We have
$
(\mathbf A+ \mathbf B)\odot \mathbf C=\left(\hmat gIJ\right)
$
where $\hmat gIJ=\left(\hmat a{I\%M}{J\%M}+ 
\hmat b{I\%M}{J\%M}\right)\hmat c{I/M}{I/M}=
\hmat a{I\%M}{J\%M}\hmat c{I/M}{I/M}+
\hmat b{I\%M}{J\%N}\hmat c{I/M}{I/M}.
$ 
This implies $(\mathbf A+ \mathbf B)\odot \mathbf C=(\mathbf A\odot \mathbf C)+(\mathbf B\odot \mathbf C)$. 
The second identities is proved in a similar way.
\end{proof}
\begin{lemma}\label{rMat}
We have 
$\mathfrak r.(\mathbf A\odot \mathbf B)=(\mathfrak r.\mathbf A)\odot \mathbf B=\mathbf A\odot (\mathfrak r.\mathbf B)$\\
\end{lemma}
\begin{proof}
This is straightforward from the definition.
\end{proof}
The Kronecker product allows to give an analogue to a result which is classical for matrices.
\begin{lemma}\label{lemmaRNN'RPRO}
The $\mathbb K$-module $\mathbb K(MN;p,q)$ and $\mathbb K(M;p,q)\otimes \mathbb K(N;p,q)$ are isomorphic.
	An explicit morphism $\phi$ sends each $\mathbf A\otimes\mathbf B$ to $\mathbf A\odot\mathbf B$. 
\end{lemma}
\begin{proof}
	The fact that $\phi$ is an isomorphism 
	comes that one can construct explicitly its inverse as the unique morphism satisfying 
	$\phi^{-1}(\mathbf E(MN,p,q;I,J))=\phi^{-1}(\mathbf E(M,p,q;I\%M,J\%M)\odot{}
	\mathbf E(M,p,q;I/M,J/M))=\mathbf E(M,p,q;I\%M,J\%M)\otimes
	\mathbf E(M,p,q;I/M,J/M)$.
\end{proof}

The set $\mathbb K(M)\otimes\mathbb K(N):=\bigcup_{p,q}
\mathbb K(M;p,q)\otimes\mathbb K(N;p,q)$  is naturally
endowed with a structure of Pro. It suffices to set $(\mathbf A\otimes\mathbf A')\leftrightarrow (\mathbf B\otimes\mathbf B')
=(\mathbf A\leftrightarrow \mathbf B)\otimes (\mathbf A'\leftrightarrow \mathbf B')$ and
$\begin{array}{c}(\mathbf A\otimes\mathbf A')\\\updownarrow\\ (\mathbf B\otimes\mathbf B')\end{array}
=\begin{array}{c}\mathbf A\\\updownarrow\\ \mathbf B\end{array}\otimes \begin{array}{c}\mathbf A'\\\updownarrow\\ \mathbf B'\end{array}$.
Furthermore, one checks that the action of $\mathbb K$ on each $\mathbb K(M;p,q)\otimes\mathbb K(N;p,q)$ is compatible with the operations
$\leftrightarrow$ and $\updownarrow$ in the sense of the definition of ModPro. So, $\mathbb K(M)\otimes\mathbb K(N)$ is a $\mathbb K$-ModPro.
\begin{theorem}\label{thkronecker}
The $\mathbb K$-ModPro $\mathbb K(M)\otimes\mathbb K(N)$ is isomorphic to $\mathbb K(MN)$.
\end{theorem}
\begin{proof}
From Lemma \ref{lemmaRNN'RPRO} it suffices to prove that 
$\mathbb K(M)\otimes\mathbb K(N)$ and $\mathbb K(NN)$ are isomorphic as PROs.\\
Consider, as in Lemma  \ref{lemmaRNN'RPRO}, the map $\phi$ sending each $(\mathbf A\otimes\mathbf B)$ to $\mathbf A\odot\mathbf B$.
We have to prove that it is a morphism of PRO. 
\begin{itemize}
\item {\it Image of the units}\\ Obviously $\phi\left(\mathbf I(M)^{\leftrightarrow p}\otimes \mathbf I(N)^{\leftrightarrow p}\right)={}
\mathbf I(M)^{\leftrightarrow p}\odot\mathbf I(N)^{\leftrightarrow p}=\mathbf I(MN)^{\leftrightarrow p}$ and 
$\phi\left(\mathbf I_0(M)\otimes \mathbf I_0(N)\right)=0=\mathbf I_0(MN)$.
\item {\it Compatibility with $\leftrightarrow$}\\ Let $\mathbf A=\left(\hmat {a}{I_1\%N}{J_1\%N}\right)\in \mathbb K(M;p_{A},q_{A})$,
 $\mathbf B=\left(\hmat {b}{I_1\%N}{J_1\%N}\right)\in \mathbf K(M;p_{B},q_{B})$,
 $\mathbf A'=\left(\hmat {a'}{I_1\%N}{J_1\%N}\right)\in \mathbb K(N;p_{A},q_A)$ and $\mathbf B'=\left(\hmat {b'}{I_1\%N}{J_1\%N}\right)\in \mathbb K(N;p_B,q_B)$. We have
\[
\phi\left((\mathbf A\otimes \mathbf B)\leftrightarrow (\mathbf A'\otimes \mathbf B')\right)={}
\phi((\mathbf A\leftrightarrow \mathbf B)\otimes (\mathbf A'\leftrightarrow \mathbf B'))
=\left(
\hmat c{I}{J}
\right)_{I\in[NN']^{p_A+p_B}\atop j\in[NN']^{q_A+q_B}}
\] 
where 
$\hmat c{I_1I_2}{J_1J_2}=\hmat {a}{I_1\%N}{J_1\%N}\hmat {b}{I_2\%N}{J_2\%N}\hmat {a'}{I_1/N}{J_1/N}\hmat {b'}{I_2/N}{J_2/N}
=\hmat {a}{I_1\%N}{J_1\%N}\hmat {a'}{I_1/N}{J_1/N}\hmat {b}{I_2\%N}{J_2\%N}\hmat {b'}{I_2/N}{J_2/N}$.
Hence $$ \phi\left((\mathbf A\otimes \mathbf B)\leftrightarrow (\mathbf A'\otimes \mathbf B')\right)={}
\left(\phi(\mathbf A)\odot \phi(\mathbf B')\right)\leftrightarrow \left(\phi(\mathbf B)\odot \phi(\mathbf B')\right)={}
\left(\phi(\mathbf A\otimes \mathbf A')\leftrightarrow 
\phi(\mathbf B\otimes \mathbf B')\right).$$
\item {\it Compatibility with $\updownarrow$}\\ Let $\mathbf A=\left(\hmat {a}{I_1\%N}{J_1\%N}\right)\in \mathbb K(M;p,q)$,
 $\mathbf B=\left(\hmat {b}{I_1\%N}{J_1\%N}\right)\in \mathbf K(M;q,r)$,
 $\mathbf A'=\left(\hmat {a'}{I_1\%N}{J_1\%N}\right)\in \mathbb K(N;p,q)$ and $\mathbf B'=\left(\hmat {b'}{I_1\%N}{J_1\%N}\right)\in \mathbb K(N;q,r)$. We have
\[
\phi\left(\begin{array}{c}\mathbf A\otimes\mathbf A'\\\updownarrow\\\mathbf B\otimes\mathbf B'  \end{array}\right)
=\left(\begin{array}{c}\mathbf A\\\updownarrow\\ \mathbf B\end{array}\right)\odot \left(\begin{array}{c}\mathbf A'\\\updownarrow\\\mathbf B'  \end{array}\right)\\
=\left(\hmat cIJ\right)
\]
where
\[
\hmat cIJ=\displaystyle\mathop\sum_{K\in [M]^q}\mathop\sum_{K'\in[N]^q}\hmat {a}{I\%M}{K_1}
 \hmat {b}{K}{J\%M}\hmat {a'}{I/N}{K'}\hmat {b'}{K'}{J/N}=\sum_{{K\in [MN]^q}}
 \hmat {a}{I\%M}{K\%M}
 \hmat {a'}{I/N}{K/M}\hmat {b}{K\%M}{J\%M}\hmat {b'}{K/M}{J/N}.
\]
Hence we recognize:
\[\phi\left(\begin{array}{c}\mathbf A_1\otimes\mathbf B_1\\\updownarrow\\\mathbf A_2\otimes\mathbf B_2  \end{array}\right)=
\begin{array}{c}\phi\left(\mathbf A_1\otimes\mathbf B_1\right)\\\updownarrow\\\phi\left(\mathbf A_2\otimes\mathbf B_2  \right) \end{array}.
\]
\end{itemize}
This proves the result.
\end{proof}
The proof of Theorem \ref{thkronecker} exhibits  an explicit isomorphism sending each $\mathbf A\otimes\mathbf B$ to $\mathbf A\odot\mathbf B$. It follows that
\begin{equation}\label{KroMorphism1}
(\mathbf A\leftrightarrow\mathbf B)\odot(\mathbf A'\leftrightarrow\mathbf B')=(\mathbf A\odot\mathbf A')\leftrightarrow (\mathbf B\odot\mathbf B').
\end{equation}
Indeed, the preimage by $\phi$ of this identity is
\begin{equation}
(\mathbf A\leftrightarrow\mathbf B)\otimes(\mathbf A'\leftrightarrow\mathbf B')=(\mathbf A\otimes\mathbf A')\leftrightarrow (\mathbf B\otimes\mathbf B').
\end{equation}
In the same way, we obtain
\begin{equation}\label{KroMorphism2}
\begin{array}{c}\mathbf A\\\updownarrow\\\mathbf B\end{array}\odot\begin{array}{c}\mathbf A'\\\updownarrow\\\mathbf B'\end{array}=\begin{array}{c}\mathbf A\odot\mathbf A'\\\updownarrow\\ \mathbf B\odot\mathbf B'\end{array}.
\end{equation}

\subsection{Quasi-direct sum\label{Sum}}
Let $\mathbf A=\left(\hmat aIJ\right)\in\mathbb K(M;p,q)$ and 
$\mathbf B=\left(\hmat bIJ\right)\in\mathbb K(N,p,q)$. We define the \emph{quasi-direct sum} of $\mathbf A$ and $\mathbf B$ by 
\[
\mathbf A\hat\oplus\mathbf B:=\left(\hmat cIJ\right)_{I,J}\in\mathbb K(M+N,p,q)
\]
with
\[
\hmat cIJ=\left\{
\begin{array}{ll}
\hmat aIJ&\mbox{ if }I\in[M]^p\mbox{ and }J\in[M]^q,\\
\hmat b{I-[M,\dots,M]}{J-[M,\dots,M]}&\mbox{ if }I-[M,\dots,M]\in[N]^p\mbox{ and }J-[M,\dots,M]\in[N]^q,\\
0&\mbox{ otherwise.}
\end{array}
\right.
\]
From this definition, one has
\begin{equation}\label{IoplusI}
\mathbf I(M)\hat\oplus\mathbf I(N)=\mathbf I(M+N).
\end{equation}
and for any $I\in[M]^{p}$, $J\in [M]^{q}$, $I'=[i'_{1},\dots,i'_{p}]\in[N]^{p}$, and $J'=[j'_{1},\dots,j'_{q}]\in[N]^{q}$, we have
\begin{equation}\label{EoplusE}
	\mathbf E(M,p,q;I,J)\hat\oplus \mathbf E(N,p,q;I',J')=\mathbf E(M+N,p,q;I,J)+
	\mathbf E(M+N,p,q;[i'_{1},\dots,i'_{p}]+[M^{p}],[j'_{1},\dots,j'_{q}]+[M^{q}]),
\end{equation}
where $[M^{n}]=[\overbrace{M,\dots,M}^{n\times}]$.
We do not call this operation ``direct sum'' because it not compatible with the horizontal composition as 
shown by the simplest counter example below.\\
Consider the hypermatrix $\left(\mathbf I(1)\hat\oplus\mathbf I(1)\right)\leftrightarrow \left(\mathbf I(1)\hat\oplus\mathbf I(1)\right)=\mathbf I(2)^{\leftrightarrow 2}$ which differs from $\left(\mathbf I(1)\leftrightarrow\mathbf I(1)\right)\hat\oplus \left(\mathbf I(1)\leftrightarrow \mathbf I(1)\right)$. Indeed,
\[\left(\mathbf I(1)\leftrightarrow \mathbf I(1)\right)\hat\oplus \left(\mathbf I(1)\leftrightarrow I(1)\right)={}
\mathbf I(1)^{\leftrightarrow 2}\hat\oplus \mathbf I(1)^{\leftrightarrow 2}=\left(\hmat \iota IJ\right)_{I,J}
\]
where
\[
\hmat \iota IJ=\left\{
\begin{array}{ll}
1&\mbox{ if }I=J\mbox{ and }(I=[1,1]\mbox{ or }I=[2,2]),\\
0&\mbox{ otherwise}.
\end{array}
\right.
\]
We remark that $\hmat \iota{12}{12}=0$ whilst the corresponding entry in $\mathbf I(2)^{\leftrightarrow 2}$ equals $1$.\\
So in general
\begin{equation}
(\mathbf A\hat\oplus\mathbf A')\leftrightarrow(\mathbf B\hat\oplus\mathbf B')\neq 
(\mathbf A\leftrightarrow \mathbf B)\hat\oplus(\mathbf A'\leftrightarrow\mathbf B').
\end{equation}
However the quasi-direct sum is compatible with the vertical composition.
\begin{proposition}\label{Moplus}
Let $\mathbf A\in\mathbb K(M;p,q)$, $\mathbf B\in\mathbb K(M;q,r)$, $\mathbf A'\in\mathbb K(N,p,q)$ and 
$\mathbf B'\in\mathbb K(N,q,r)$ be four hypermatrices. We have
\[
\begin{array}{c}
\mathbf A\hat\oplus\mathbf A'\\
\updownarrow\\
\mathbf B\hat\oplus\mathbf B'\\
\end{array}=\left(\begin{array}{c}\mathbf A\\\updownarrow\\\mathbf B\end{array}\right)\hat\oplus
\left(\begin{array}{c}\mathbf A'\\\updownarrow\\\mathbf B'\end{array}\right).
\]
\end{proposition}
\begin{proof}
For convenience let us set $\mathbf A\hat\oplus\mathbf A'=\left(\hmat cIJ\right)_{IJ}$, 
$\mathbf B\hat\oplus\mathbf B'=\left(\hmat dIJ\right)_{IJ}$, $\begin{array}{c}
\mathbf A\hat\oplus\mathbf A'\\
\updownarrow\\
\mathbf B\hat\oplus\mathbf B'\\
\end{array}=\left(\hmat eIJ\right)_{IJ}$,  and  $\left(\begin{array}{c}\mathbf A\\\updownarrow\\\mathbf B\end{array}\right)\hat\oplus
\left(\begin{array}{c}\mathbf A'\\\updownarrow\\\mathbf B'\end{array}\right)=\left(\hmat {g}IJ\right)_{IJ}$.
Our goal consists to prove that for each $I,J$, $\hmat eIJ=\hmat gIJ$. To this aim we  write the entries of each matrices as a function of $\hmat aIJ$, $\hmat {a'}IJ$, $\hmat bIJ$ and $\hmat {b'}IJ$ which are, respectively, the entries of $\mathbf A,\mathbf  A',\mathbf  B$ and $\mathbf B'$.
First note that
\[
\hmat eIJ=\displaystyle\sum_{K\in[M+N]^q}\hmat cIK\hmat dKJ
=\left(\displaystyle\sum_{K\in[M]^q}\hmat cIK\hmat dKJ\right)
+
\left(\displaystyle\sum_{K-[M,\dots,M]\in[N]^q}\hmat cIK\hmat dKJ\right)
\]
But the two sums in the right hand side can not be simultaneously non zero. Hence 
\[\hmat eIJ=\left\{\begin{array}{ll}
\displaystyle\sum_{K\in[M]^q}\hmat aIK\hmat bKJ
&\mbox{ if }I\in[M]^p\mbox{ and }J\in[M]^q\\
\displaystyle\sum_{K\in[N']^q}\hmat {a'}{I-[M,\dots,M]}{K}\hmat {b'}{K}{J-[M,\dots,M]}
&\mbox{ if }I-[M,\dots,M]\in[N]^p\mbox{ and }J-[M,\dots,M]\in[N]^q\\
0&\mbox{ otherwise}.
\end{array}\right.
\]
This is exactly the value of $\hmat gIJ$.
\end{proof}
Remark also that $\hat\oplus$ is bilinear
\begin{equation}\label{bilinoplus}
	(\mathbf A+\mathbf A')\hat\oplus (\mathbf B+\mathbf B')=
	\mathbf A\hat\oplus\mathbf B+\mathbf A\hat\oplus\mathbf B'+\mathbf A'\hat\oplus\mathbf B+\mathbf A'\hat\oplus\mathbf B'
\end{equation}

Consider now  more complicated identities.
\begin{lemma}
Let $m_a, n_a, m_b, m_c, p_q, q_d, q_e\geq 0$ and $N_0, N_1, n_b, n_c, p_e>0$ such that $n_a+n_b> p_d$ and $n_a+n_b+n_c=p_q+p_e$.
\begin{itemize}
\item
Let $\mathbf A=\left(\hmat a IJ\right)\in\mathbb K(N_0+ N_1; m_a,n_a)$, 
$\mathbf D=\left(\hmat d IJ\right)\in\mathbb K(N_0+ N_1; p_d,q_d)$, 
$\mathbf B_i=\left(\hmat {b_i} IJ\right)\in\mathbb K(N_i; m_b,n_b)$, 
$\mathbf C_i=\left(\hmat {c_i} IJ\right)\in\mathbb K(N_i; m_c,n_c)$ and 
$\mathbf E_i=\left(\hmat {e_i} IJ\right)\in\mathbb K(N_i; p_e,q_e)$ ($i=0,1$). One has
\begin{equation}\label{A+B+C}
\begin{array}{c}
\mathbf A\leftrightarrow\left((\mathbf B_0\leftrightarrow \mathbf C_0)\hat\oplus(\mathbf B_1\leftrightarrow \mathbf C_1)\right)\\
\updownarrow\\
\mathbf D\leftrightarrow (\mathbf E_0\hat\oplus\mathbf E_1)
\end{array}
=\begin{array}{c}
\mathbf A\leftrightarrow (\mathbf B_0\hat\oplus\mathbf B_1)\leftrightarrow (\mathbf C_0\hat\oplus\mathbf C_1)\\
\updownarrow \\
\mathbf D\leftrightarrow (\mathbf E_0\hat\oplus \mathbf E_1)
\end{array}
\end{equation}
\item Let $\mathbf A=\left(\hmat a IJ\right)\in\mathbb K(N_0+ N_1; n_a,m_a)$,
 $\mathbf D=\left(\hmat d IJ\right)\in\mathbb K(N_0+ N_1; q_d,p_d)$,
  $\mathbf B_i=\left(\hmat {b_i} IJ\right)\in\mathbb K(N_i; n_b,m_b)$,
   $\mathbf C_i=\left(\hmat {c_i} IJ\right)\in\mathbb K(N_i; n_c,m_c)$ and 
   $\mathbf E_i=\left(\hmat {e_i} IJ\right)\in\mathbb K(N_i; q_e,p_e)$ ($i=0,1$). One has
\begin{equation}\label{D+E}
\begin{array}{c}
\mathbf D\leftrightarrow (\mathbf E_0\hat\oplus\mathbf E_1)\\
\updownarrow\\
\mathbf A\leftrightarrow\left((\mathbf B_0\leftrightarrow \mathbf C_0)\hat\oplus(\mathbf B_1\leftrightarrow \mathbf C_1)\right)
\end{array}
=\begin{array}{c}
\mathbf D\leftrightarrow (\mathbf E_0\hat\oplus \mathbf E_1)\\
\updownarrow \\
\mathbf A\leftrightarrow (\mathbf B_0\hat\oplus\mathbf B_1)\leftrightarrow (\mathbf C_0\hat\oplus\mathbf C_1)
\end{array}
\end{equation}
\end{itemize}
\end{lemma}
\begin{proof}
We  prove only (\ref{A+B+C}); the other identity being hence obtained by symmetry.
From the bilinearity of $\hat\oplus$, $\leftrightarrow$ and $\updownarrow$ and Claim \ref{M2E}, it suffices to prove the result in the cases where $\mathbf A=\mathbf E(N_{0}+N_{1},m_{a},n_{a};I_{a},J_{a})$,
$\mathbf D=\mathbf E(N_{0}+N_{1},p_{d},q_{d};I_{d},J_{d})$,
$\mathbf B_{i}=\mathbf E(N_{i},m_{b},n_{b};I_{b}^{(i)},J_{b}^{(i)})$,
$\mathbf C_{i}=\mathbf E(N_{i},m_{c},n_{c};I_{c}^{(i)},J_{c}^{(i)})$, and
$\mathbf E_{i}=\mathbf E(N_{i},p_{e},q_{e};I_{e}^{(i)},J_{e}^{(i)})$.
One has
\[{}\begin{array}{l}
\mathbf A\leftrightarrow\left((\mathbf B_0\leftrightarrow \mathbf C_0)\hat\oplus(\mathbf B_1\leftrightarrow \mathbf C_1)\right)={}
\mathbf E\left(N_{0}+N_{1},m_{a}+m_{b}+m_{c},n_{a}+n_{b}+n_{c};I_{a}I_{b}^{(0)}I_{c}^{(0)},J_{a}J_{b}^{(0)}J_{c}^{(0)}\right)\\
+\mathbf E\left(N_{0}+N_{1},m_{a}+m_{b}+m_{c},n_{a}+n_{b}+n_{c};I_{a}I_{b}^{(1)}I_{c}^{(1)}+[0^{m_{a}}N_{0}^{m_{b}+m_{v}}],
J_{a}J_{b}^{(1)}J_{c}^{(1)}+[0^{n_{a}}N_{0}^{n_{b}+n_{v}}]\right)
\end{array}
\]
and
\[{}\begin{array}{l}
\mathbf D\leftrightarrow\left(\mathbf E_0\hat\oplus \mathbf E_1\right)={}
\mathbf E\left(N_{0}+N_{1},p_{d}+p_{e},q_{d}+q_{e};I_{d}I_{e}^{(0)},J_{d}J_{e}^{(0)}\right)
\\+\mathbf E\left(N_{0}+N_{1},p_{d}+p_{e},q_{d}+q_{e};I_{d}I_{e}^{(1)}+[0^{p_{d}}N_{0}^{p_{e}}],
J_{d}J_{e}^{(1)}+[0^{q_{d}}N_{0}^{q_{e}}]\right)\end{array}
\]
Since $n_{a}+n_{b}>p_{e}$, we have
\[\begin{array}{l}
\begin{array}{c}
\mathbf A\leftrightarrow\left((\mathbf B_0\leftrightarrow \mathbf C_0)\hat\oplus(\mathbf B_1\leftrightarrow \mathbf C_1)\right)\\
\updownarrow\\
\mathbf D\leftrightarrow (\mathbf E_0\hat\oplus\mathbf E_1)
\end{array}=
\begin{array}{c}
\mathbf E\left(N_{0}+N_{1},m_{a}+m_{b}+m_{c},n_{a}+n_{b}+n_{c};I_{a}I_{b}^{(0)}I_{c}^{(0)},J_{a}J_{b}^{(0)}J_{c}^{(0)}\right)
\\\updownarrow\\
\mathbf E\left(N_{0}+N_{1},p_{d}+p_{e},q_{d}+q_{e};I_{d}I_{e}^{(0)},J_{d}J_{e}^{(0)}\right)
\end{array}\\
+\begin{array}{c}
\mathbf E\left(N_{0}+N_{1},m_{a}+m_{b}+m_{c},n_{a}+n_{b}+n_{c};I_{a}I_{b}^{(1)}I_{c}^{(1)}+[0^{m_{a}}N_{0}^{m_{b}+m_{v}}],
J_{a}J_{b}^{(1)}J_{c}^{(1)}+[0^{n_{a}}N_{0}^{n_{b}+n_{v}}]\right)\\\updownarrow\\
\mathbf E\left(N_{0}+N_{1},p_{d}+p_{e},q_{d}+q_{e};I_{d}I_{e}^{(1)}+[0^{p_{d}}N_{0}^{p_{e}}],
J_{d}J_{e}^{(1)}+[0^{q_{d}}N_{0}^{q_{e}}]\right)\end{array}\\=
\begin{array}{c}
\mathbf A\leftrightarrow (\mathbf B_0\hat\oplus\mathbf B_1)\leftrightarrow (\mathbf C_0\hat\oplus\mathbf C_1)\\
\updownarrow \\
\mathbf D\leftrightarrow (\mathbf E_0\hat\oplus \mathbf E_1)
\end{array}.
\end{array}
\]

\end{proof}

More generally, 
\begin{theorem}\label{GenIdoplus}
Let $\mathbf A_\alpha(\epsilon)=\left(\hmat {a(\alpha,\epsilon)}IJ\right)_{I,J}\in \mathbb K(N_{\epsilon};m_\alpha,n_\alpha)$, 
$\mathbf B_\beta(\epsilon)=\left(\hmat {b(\beta,\epsilon)}IJ\right)_{I,J}\in \mathbb K(N_{\epsilon};p_\beta,q_\beta)$,{}
for $\alpha\in \{1,\dots,k\}$, $\beta\in \{1,\dots,\ell\}$ and $\epsilon=0,1$  satisfying
\begin{enumerate}
\item  $p_1+\cdots+p_\beta=n_1+\cdots+n_\alpha$ implies $\beta=\ell$ and $\alpha=k$.\label{enumcond1}
\item $ n_\alpha, p_\beta>0$ 
\end{enumerate}  
Then we have
\begin{equation}
\begin{array}{c}
\left(\mathbf A_1(0)\hat\oplus\mathbf A_1(1)\right)\leftrightarrow\cdots\leftrightarrow \left(\mathbf A_k(0)\hat\oplus\mathbf A_k(1)\right)\\
\updownarrow\\
\left(\mathbf B_1(0)\hat\oplus\mathbf B_1(1)\right)\leftrightarrow\cdots\leftrightarrow \left(\mathbf B_\ell(0)\hat\oplus\mathbf B_\ell(1)\right)\end{array}=
\begin{array}{c}
\left(\mathbf A_1(0)\leftrightarrow\cdots\leftrightarrow \mathbf A_k(0)\right)\\
\updownarrow\\
\left(\mathbf B_1(0)\leftrightarrow\cdots\leftrightarrow \mathbf B_\ell(0)\right)\end{array}
\hat\oplus
\begin{array}{c}
\left(\mathbf A_1(1)\leftrightarrow\cdots\leftrightarrow \mathbf A_k(1)\right)\\
\updownarrow\\
\left(\mathbf B_1(1)\leftrightarrow\cdots\leftrightarrow \mathbf B_\ell(1)\right)\end{array}
\end{equation}
\end{theorem}
\begin{proof}
We prove the result by induction on $k+\ell$. The initial case is given by Proposition \ref{Moplus}. The condition \ref{enumcond1} implies either $p_1+\dots+p_{\ell-1}<n_1+\dots+n_{k-1}$ or $p_1+\dots+p_{\ell-1}>n_1+\dots+n_{k-1}$. The two cases been symmetrical let us consider only the case when $p_1+\dots+p_{\ell-1}<n_1+\dots+n_{k-1}$. Set $\mathbf A'(\epsilon)=\mathbf A_{k-1}(\epsilon)\leftrightarrow \mathbf A_{k}(\epsilon)$ for $\epsilon=\{0,1\}$. By induction, we have
\[\begin{array}{l}
\begin{array}{c}
\left(\mathbf A_0(0)\leftrightarrow\cdots\leftrightarrow \mathbf A_{k-2}(0)\leftrightarrow \mathbf A'(0)\right)\\
\updownarrow\\
\left(\mathbf B_1(0)\leftrightarrow\cdots\leftrightarrow \mathbf B_\ell(0)\right)\end{array}
\hat\oplus
\begin{array}{c}
\left(\mathbf A_0(1)\leftrightarrow\cdots\leftrightarrow \mathbf A_{k-2}(1)\leftrightarrow \mathbf A'(1)\right)\\
\updownarrow\\
\left(\mathbf B_1(1)\leftrightarrow\cdots\leftrightarrow \mathbf B_\ell(1)\right)\end{array}
\\\ \\
=\begin{array}{c}
\left(\mathbf A_0(0)\hat\oplus\mathbf A_0(1)\right)\leftrightarrow\cdots\leftrightarrow \left(\mathbf A_{k-2}(0)\hat\oplus\mathbf A_{k-2}(1)\right)\leftrightarrow\left(\mathbf A'(0)\hat\oplus\mathbf A'(1)\right)\\
\updownarrow\\
\left(\mathbf B_1(0)\hat\oplus\mathbf B_1(1)\right)\leftrightarrow\cdots\leftrightarrow \left(\mathbf B_\ell(0)\hat\oplus\mathbf B_\ell(1)\right)\end{array}
\end{array}\] 
with the notation $\mathbf A_0(\epsilon)=\mathbf I_0=1$.

Setting $\mathbf A''=\left(\mathbf A_0(0)\hat\oplus\mathbf A_0(1)\right)\leftrightarrow\cdots\leftrightarrow \left(\mathbf A_{k-2}(0)\hat\oplus\mathbf A_{k-2}(1)\right)$, $\mathbf B''_\epsilon=\mathbf A_{k-1}(\epsilon)$, $\mathbf C''_\epsilon=\mathbf A_{k}(\epsilon)$, $\mathbf D''=
\left(\mathbf B_1(0)\hat\oplus\mathbf B_1(1)\right)\leftrightarrow\cdots\leftrightarrow 
\left(\mathbf B_{\ell-1}(0)\hat\oplus\mathbf B_{\ell-1}(1)\right)$, $\mathbf E''_\epsilon=\mathbf B_\ell(\epsilon)$, $m_a=m_1+\cdots+m_{k-2}$, $n_a=n_1+\dots+n_{k-2}$, $m_b=m_{k-1}$, $n_b=n_{k-1}$, $m_c=m_{k}$, $n_c=n_{k}$, $p_d=p_1+\cdots+p_{\ell-1}$, $q_d=q_1+\cdots+q_{\ell-1}$, $p_e=p_\ell$ and $q_e=q_\ell$, we obtain
\[
\begin{array}{c}
\left(\mathbf A_1(0)\leftrightarrow\cdots\leftrightarrow \mathbf A_k(0)\right)\\
\updownarrow\\
\left(\mathbf B_1(0)\leftrightarrow\cdots\leftrightarrow \mathbf B_\ell(0)\right)\end{array}
\hat\oplus
\begin{array}{c}
\left(\mathbf A_1(1)\leftrightarrow\cdots\leftrightarrow \mathbf A_k(1)\right)\\
\updownarrow\\
\left(\mathbf B_1(1)\leftrightarrow\cdots\leftrightarrow \mathbf B_\ell(1)\right)\end{array}
=
\begin{array}{c}
\mathbf A''\leftrightarrow\left((\mathbf B''_0\leftrightarrow \mathbf C''_0)\hat\oplus(\mathbf B''_1\leftrightarrow \mathbf C''_1)\right)\\
\updownarrow\\
\mathbf D''\leftrightarrow (\mathbf E''_0\hat\oplus\mathbf E''_1)
\end{array}
\]
with $m_a, n_a, m_b, m_c, p_d, q_d, q_e\geq 0$, $n_b,n_c, p_e>0$, $n_a+n_b>p_d$ and $n_a+n_b+n_c=p_d+p_e$. So by Equation (\ref{A+B+C}), we obtain

\[\begin{array}{rcl}
\begin{array}{c}
\left(\mathbf A_1(0)\leftrightarrow\cdots\leftrightarrow \mathbf A_k(0)\right)\\
\updownarrow\\
\left(\mathbf B_1(0)\leftrightarrow\cdots\leftrightarrow \mathbf B_\ell(0)\right)\end{array}
\hat\oplus
\begin{array}{c}
\left(\mathbf A_1(1)\leftrightarrow\cdots\leftrightarrow \mathbf A_k(1)\right)\\
\updownarrow\\
\left(\mathbf B_1(1)\leftrightarrow\cdots\leftrightarrow \mathbf B_\ell(1)\right)\end{array}
&=&\left.\begin{array}{c}
\mathbf A''\leftrightarrow (\mathbf B''_0\hat\oplus\mathbf B''_1)\leftrightarrow (\mathbf C''_0\hat\oplus\mathbf C''_1)\\
\updownarrow \\
\mathbf D''\leftrightarrow (\mathbf E''_0\hat\oplus \mathbf E''_1)
\end{array}\right.
\\\ \\
&=&\left.
\begin{array}{c}
\left(\mathbf A_0(0)\hat\oplus\mathbf A_0(1)\right)\leftrightarrow\cdots\leftrightarrow 
\left(\mathbf A_k(0)\hat\oplus\mathbf A_k(1)\right)\\
\updownarrow\\
\left(\mathbf B_1(0)\hat\oplus\mathbf B_1(1)\right)\leftrightarrow\cdots\leftrightarrow{}
 \left(\mathbf B_\ell(0)\hat\oplus\mathbf B_\ell(1)\right)\end{array}
\right.\\\ \\&=&
\left.\begin{array}{c}
\left(\mathbf A_1(0)\hat\oplus\mathbf A_1(1)\right)\leftrightarrow\cdots\leftrightarrow \left(\mathbf A_k(0)\hat\oplus\mathbf A_k(1)\right)\\
\updownarrow\\
\left(\mathbf B_1(0)\hat\oplus\mathbf B_1(1)\right)\leftrightarrow\cdots\leftrightarrow \left(\mathbf B_\ell(0)\hat\oplus\mathbf B_\ell(1)\right)\end{array}\right..
\end{array}\]
This ends the proof.
\end{proof}

\section{Representations of circuit PROs}
\subsection{Representations of Free PROs\label{RepFreePros}}
Let $\mathbb K$ be a semiring and $\mathcal P$ be a PRO. A \emph{multilinear representation} of $\mathcal P$ is a morphism of PRO 
from $\mathcal P$
to $\mathbb K(N)$. Straightforwardly from the definition of freeness, one has
\begin{claim}
  Any free PRO has multilinear representations.
\end{claim}

Remark that the free PRO $\mathcal F(\mathcal X)$ is naturally extended as a $\mathbb K$-ModPro $\mathbb K\langle\mathcal X\rangle:={}
\bigcup_{p,q\geq 1}\mathbb K[\mathcal F_{p,q}(\mathcal X)]$ 
where $\mathbb K[\mathcal F_{p,q}(\mathcal X)]$ denotes the free $\mathbb K$-module generated by $\mathcal F_{p,q}(\mathcal X)$.
 A representation
$\mu:\mathcal F(\mathcal X)\rightarrow \mathbb K(N)$ induces  a morphism of ModPro $\mathbb K\langle\mathcal X\rangle\rightarrow \mathbb K(N)$ .

The \emph{Hadamard product} of two representations $\mu:\mathcal F(\mathcal X)\rightarrow\mathbb K(M)$ and 
$\mu':\mathcal F(\mathcal X)\rightarrow\mathbb K(N)$ is the unique representation
 $\mu\hat\odot\mu':\mathcal F(\mathcal X)\rightarrow\mathbb K(MN)$ satisfying 
 $(\mu\hat\odot\mu')(\mathfrak x)=\mu(\mathfrak x)\odot\mu'(\mathfrak x)$ for each $\mathfrak x\in\mathcal X$.\\
As a direct consequence of Theorem \ref{thkronecker}, we obtain the following result.
\begin{corollary}\label{HadamardRep}
Let $\mu:\mathcal F(\mathcal X)\rightarrow\mathbb K(M)$ and 
$\mu':\mathcal F(\mathcal X)\rightarrow\mathbb K(N)$ be two representations of the free PRO. 
Then for any $\mathfrak x\in \mathcal F(\mathcal C)$, $(\mu\hat\odot\mu')(\mathfrak x)=\mu(\mathfrak x)\odot\mu'(\mathfrak x)$.
\end{corollary}
\begin{proof}
The result is proved by induction using equation (\ref{KroMorphism1}) and (\ref{KroMorphism2}).
\end{proof}
In this section, we focus on representation of Circuit PROs. Consider a set of chips $\mathcal C=\bigcup_{n,m\geq 1}\mathcal C_{n,m}$. 
Any graded map $\mu:\mathcal C\rightarrow\mathbb K(N)$ defines a unique 
morphism of PROs, 
$\mu:\mathcal Circ(\mathcal C)\rightarrow \mathbb K(N)$.\\


\subsection{Quasi direct sum  of representations of circuit PROs\label{RepQuasi}}
The \emph{quasi direct sum} of two representations $\mu:\mathcal Circ(\mathcal C)\rightarrow\mathbb K(M)$ and 
$\mu':\mathcal Circ(\mathcal C)\rightarrow\mathbb K(N)$ is the unique representation 
$\mu\hat\oplus\mu':\mathcal Circ(\mathcal C)\rightarrow\mathbb K(M+N)$ satisfying $(\mu\hat\oplus\mu')(\mathfrak c)=\mu(\mathfrak c)\hat\oplus\mu'(\mathfrak c)$ for each $\mathfrak c\in\mathcal C$.\\
Note that in general
\[
(\mu\hat\oplus\mu')(\mathfrak c\leftrightarrow \mathfrak c')\neq \mu(\mathfrak c\leftrightarrow \mathfrak c')\hat\oplus \mu'(\mathfrak c\leftrightarrow \mathfrak c').
\]

\begin{example}\rm
Let $\mu^0$ and $\mu^1$ be two $2$-dimensional representations. One has
\[
\mu^0(\multimapdotbothBvert\multimapdotbothBvert)=\mu^1(\multimapdotbothBvert\multimapdotbothBvert)=\left[\begin{array}{c|c|c|c}
1& 0&0&0\\\hline
0&1&0&0\\\hline
0&0&1&0\\\hline
0&0&0&1
\end{array}\right].
\]
Hence,
\begin{equation}\label{mumu}
\mu^0(\multimapdotbothBvert\multimapdotbothBvert)\hat\oplus\mu^1(\multimapdotbothBvert\multimapdotbothBvert)=\begin{array}{l}\ \\\ ^{11}\\\ ^{12}\\\ ^{13}\\\  ^{14}\\
\ ^{21}\\\ ^{22}\\\ ^{23}\\\  ^{24}\\
\ ^{31}\\\ ^{32}\\\ ^{33}\\\  ^{34}\\
\ ^{41}\\\ ^{42}\\\ ^{43}\\\  ^{44}\end{array}
\left[\begin{array}{c|c|c|c|c|c|c|c|c|c|c|c|c|c|c|c}\displaystyle\mathop1^{11}&\displaystyle\mathop0^{12}&\displaystyle\mathop0^{13}&\displaystyle\mathop0^{14}&\displaystyle\mathop0^{21}&\displaystyle\mathop0^{22}&\displaystyle\mathop0^{23}&\displaystyle\mathop0^{24}&\displaystyle\mathop0^{31}&\displaystyle\mathop0^{32}&\displaystyle\mathop0^{33}&\displaystyle\mathop0^{34}&\displaystyle\mathop0^{41}&\displaystyle\mathop0^{42}&\displaystyle\mathop0^{43}&\displaystyle\mathop0^{44}\\\hline
0&1&0&0&0&0&0&0&0&0&0&0&0&0&0&0\\\hline
0&0&\color{red}\mathbf0&0&0&0&0&0&0&0&0&0&0&0&0&0\\\hline
0&0&0&\color{red}\mathbf0&0&0&0&0&0&0&0&0&0&0&0&0\\\hline
0&0&0&0&1&0&0&0&0&0&0&0&0&0&0&0\\\hline
0&0&0&0&0&1&0&0&0&0&0&0&0&0&0&0\\\hline
0&0&0&0&0&0&\color{red}\mathbf0&0&0&0&0&0&0&0&0&0\\\hline
0&0&0&0&0&0&0&\color{red}\mathbf0&0&0&0&0&0&0&0&0\\\hline
0&0&0&0&0&0&0&0&\color{red}\mathbf0&0&0&0&0&0&0&0\\\hline
0&0&0&0&0&0&0&0&0&\color{red}\mathbf0&0&0&0&0&0&0\\\hline
0&0&0&0&0&0&0&0&0&0&1&0&0&0&0&0\\\hline
0&0&0&0&0&0&0&0&0&0&0&1&0&0&0&0\\\hline
0&0&0&0&0&0&0&0&0&0&0&0&\color{red}\mathbf0&0&0&0\\\hline
0&0&0&0&0&0&0&0&0&0&0&0&0&\color{red}\mathbf0&0&0\\\hline
0&0&0&0&0&0&0&0&0&0&0&0&0&0&1&0\\\hline
0&0&0&0&0&0&0&0&0&0&0&0&0&0&0&1
\end{array}\right].
\end{equation}
Whilst
\[
\mu^0\hat\oplus\mu^1(\multimapdotbothBvert\multimapdotbothBvert)=
\mu^0\hat\oplus\mu^1(\multimapdotbothBvert)\leftrightarrow \mu^0\hat\oplus\mu^1(\multimapdotbothBvert)=\left(\hmat \delta IJ\right)_{I\in[4]^2,J\in[4]^2}.
\]
Hence, we verify easily that $\mu^0\hat\oplus\mu^1(\multimapdotbothBvert\multimapdotbothBvert)\neq 
\mu^0(\multimapdotbothBvert\multimapdotbothBvert)\hat\oplus\mu^1(\multimapdotbothBvert\multimapdotbothBvert)$ (look at the diagonal in equality (\ref{mumu})).
\end{example}

\medskip
Nevertheless, Proposition \ref{Moplus} implies that if $(\mu\hat\oplus\mu')(\mathfrak c)=\mu(\mathfrak c)\hat\oplus\mu'(\mathfrak c)$ and
$(\mu\hat\oplus\mu')(\mathfrak c')=\mu(\mathfrak c')\hat\oplus\mu'(\mathfrak c')$ then
\begin{equation}
(\mu\hat\oplus\mu')\left(\begin{array}{c}\mathfrak c\\\updownarrow\\ \mathfrak c'\end{array}\right)= \mu\left(\begin{array}{c}\mathfrak c\\\updownarrow \\\mathfrak c'\end{array}\right)\hat\oplus \mu'\left(\begin{array}{c}\mathfrak c\\\updownarrow \\\mathfrak c'\end{array}\right).
\end{equation}
More generally on has
\begin{theorem}\label{sumRep}
Let $\mu^0:\mathcal Circ(\mathcal C)\rightarrow \mathbb K(N_0)$  and $\mu^1:\mathcal Circ(\mathcal C)\rightarrow \mathbb K(N_1)$ 
be two representations.
For any  connected circuit $\mathfrak c\in \mathcal Circ(\mathcal C)_{m,n}$,
 one has \begin{equation}(\mu^0\hat\oplus\mu^1)(\mathfrak c)=\mu^0(\mathfrak c)\hat\oplus\mu^1(\mathfrak c).\end{equation}
\end{theorem}
\begin{proof}
We prove the result by induction following the point 3 of Proposition \ref{Pconnected}.\\
If $\mathfrak c$ is a chip then the property is straightforward from the definition.\\
If $\mathfrak c=\multimapdotbothBvert$ then $\mu^0(\multimapdotbothBvert)=\mathbf I(N_{0})$, 
$\mu^1(\multimapdotbothBvert)=\mathbf I(N_{1})$ and 
$$(\mu^0\hat\oplus\mu^1)(\multimapdotbothBvert)=\mathbf I(N_{0}+N_{1})=\mathbf I(N_{0})\hat\oplus\mathbf I(N_{1})=
\mu^0(\multimapdotbothBvert)\hat\oplus \mu^1(\multimapdotbothBvert),$$ from  equality (\ref{IoplusI}).\\
For the other cases, $\mathfrak c$ admits an expression as
\[
\mathfrak c=\begin{array}{c}\mathfrak p_1\leftrightarrow\cdots\leftrightarrow\mathfrak p_k\\\updownarrow \\\mathfrak q_1\leftrightarrow\cdots\leftrightarrow\mathfrak q_\ell\end{array},
\]
such that for each $1\leq i \leq k$ and $1\leq j\leq \ell$, $\mathfrak p_i\in\mathcal Circuit(\mathcal C)_{m_i,n_i}$ and 
$\mathfrak q_j\in\mathcal Circ(\mathcal C)_{p_j,q_j}$ are connected circuits and $n_1+\cdots+n_i=p_1+\cdots p_j$ implies $i=k$ and $j=\ell$.
 By induction, for each $i$ and $j$, 
 $(\mu^0\hat\oplus\mu^1)(\mathfrak p_i)=\mu^0(\mathfrak p_i)\hat\oplus\mu^1(\mathfrak p_i)$ and  
 $(\mu^0\hat\oplus\mu^1)(\mathfrak q_j)=\mu^0(\mathfrak q_j)\hat\oplus\mu^1(\mathfrak q_j)$.\\

Hence,
\[
(\mu^0\hat\oplus\mu^1)(c)=\begin{array}{c}\mu^0(\mathfrak p_1)\hat\oplus\mu^1(\mathfrak p_1)\leftrightarrow\cdots\leftrightarrow\mu^0(\mathfrak p_k)\hat\oplus\mu^1(\mathfrak p_k)\\\updownarrow \\\mu^0(\mathfrak q_1)\hat\oplus\mu^1(\mathfrak q_1)\leftrightarrow\cdots\leftrightarrow\mathfrak \mu^0(\mathfrak q_\ell)\hat\oplus\mu^1(\mathfrak q_\ell)\end{array}.
\]
From Theorem \ref{GenIdoplus} we obtain
\[
(\mu^0\hat\oplus\mu^1)(c)=\begin{array}{c}\mu^0(\mathfrak p_1)\leftrightarrow\cdots\leftrightarrow\mu^0(\mathfrak p_k)\\\updownarrow \\\mu^0(\mathfrak q_1)\leftrightarrow\cdots\leftrightarrow\mathfrak \mu^0(\mathfrak q_\ell)\end{array}\hat\oplus\begin{array}{c}\mu^1(\mathfrak p_1)\leftrightarrow\cdots\leftrightarrow\mu^1(\mathfrak p_k)\\\updownarrow \\\mu^1(\mathfrak q_1)\leftrightarrow\cdots\leftrightarrow\mathfrak \mu^1(\mathfrak q_\ell)\end{array}=\mu^0(\mathfrak c)\hat\oplus\mu^1(\mathfrak c).
\]
\end{proof}

\subsection{From paths to representations\label{RepPath}}

Let $\mu:\mathcal Circ(\mathcal C)\longrightarrow \mathbb K(N)$ be a representation of a circuit PRO.
 We  denote by $\hmat{\mu(\mathfrak p)}IJ$ the entries of the hypermatrix $\mu(\mathfrak p)$ 
 for any circuit $\mathfrak p\in \mathcal F(\mathcal X)$.\\
We define the map $\tilde \mu: \mathcal Paths(\mathcal C,N)\longrightarrow \mathbb K$ by
\begin{itemize}
\item $\tilde\mu\left(\begin{array}{c}I\\\mathfrak c\\J\end{array}\right)=\hmat{\mu(\mathfrak c)}IJ$ for $\mathfrak c\in\mathcal C$
\item $\tilde\mu\left(\begin{array}{c}\mathfrak p\\\updownarrow\\\mathfrak q\end{array}\right)={
\tilde\mu}(\mathfrak p){\tilde\mu}(\mathfrak q)$ when $\mathtt{Out}(\mathfrak q)=\mathtt{In}(\mathfrak p)$.
\item $\tilde\mu\left(\mathfrak p\leftrightarrow\mathfrak p'\right)={\tilde \mu}(\mathfrak p){\tilde\mu}(\mathfrak p').$
\end{itemize}
\begin{proposition}\label{rep2path}
We have
$$\hmat{\mu(\mathfrak p)}IJ=\mathop\sum\tilde\mu\left(\mathfrak q\right)$$
where the sum is over the paths $\mathfrak q\in\mathcal Paths(\mathcal C,N)$ such that ${u\left(\mathfrak q\right)=\mathfrak p}$, $\mathtt{Out}(\mathfrak q)=I$, and $\mathtt{In}(\mathfrak q)=J$
.\end{proposition}
\begin{proof}
We prove the result by induction. We have to consider three cases
\begin{itemize}
\item If $\mathfrak p\in\mathcal C$ then the proposition is straightforward  from the definition. 
\item If $\mathfrak p=\mathfrak p_1 \leftrightarrow\mathfrak p_2$, with $\mathfrak p_1\in \mathcal Circuit(\mathcal C)_{m_1,n_1}$, 
$\mathfrak p_2\in \mathcal Circuit(\mathcal C)_{m_2,n_2}$, and $\mathfrak p_1,\mathfrak p_2\neq\mathfrak p$, then
\[
\hmat{\mu(\mathfrak p_1 \leftrightarrow\mathfrak p_2)}{I_1I_2}{J_1J_2}=
\hmat{\mu(\mathfrak p_1)}{I_1}{J_1} \hmat{\mu(\mathfrak p_2)}{I_2}{J_2}
\]
for any $I_1\in[N]^{m_1}$, $I_2\in[N]^{m_2}$, $I_1\in[N]^{n_1}$, and $J_2\in[N]^{n_2}$. By induction, one obtains
\[
\hmat{\mu(\mathfrak p_1)}{I_1}{J_1}\leftrightarrow \hmat{\mu(\mathfrak p_2)}{I_2}{J_2}=
\sum\tilde\mu\left(\mathfrak q_1\right)\tilde\mu\left(\mathfrak q_2\right)=
\sum\tilde\mu\left(\mathfrak q_1\leftrightarrow \mathfrak q_2\right)
\]
where the sum is over the pairs of paths $\mathfrak q_1,\mathfrak 
q_2\in\mathcal Paths(\mathcal C,N)$ satisfying ${u\left(\mathfrak q_1\right)=\mathfrak p_1}$, $\mathtt{Out}(\mathfrak q_1)=I_1$, $\mathtt{In}(\mathfrak q_1)=J_1$, ${u\left(\mathfrak q_2\right)=\mathfrak p_2}$, $\mathtt{Out}(\mathfrak q_2)=I_2$, and $\mathtt{In}(\mathfrak q_2)=J_2$. We deduce our result from Lemma \ref{Path2Circ}.\\
\item If $\mathfrak p=\begin{array}{rcl}\mathfrak p'\\ \updownarrow\\\mathfrak p'\end{array}$ with $\mathfrak p'\in \mathcal Circ(\mathcal C)_{m,n}$, 
$\mathfrak p''\in \mathcal Circ(\mathcal C)_{n,p}$, and $\mathfrak p',\mathfrak  p''\neq \mathfrak p$, then one has
\[
\hmat{\mu\left(\mathfrak p\right)}{I}J=\displaystyle\mathop\sum_K\hmat{\mu\left(\mathfrak p'\right)}{I}K
\hmat{\mu\left(\mathfrak p''\right)}{K}I
\]
By induction, one has
\[
\hmat{\mu\left(\mathfrak p\right)}{I}J=\displaystyle\mathop\sum_K\mathop\sum_{\mathfrak q',\mathfrak q''} \tilde\mu\left(\mathfrak q'\right)
\tilde\mu\left(\mathfrak q''\right)
\]
where the second sum is over the pairs of paths $\mathfrak q',\mathfrak q''\in\mathcal Paths(\mathcal C,N)$ 
satisfying ${u\left(\mathfrak q'\right)=\mathfrak p'}$, $\mathtt{Out}(\mathfrak q')=I$, $\mathtt{In}(\mathfrak q')=K$,
 ${u\left(\mathfrak q''\right)=\mathfrak p''}$, $\mathtt{Out}(\mathfrak q'')=K$, and $\mathtt{In}(\mathfrak q'')=J$. Hence,
\[
\hmat{\mu\left(\mathfrak p\right)}{I}J=\displaystyle\mathop\sum_{\mathfrak q',\mathfrak q''} \tilde\mu\left(\begin{array}{c}\mathfrak q'\\\updownarrow\\
\mathfrak q''\end{array}\right)
\]
where the sum is over the pairs of paths $\mathfrak q',\mathfrak 
q''\in\mathcal Paths(\mathcal C,N)$ satisfying ${u\left(\mathfrak q'\right)=\mathfrak p'}$, $\mathtt{Out}(\mathfrak q')=I$, 
${u\left(\mathfrak q''\right)=\mathfrak p''}$, $\mathtt{Out}(\mathfrak q'')=\mathtt{In}(\mathfrak q')$, and $\mathtt{In}(\mathfrak q'')=J$. Again, Lemma \ref{Path2Circ} allows us to conclude.

\end{itemize}
\end{proof}
\begin{example}\rm
	Let  $\mathcal C=\mathcal C_{1,2}\cup\mathcal C_{2,2}$ with $\mathcal C_{2,2}=\left\{
\tikdanseq { \Apro 00{0.25}{0.25}00a{green}{a}}\right\}$ and $\mathcal C_{1,2}=\left\{
\tikdanseq { \Apro 00{0.25}{0.25}00b{cyan}{b}}\right\}$.
Let  $\mu:\mathcal Circuit(C)\longrightarrow\mathbb K(3)$ be the representation satisfying
\begin{itemize}
\item $\mu\left(\tikdanseq { \Apro 00{0.25}{0.25}00a{green}{b}}\right)=\mathbf A$ where
 $\displaystyle\mathop{ \mathbf A}_{2,3}^{2,3}=x$, 
$\displaystyle\mathop {\mathbf A}_{2,3}^{3,3}=x'$ and $\displaystyle\mathop{\mathbf  A}_{i,j}^{k,l}=0$ otherwise.
\item $\mu\left(\tikdanseq { \Apro 00{0.25}{0.25}00b{cyan}{b}}\right)=B$ where $\displaystyle\mathop{\mathbf B}_{2}^{1,3}=y
$, $\displaystyle\mathop{\mathbf B}_{3}^{1,3}=y'
$, and $\displaystyle\mathop{\mathbf B}_{i}^{j,k}=0$ in the other cases.
\end{itemize} 
Proposition \ref{rep2path} means that the representation $\mu$ 
can be graphically represented by an hypergraph as in figure \ref{ExCirc}, where only the non zero transition are drawn, 
and that the image of circuit is obtained by summing over all the generalized paths.
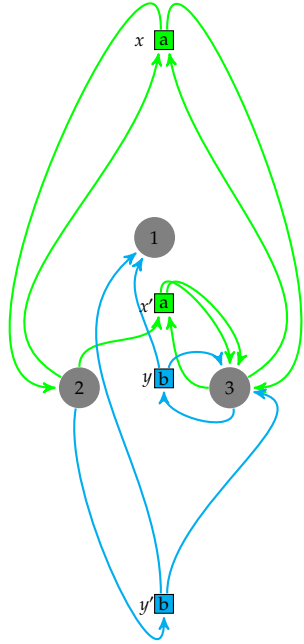
\begin{figure}[h]
\begin{center}
\begin{tikzpicture}
\node[circle,fill=gray](v1) at (1,2){1};
\node[circle,fill=gray](v2) at (0,0){2};
\node[circle,fill=gray](v3) at (2,0){3};
\Apro 11{0.25}{0.25}22a{green}{b1}
\draw[->,green,thick](uub11) to[out=90,in=90] (v3);
\draw[->,green,thick](uub12) to[out=90,in=70] (v3);
\draw[->,green,thick](v2)  to[out=90,in=260] (ddb11);
\draw[->,green,thick](v3)  to[out=180,in=280] (ddb12);

\Apro {1}{4.5}{0.25}{0.25}22a{green}{b2}
\draw[->,green,thick](uub21) to[out=90,in=180] (v2);
\draw[->,green,thick](uub22) to[out=90,in=0] (v3);
\draw[->,green,thick](v2)  to[out=150,in=260] (ddb21);
\draw[->,green,thick](v3)  to[out=30,in=280] (ddb22);

\Apro {1}{0}{0.25}{0.25}21b{cyan}{c}
\draw[->,cyan,thick](uuc1) to [out=100,in=240] (v1);
\draw[->,cyan,thick](uuc2) to [out=80,in=110] (v3);
\draw[->,cyan,thick](v3)  to [out=280,in=270] (ddc1);

\Apro {1}{-3}{0.25}{0.25}21b{cyan}{c2}
\draw[->,cyan,thick](uuc21) to [out=90,in=220] (v1);
\draw[->,cyan,thick](uuc22) to [out=90,in=350] (v3);
\draw[->,cyan,thick](v2)  to [out=260,in=270] (ddc21);

\node (p4) at (0.8,4.6) {$x$};
\node (p5) at (0.9,1.1) {$x'$};
\node (p6) at (0.9,0.1) {$y$};
\node (p8) at (0.9,-2.9) {$y'$};
\end{tikzpicture}
\end{center}
\caption{An example of hypergraph.\label{ExCirc}}
\end{figure}

For instance, we consider the following  entry:
\begin{equation}\begin{array}{c}\ 1,3,1,3\\\mu\left(\tikdanseq {
\Apro {1}{0.5}{0.25}{0.25}22a{green}{b}
\Apro {0.5}{1}{0.25}{0.25}21b{cyan}{c1}
\Apro {1.5}{1}{0.25}{0.25}21b{cyan}{c2}
\draw (uub1)--(ddc11);
\draw (uub2)--(ddc21);
}\right)\\\ 2,3\end{array}=\tilde\mu\left({}
\begin{array}{c}1,3,1,3\\
\tikdanseq{
\Apro {0}{0}{0.25}{0.25}22a{green}{b}
\Apro {-0.5}{0.5}{0.25}{0.25}21b{cyan}{c1}
\Apro {0.5}{0.5}{0.25}{0.25}21b{cyan}{c2}
\draw (uub1)--(ddc11);
\draw (uub2)--(ddc21);
\node(n1) at (-.25,0.25){3};
\node(n1) at (0.45,0.25){3};
}\\2,3\end{array}
\right)+
\tilde\mu\left({}
\begin{array}{c}1,3,1,3\\
\tikdanseq{
\Apro {0}{0}{0.25}{0.25}22a{green}{b}
\Apro {-0.5}{0.5}{0.25}{0.25}21b{cyan}{c1}
\Apro {0.5}{0.5}{0.25}{0.25}21b{cyan}{c2}
\draw (uub1)--(ddc11);
\draw (uub2)--(ddc21);
\node(n1) at (-.25,0.25){2};
\node(n1) at (0.45,0.25){3};
}\\2,3\end{array}
\right)=x'y'^{2}+xy'y.
\end{equation}.
\end{example}{}

\section{Toward a universal definition of automata\label{automata}}
Finite state automata belong to the large class of abstract machines studied in the context of information theory and formal languages.
In particular, they play an important role in the Chomsky hierarchy \cite{Chomsky,CS} since they characterize regular languages \cite{Kleene}.
 More precisely, automata provide an acceptance mechanism for words and the set of the words accepted by a given automaton is its recognized language. 
 According to  Kleene's theorem \cite{Kleene} the set of languages which are recognized by automata is the smallest set containing 
 the subsets of symbols and closed by union, concatenation and star.
 The first occurrence of this concept in literature
dates back to the paper of  McCulloch and Pitt \cite{MP} but the notion of non deterministic finite automaton (NFA) was formally 
introduced by Rabin and Scott \cite{RS}.  An automaton is a graph based machine with states and transitions labeled with letters.
More precisely, an automaton is a quintuple $(\Sigma,Q,I,F,\delta)$ where $\Sigma$ is the alphabet labeling transitions,
 $Q$ is the set of the states, some of states are called initial 
($I\subset Q$), some others are called final ($F\subset Q$), and $\delta$ is the transition function which sends each pair of $Q\times \Sigma$ to a subset of $Q$. A word is recognized by an automaton if starting from an initial state,
one can reach  a final state following a path labeled with the sequence of the letters of the word. In order to introduce weight on transitions,
it is easier to consider automata as linear representations \cite{Schutz}. More precisely, an automaton with weight taken in a semiring $\mathbb K$
is a triplet $\mathtt A=(\lambda,\rho,\gamma)$ where $\lambda\in\mathbb K^{1\times N}$, $\gamma\in\mathbb K^{N\times 1}$ and $\rho$ is a morphism of monoid from the free monoid $\Sigma^{*}$ to $\mathbb K^{N\times N}$, each $\rho(a)$, for $a\in\Sigma$, is the adjacency matrix associated to
the weighted subgraph whose edges are labeled with $a$ ($N$ is the number of states and $\rho$ encodes the transitions). The automaton $\mathtt A$ associates to each word $w\in\Sigma$ the scalar $\lambda\rho(w)\gamma\in\mathbb K$. This number
is also the sum of all the value obtained by reading the paths labeled with $w$.  Weighted tree automata \cite{FVW} are another example of finite state machines 
but instead of
produce a scalar from a word, a tree automaton associates a scalar to each tree. There exist two kind of tree automata: bottom-up and top-down depending on the trees being read from the leaves to the root or the root to the leaves.
In both cases, the computation has in commons with that of classical automaton the fact that it is based on both  (multi)linear representations and an appropriate notion of paths.

However, the history of the theory of finite automata begins long before the work of postwar pioneers. 
Indeed, one of the important sources of inspiration was the theory of electronic circuits as it was studied from the nineteenth century.
Graphically, a finite automaton is very closed to an electronic circuit: the states of the automaton play the same role as 
the nodes of a circuit, the transitions the same role as electric components and the letter labeling the transition
 refer to the type of the component connecting the two nodes. 
 To continue the analogy, the letters of an alphabet symbolize  components with only one input and one output. The letters in a tree automaton
 can be considered as one input/ many output or many input/ one output 
 components depending on the automata being  bottom-up or top-down.\\
 The goal of this section is to show that multilinear representations of 
 PROs underly in the definition of many kind of automata.
 \subsection{Word automata\label{word}}
 Suppose that $\mathcal K$ is a commutative semiring. 
 Recall that a word automata is a triplet $\mathtt 
 A:=(\lambda,\rho,\gamma)$ 
where $\lambda\in \mathbb K^{1\times N}$, $\gamma\in \mathbb K^{N\times 
1}$, and $\rho$ is a morphism from a free monoid $A^{*}$ to
$\mathbb K^{N\times N}$.  The behavior of $\mathtt A$ is the series 
$B_{\mathtt A}=\sum_{w\in A^{*}}\lambda \rho(w)F\gamma w$.  Since $\mathbb K$ is commutative, this kind of 
automaton is easily mimicked through a multilinear representation of 
PRO. First we consider the free PROs $\mathcal F(\mathcal X)$ generated 
by $\mathcal X=\mathcal X_{1,0}\cup \mathcal X_{1,1}\cup \mathcal 
X_{0,1}$   with
$\mathcal X_{1,0}=\{\tikdanseq{
\Apro 10{0.25}{0.25}10{$\bot$}{red}{bot}}\}$, $\mathcal X_{1,1}=\{\tikdanseq{
\Apro 00{0.25}{0.25}10{$a$}{green}{bot}}:a\in A\}$ and $\mathcal X_{0,1}=\{\tikdanseq{
\Apro 10{0.25}{0.25}10{$\top$}{red}{bot}}\}$.  We consider also 
representation of PROs $\mu:\mathcal F(\mathcal X)\rightarrow \mathcal 
K(N)$ defined by 
$\mu(\tikdanseq{
\Apro 10{0.25}{0.25}10{$\bot$}{red}{bot}})=\lambda$, $\mu(\tikdanseq{
\Apro 10{0.25}{0.25}10{$\top$}{red}{bot}})=\gamma$ and $\mu(\tikdanseq{
\Apro 10{0.25}{0.25}10{$a$}{green}{bot}})=\rho(a).$
We assimilate a word to an element of $\mathcal F(\mathcal X)$ through 
the correspondence

 \begin{equation}\mathrm{circuit}(a_{1}\dots a_{n})= \begin{array}{c}
	\tikdanseq{
\Apro 10{0.25}{0.25}10{$\top$}{red}{bot}}\\\updownarrow\\\tikdanseq{
\Apro 10{0.25}{0.25}11{$a_{n}$}{green}{bot}}\\\updownarrow\\\vdots\\\updownarrow\\
\tikdanseq{
\Apro 10{0.25}{0.25}11{$a_{1}$}{green}{bot}}\\\updownarrow\\\tikdanseq{
\Apro 10{0.25}{0.25}10{$\bot$}{red}{top}}
	\end{array}. \end{equation}
With such a notation, the behavior becomes
\begin{equation}
	B_{\mathtt A}=\sum_{w\in A^{*}}\mu(\mathrm{circuit}(w))w.
\end{equation}
A formal series is \emph{recognizable} if it is the behavior of an automaton.
It is well known that if two series $S=\sum_{w}\alpha_{w}w$ and 
$S'=\sum_{w}\alpha'_{w}w$ are recognizable then their sum 
$S+S'=\sum_{w}(\alpha_{w}+\alpha_{w}')w$ and Hadamard product $S\odot 
S'=\sum_{w}(\alpha_{w}\alpha_{w}')w$ are also recognizable. This can  be
seen as a consequence of the structure of PRO.
 Indeed, since both series are recognizable, one associates a representation 
 $\mu$ with $S$ and a representation $\mu'$ to $S'$. From Corollary 
 \ref{HadamardRep}, one has
 \begin{equation}
	 S\odot 
S'=\sum_{w}(\mu(\mathrm{circuit}(w))\mu'(\mathrm{circuit}(w))w=
\sum_{w}(\mu\hat\odot\mu'(\mathrm{circuit}(w))w.
 \end{equation}
 Applying  Theorems \ref{GenIdoplus} and \ref{sumRep}, one obtains
 \begin{equation}
	 \mu\hat\oplus\mu'\left(\begin{array}{c}
	\tikdanseq{
\Apro 10{0.25}{0.25}10{$\top$}{red}{bot}}\\\updownarrow\\\tikdanseq{
\Apro 10{0.25}{0.25}11{$a_{n}$}{green}{bot}}\\\updownarrow\\\vdots\\\updownarrow\\
\tikdanseq{
\Apro 10{0.25}{0.25}11{$a_{1}$}{green}{bot}}\\\updownarrow\\\tikdanseq{
\Apro 10{0.25}{0.25}10{$\bot$}{red}{top}}
	\end{array}\right)={}
	\begin{array}{c}
	\mu\hat\oplus\mu'(\tikdanseq{
\Apro 10{0.25}{0.25}10{$\top$}{red}{bot}})\\\updownarrow\\
\mu\hat\oplus\mu'\left({}
\begin{array}{c}\tikdanseq{
\Apro 10{0.25}{0.25}11{$a_{n}$}{green}{bot}}\\
\updownarrow\\
\vdots\\
\updownarrow\\
\tikdanseq{\Apro 10{0.25}{0.25}11{$a_{1}$}{green}{bot}}
\end{array}\right){}
\\\updownarrow\\\mu\hat\oplus\mu'(\tikdanseq{
\Apro 10{0.25}{0.25}10{$\bot$}{red}{top}})
	\end{array}=
 \begin{array}{c}
	\mu(\tikdanseq{
\Apro 10{0.25}{0.25}10{$\top$}{red}{bot}})\hat\oplus\mu'(\tikdanseq{
\Apro 10{0.25}{0.25}10{$\top$}{red}{bot}})\\\updownarrow\\
\mu\left({}
\begin{array}{c}\tikdanseq{
\Apro 10{0.25}{0.25}11{$a_{n}$}{green}{bot}}\\
\updownarrow\\
\vdots\\
\updownarrow\\
\tikdanseq{\Apro 10{0.25}{0.25}11{$a_{1}$}{green}{bot}}
\end{array}\right)\hat\oplus\mu'\left({}
\begin{array}{c}\tikdanseq{
\Apro 10{0.25}{0.25}11{$a_{n}$}{green}{bot}}\\
\updownarrow\\
\vdots\\
\updownarrow\\
\tikdanseq{\Apro 10{0.25}{0.25}11{$a_{1}$}{green}{bot}}
\end{array}\right){}
\\\updownarrow\\\mu(\tikdanseq{
\Apro 10{0.25}{0.25}10{$\bot$}{red}{top}})\hat\oplus\mu'(\tikdanseq{
\Apro 10{0.25}{0.25}10{$\bot$}{red}{top}})
	\end{array}=
		 \mu\left(\begin{array}{c}
	\tikdanseq{
\Apro 10{0.25}{0.25}10{$\top$}{red}{bot}}\\\updownarrow\\\tikdanseq{
\Apro 10{0.25}{0.25}11{$a_{n}$}{green}{bot}}\\\updownarrow\\\vdots\\\updownarrow\\
\tikdanseq{
\Apro 10{0.25}{0.25}11{$a_{1}$}{green}{bot}}\\\updownarrow\\\tikdanseq{
\Apro 10{0.25}{0.25}10{$\bot$}{red}{top}}
	\end{array}\right)+	 \mu'\left(\begin{array}{c}
	\tikdanseq{
\Apro 10{0.25}{0.25}10{$\top$}{red}{bot}}\\\updownarrow\\\tikdanseq{
\Apro 10{0.25}{0.25}11{$a_{n}$}{green}{bot}}\\\updownarrow\\\vdots\\\updownarrow\\
\tikdanseq{
\Apro 10{0.25}{0.25}11{$a_{1}$}{green}{bot}}\\\updownarrow\\\tikdanseq{
\Apro 10{0.25}{0.25}10{$\bot$}{red}{top}}
	\end{array}\right).
 \end{equation}
 Hence,
 \begin{equation}
	 S+S'=\sum_{w}\mu\hat\oplus\mu'(\mathrm{circuit}(w))w,
 \end{equation}
 as expected.
 \subsection{Tree automata\label{tree}}
 For the sake of simplicity we suppose that $\mathcal K=\mathbb B$ the 
 boolean semiring but all the theory is transposable for  any other 
 commutative semiring.
 Recall that there exist two kind of tree automata: Bottom-Up and Top-Down. The two constructions being symmetrical, 
we consider here only the Bottom-Up automata.
A Bottom-Up automaton (see \emph{eg} \cite{Drewes}) is a tuple $A=(Q,\Sigma,\delta,F)$ where
\begin{itemize}
	\item $Q$ is a finite set of states (without loss of generality we assume $Q=\{1,\dots,N\}$ for some integer $N$,
	\item $\Sigma=\cup_{k}\Sigma^{(k)}$ is the ranked input alphabet 
	\item $\delta:\Sigma(Q)\rightarrow 2^Q$ is the transition function, where $\Sigma(Q)=\left\{\tikdanseq{
\node(a) at (0,0){$a$};
\node(n1) at (-0.5,0.75){$q_{1}$};
\node(n2) at (-0.2,0.75){$q_{2}$};
\node(n3) at (0.2,0.75){$\dots$};
\node(n4) at (0.5,0.75){$q_{k}$};
\draw (a)--(n1);
\draw (a)--(n2);
\draw (a)--(n4);
}
:k\in\N, a\in\Sigma^{(k)}, q_{1},\dots,q_{k}\in Q\right\}$.
	\item $F\subseteq Q$ is the set of final states.
\end{itemize}
For a tree $t=\tikdanseq{
\node(a) at (0,0){$a$};
\node(n1) at (-0.5,0.75){$t_{1}$};
\node(n2) at (-0.2,0.75){$t_{2}$};
\node(n3) at (0.2,0.75){$\dots$};
\node(n4) at (0.5,0.75){$t_{k}$};
\draw (a)--(n1);
\draw (a)--(n2);
\draw (a)--(n4);
}
$, we define \[\delta^{*}(t)=\bigcup_{q_{1}\in \delta^{*}(t_{1}),\dots,q_{k}\in \delta^{*}(t_{k})}\delta\left(\tikdanseq{
\node(a) at (0,0){$a$};
\node(n1) at (-0.5,0.75){$q_{1}$};
\node(n2) at (-0.2,0.75){$q_{2}$};
\node(n3) at (0.2,0.75){$\dots$};
\node(n4) at (0.5,0.75){$q_{k}$};
\draw (a)--(n1);
\draw (a)--(n2);
\draw (a)--(n4);
}\right).\]
The language accepted by $A$ is defined by $L(\mathtt A):=\{t: 
\delta^{*}(t)\cap F\neq\emptyset\}$. Notice that a set is nothing but a 
formal series with multiplicities in $\mathbb B$. So,
\begin{equation}
	L(\mathtt A)=\sum_{t: 
\delta^{*}(t)\cap F\neq\emptyset} t.
\end{equation}
We construct a free PRO together with a multilinear representation 
mimicking the behavior of the automaton. We consider the bigraded set 
$\mathcal X=\mathcal X_{1,0}\cup\mathcal X_{0,1}\cup\bigcup_{m,1}$ with 
 $\mathcal X_{k,1}=\{\tikdanseq{
\Apro 00{0.25}{0.25}10{$a$}{green}{a}}:a\in\Sigma^{(a)}\}$, $\mathcal 
X_{1,0}=\{\tikdanseq{
\Apro 10{0.25}{0.25}10{$\bot$}{red}{bot}}\}$, and $\mathcal X_{0,1}=\{\tikdanseq{
\Apro 10{0.25}{0.25}10{$\top$}{red}{bot}}\}$.
To each tree we associate a an element of the free Pro $\mathcal 
F(\mathcal X)$ by setting 
\begin{equation}\mathrm{circuit}(t):=\begin{array}{c}\overbrace{\tikdanseq{
\Apro 
10{0.25}{0.25}10{$\top$}{red}{bot}}\leftrightarrow\cdots\leftrightarrow{}
\tikdanseq{
\Apro 
10{0.25}{0.25}10{$\top$}{red}{bot}}}^{\#leaves(t)\mbox{ 
\footnotesize{times}}}\\\updownarrow\\
\tilde t\\\updownarrow\\
\tikdanseq{
\Apro 10{0.25}{0.25}10{$\bot$}{red}{bot}}
\end{array}\end{equation}
 where $leaves(t)$ denotes the set of the leaves of $t$ and $\tilde t$ 
 is the element of $\mathcal F(\mathcal X)$ obtained by substituting each occurrence of letters 
in $\Sigma^{(k)}$ by its corresponding symbol in $\mathcal X_{k,1}$.
 For instance,
\begin{equation}
	\mathrm{circuit}\left(
\tikdanseq{
\node(a) at (0,0){$a$};
\node(b) at (-0.5,0.75){$b$};
\node(c) at (0,0.75){$c$};
\node(vide) at (0.5,0.75){};
\draw (a)--(b);
\draw (a)--(c);
\draw (a)--(vide);
\node(vide2) at (-0.7,1.5){};
\node(vide3) at (-0.3,1.5){};
\node(vide4) at (0,1.5){};
\draw (b)--(vide2);
\draw (b)--(vide3);
\draw (c)--(vide4);
}\right)
=
\begin{array}{c}\\
\tikdanseq{
\Apro 10{0.25}{0.25}10{$\top$}{red}{bot}}\leftrightarrow
\tikdanseq{
\Apro 10{0.25}{0.25}10{$\top$}{red}{bot}}\leftrightarrow
\tikdanseq{
\Apro 10{0.25}{0.25}10{$\top$}{red}{bot}}\leftrightarrow
\tikdanseq{
\Apro 10{0.25}{0.25}10{$\top$}{red}{bot}}\\
\updownarrow\\
\tikdanseq{
\Apro {-1}{0.75}{0.25}{0.25}11{$b$}{green}{b}}\leftrightarrow 
\tikdanseq{\Apro 
{0}{0.75}{0.25}{0.25}11{$c$}{green}{c}}
\\\updownarrow\\
\tikdanseq{
\Apro 00{0.25}{0.25}11{$a$}{green}{a}}
\\\updownarrow\\
\tikdanseq{
\Apro 10{0.25}{0.25}10{$\bot$}{red}{bot}}
\end{array}
\end{equation}
Now we define the representation $\mu$ satisfying
\begin{itemize}
	\item $\displaystyle\mathop{\mu(\tikdanseq{
\Apro 10{0.25}{0.25}10{$\bot$}{red}{bot}})}^q:=1$, 
\item $\displaystyle\mathop{\mu(\tikdanseq{
\Apro 10{0.25}{0.25}10{$\top$}{red}{bot}})}_{q}=1$ if  $q\in F$ and $0$ 
otherwise,
 \item for each $a\in\Sigma^{(k)}$, we define  $\mu(\tikdanseq{
\Apro 00{0.25}{0.25}00{$a$}{green}{a}})$ such that $\begin{array}{c}q_{1},\dots,q_{k}\\\mu(\tikdanseq{
\Apro 00{0.25}{0.25}00{$a$}{green}{a}})\\q\end{array}=1$ if $q\in\delta\left(\tikdanseq{
\node(a) at (0,0){$a$};
\node(n1) at (-0.5,0.75){$q_{1}$};
\node(n2) at (-0.2,0.75){$q_{2}$};
\node(n3) at (0.2,0.75){$\dots$};
\node(n4) at (0.5,0.75){$q_{k}$};
\draw (a)--(n1);
\draw (a)--(n2);
\draw (a)--(n4);
}\right)$ and $0$ otherwise.
\end{itemize}
With this notation one has
\begin{equation}
L(\mathtt A)=\sum_{t}\mu(\mathrm{circuit}(t))t.
\end{equation}
In terms of series, the union of two languages is translated as the 
sum of the series while the intersection is the Hadamard product.
As in the case of the words, if $L$ and $L'$ are two languages 
recognized by automata then $L\cup L'$ and $L\cap L'$ are also 
recognized. This also can be seen as consequences of Theorems \ref{sumRep}
and \ref{HadamardRep}.
 \subsection{Branching automata\label{branching}}
 Branching automata are a generalisation of usual (Kleene) automata introduced by Lodaya and Weil \cite{LW1,LW2} in the aim to take into account
both sequentiality and parallelism. Notice that, this kind of automata have been recently connected to a logic named Presburger-MSO \cite{Bedon}, as expressive as branching automata.
Recall first that a branching automata is a tuple $\mathtt A=(Q,\Sigma,E,I,F)$ where  $\Sigma$ is an alphabet, $Q$ is the set of the states, 
$I, F\subset Q$ are respectively the set of initial and final states, and the transition $E$ splits into $E=(E_{seq},E_{fork},E_{joint})$ with
\begin{itemize}
	\item $E_{seq}$ are usual transitions.
	\item $E_{fork}\subset Q\times\mathcal M^{>1}(Q)$ and $E_{joint}\subset \mathcal M^{>1}(Q)\times Q$ where $\mathcal M^{>1}(Q)$ denotes the set
	of multi-sets of $Q$ with at least two elements.
\end{itemize}
We construct a representation of a free PRO in $\mathbb B(N)$.
We consider the multiset $\mathcal X=\mathcal X_{1,0}\cup\mathcal 
X_{0,1}\cup\mathcal X_{1,1}\cup\bigcup_{n\geq 2}\mathcal 
X_{1,n}\cup\bigcup_{m\geq 2}\mathcal X_{m,1}$ with $\mathcal 
X_{1,0}=\{\tikdanseq{\Apro 10{0.25}{0.25}10{$\top$}{red}{bot}}\}$,
$\mathcal X_{0,1}=\{\tikdanseq{\Apro 
01{0.25}{0.25}10{$\bot$}{red}{bot}}\}$,
$\mathcal X_{1,1}:= \{\tikdanseq{\Apro 
11{0.25}{0.25}10{$a$}{green}{bot}}:a\in\Sigma\}$,
$\mathcal X_{n,1}:= \{\tikdanseq{\Apro 11{0.25}{0.25}10{$\overline n$}{cyan}{bot}}\}$, and
	$\mathcal X_{1,m}:= \{\tikdanseq{\Apro 11{0.25}{0.25}10{$\underline 
	m$}{cyan}{bot}}\}$.
For $\mathfrak x\in\mathcal X_{1,0}\cup\mathcal X_{0,1}\mathcal 
X_{1,1}$, $\mu(\mathfrak x)$ is defined as in the case of word 
automata. For each $m, n\geq 2$ , the hypermatrices
$ \mu(\tikdanseq{\Apro 11{0.25}{0.25}10{$\overline n$}{cyan}{bot}})$ 
and
$ \mu(\tikdanseq{\Apro 11{0.25}{0.25}10{$\underline m$}{cyan}{bot}})$ 
are symmetric in the sense that they satisfy
\begin{equation}\begin{array}{c} i_{1},\dots,i_{n}\\ \mu(\tikdanseq{\Apro 11{0.25}{0.25}10{$\overline n$}{cyan}{bot}})\\i\end{array}=
	\begin{array}{c} i_{\sigma(1)},\dots,i_{\sigma(n)}\\ 
	\mu(\tikdanseq{\Apro 11{0.25}{0.25}10{$\overline 
	n$}{cyan}{bot}})\\i\end{array}\mbox{ and }\begin{array}{c} i\\ \mu(\tikdanseq{\Apro 
	 11{0.25}{0.25}10{$\underline m$}{cyan}{bot}})\\i_{1},\dots,i_{m}\end{array}=
	\begin{array}{c} i\\ \mu(\tikdanseq{\Apro 11{0.25}{0.25}10{$\underline 
	m$}{cyan}{bot}})\\i_{\tau(1)},\dots,i_{\tau(m)}\end{array}\end{equation}
for each permutations $\sigma\in\mathfrak S_{n}$ and $\tau\in\mathfrak 
S_{m}$. Furthermore,
$\begin{array}{c} i_{1},\dots,i_{n}\\ \mu(\tikdanseq{\Apro 
11{0.25}{0.25}10{$\overline n$}{cyan}{bot}})\\i\end{array}=1$ if 
$(i,\{i_{1},\dots,i_{n}\})\in E_{fork}$ and $0$ otherwise, and 
$\begin{array}{c} i\\ \mu(\tikdanseq{\Apro 
	 11{0.25}{0.25}10{$\underline 
	 m$}{cyan}{bot}})\\i_{1},\dots,i_{m}\end{array}=1$ if 
	 $(\{i_{1},\dots,i_{m}\},i)\in E_{joint}$ and $0$ otherwise.
For this kind of automata, the recognized language is assimilated to 
the formal series
\begin{equation}
	L^{//}(\mathtt A)=\sum_{\mathfrak c\in\mathcal F(\mathcal 
	X)_{1,1}}\mu\left(\begin{array}{c}\tikdanseq{\Apro 
	10{0.25}{0.25}10{$\top$}{red}{bot}}\\\updownarrow\\\mathfrak 
	c\\\updownarrow\\
\tikdanseq{\Apro 10{0.25}{0.25}10{$\bot$}{red}{bot}} 
\end{array}\right)\mathfrak c.
\end{equation}
Also notice that the notion of path in this  model of automata 
coincides with the notion of path in the hypergraph associated to a 
representation. Closure under finite union (resp. 
finite intersection) of languages recognized by branching automata can 
be seen as a consequence of Theorem \ref{sumRep} 
(resp. Corollary \ref{HadamardRep}).
\begin{example}\rm
	Consider  the branching automaton 
	$$\mathtt A=([1,6],{a,b},(\{(2,a,4),(3,b,5)\},\{(1,\{1,1\}),(1,\{2,3\})\},\{(\{6,6\},6),(\{4,5\},6)\}),\{1\},\{1,6\}).$$
	The hypergraph of the associated representation
	is drawn in figure \ref{ExBran}. After removing the symbols $\tikdanseq{\Apro 10{0.25}{0.25}10{$\top$}{red}{bot}}$
,
$\tikdanseq{\Apro 10{0.25}{0.25}10{$\bot$}{red}{bot}}$, 
$\tikdanseq{\Apro {-1.5}1{0.25}{0.25}21{$\overline 2$}{cyan}{h1}}$, and 
$\tikdanseq{\Apro {-1.5}1{0.25}{0.25}21{$\underline 2$}{cyan}{h1}}$ in the  language $L^{//}$ we obtain exactly all the elements under the forms
$\tikdanseq{\Apro {-1.5}1{0.4}{0.4}22{$a_{1}$}{green}{h1}}\leftrightarrow\cdots\leftrightarrow{}
\tikdanseq{\Apro {-1.5}1{0.4}{0.4}22{$a_{2n}$}{green}{h1}}$ 
for $n\in\N$ and $a_{i}\in\{a,b\}$ such 
that $\mathrm{card}\{i:a_{i}=a\}=\mathrm{card}\{i:a_{i}=b\}$.
	
	\begin{figure}[h]
\begin{center}
\begin{tikzpicture}
\node[circle,fill=gray](v1) at (0,0){1};
\node[circle,fill=gray](v2) at (1,2){2};
\node[circle,fill=gray](v3) at (-1,2){3};
\node[circle,fill=gray](v4) at (1,4){4};
\node[circle,fill=gray](v5) at (-1,4){5};
\node[circle,fill=gray](v6) at (0,6){6};

\Apro 13{0.25}{0.25}11a{green}{a}
\draw[->,green,thick](uua1) to[out=90,in=270] (v4);
\draw[->,green,thick](v2)  to[out=90,in=260] (dda1);

\Apro {-1}3{0.25}{0.25}11b{green}{b}
\draw[->,green,thick](uub1) to[out=90,in=270] (v5);
\draw[->,green,thick](v3)  to[out=90,in=260] (ddb1);

\Apro {-3}{0.5}{0.25}{0.25}01{$\top$}{red}{t1}
\draw[->,red,thick](v1) to[out=180,in=270] (ddt11);

\Apro {3}{-0.5}{0.25}{0.25}10{$\bot$}{red}{b1}
\draw[->,red,thick] (uub11) to [out=90,in=0] (v1);

\Apro {-3}{6.5}{0.25}{0.25}01{$\top$}{red}{t2}
\draw[->,red,thick](v6) to[out=180,in=270] (ddt21);

\Apro {-1.5}1{0.25}{0.25}21{$\overline 2$}{cyan}{h1}
\Apro {0}1{0.25}{0.25}21{$\overline 2$}{cyan}{h2}
\Apro {1.5}1{0.25}{0.25}21{$\overline 2$}{cyan}{h3}
\draw[->,cyan,thick](v1) to[out=90,in=270] (ddh11);
\draw[->,cyan,thick](v1) to[out=90,in=270] (ddh21);
\draw[->,cyan,thick](v1) to[out=90,in=270] (ddh31);
\draw[->,cyan,thick](uuh21) to[out=120,in=120](v1) ;
\draw[->,cyan,thick](uuh22) to[out=60,in=60](v1) ;
\draw[->,cyan,thick](uuh11) to[out=120,in=240](v3);
\draw[->,cyan,thick](uuh12) to[out=60,in=240](v2);
\draw[->,cyan,thick](uuh32) to[out=60,in=300](v3);
\draw[->,cyan,thick](uuh31) to[out=120,in=300](v2);

\draw[->,cyan,thick](uuh22) to[out=60,in=60](v1) ;

\Apro {0}7{0.25}{0.25}21{$\overline 2$}{cyan}{h4}
\Apro {-1.5}5{0.25}{0.25}12{$\underline 2$}{cyan}{d2}
\Apro {1.5}5{0.25}{0.25}12{$\underline 2$}{cyan}{d3}
\draw[->,cyan,thick](v6) to[out=90,in=270] (ddh41);
\draw[->,cyan,thick](uuh41) to[out=120,in=120](v6) ;
\draw[->,cyan,thick](uuh42) to[out=60,in=60](v6) ;

\draw[->,cyan,thick](v5) to[out=90,in=240] (ddd21);
\draw[->,cyan,thick](v5) to[out=90,in=300] (ddd32);
\draw[->,cyan,thick](v4) to[out=90,in=300] (ddd21);
\draw[->,cyan,thick](v4) to[out=90,in=240] (ddd31);

\draw[->,cyan,thick](uud21) to[out=90,in=240](v6) ;
\draw[->,cyan,thick](uud31) to[out=90,in=300](v6) ;

\end{tikzpicture}

\end{center}
\caption{An example of branching automaton seen as the hypergraph of a 
representation.\label{ExBran}}
\end{figure}
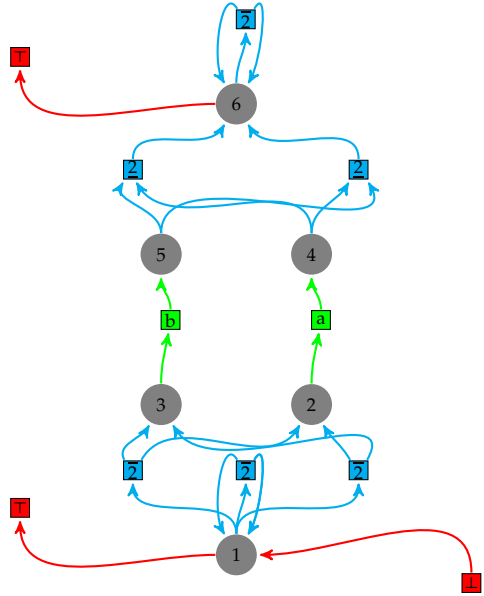
\end{example}
\subsection{Heindel Pro automata\label{Heindel}}
Recently,  Heindel \cite{Heindel} proposed a definition for PRO-automata. In his paper, he used a more general
definition for PROs (category in Mon whose object monoid is free) which contains PROs and colored PROs as special cases.
He used also a definition of free PROs (called \emph{placids}) which matches to free colored PROs in our notations.
To be more precise, he defined the notion of \emph{signature} associated with a list of colors $C$, as a triplet $(\Sigma,s,t)$ where $\Sigma$ is a finite set and $s:\Sigma\rightarrow C^{*}$ and $t:\Sigma\rightarrow C^{*}$ are two maps which associate to each symbol $\sigma\in\Sigma$ a sequence of inputs $s(\sigma)\in C^{*}$ and a sequence of outputs $t(\sigma)\in C^{*}$. A signature is seen as a graph whose vertices belong to a subset of $C^{*}$ and the arrows are in $\Sigma$. A \emph{generalised} PRO $\mathcal D$ is seen as  graph $(D_{0},D_{1})$ where $D_{0}$ is the set of the objects and $D_{1}$ the set of the arrows. The set $D_{0}$ endowed with the tensor product is a monoid and to each arrow $a$ in $D_{1}$ we associate its domain (inputs) $\mathtt{dom}(a)\in D_{0}$ and its codomain (outputs) 
$\mathtt{codom}(a)\in D_{0}$. The definition of a \emph{signature embedding}  generalizes the notion of alphabet in the senses that it is a morphism of graphs from a signature $(\Sigma,s,t)$ to a generalized PRO $(D_{0},D_{1})$ together with a morphism of monoid $C^{*}\rightarrow D_{0}$.\\
Let us recall the definition of a placid as  a free (generalized) PRO \cite{Heindel}. If $C\neq\emptyset$ be a set of colors and $(\Sigma,s,t)$ be a signature. There exists a placid $\mathcal  P(\Sigma)$ into which the signature $\Sigma$ embeds via the inclusion $\Sigma\subset \mathcal P(\Sigma)$ such that for each generalized PRO $\mathcal D$ and embedding $\iota:\Sigma\rightarrow \mathcal D$, there exists a unique (generalized) PRO-morphism $F_{\iota}:\mathcal P(\Sigma)\rightarrow \mathcal D$ that restrict to $\iota$. From this definition, we notice that a  placid is nothing but a Path PRO (see section \ref{Path}).\\
Now we have all the ingredients to recall the definition of PRO automata by \cite{Heindel}.
\begin{definition}(Heindel \cite{Heindel})
	A PRO automaton  over the signature $C$-colored signature $(\Sigma,s,t)$ is a tuple $\mathcal A=(Q,\Gamma,\delta,I,F)$ where $Q$ is a finite set of states, $(\Gamma,s',t')$ a signature over the set of colors $Q$,  $I$ and $F$ are two subset of words in $Q^{*}$, and $\delta$ is a morphism of signatures  (that is a morphism of graph induced by a morphism of monoid $Q^{*}\rightarrow C^{*}$).
\end{definition} 
For our purpose we assume three restrictions:
\begin{itemize}
	\item first $\#C=1$ this means that the placid $\mathcal P(\Sigma)$ is nothing but the free PRO $\mathcal F(\Sigma)$, \item all the elements of $\Sigma$ has at least one input and at least one output,
	\item the map $s'\times \delta_{1}\times t':\Gamma\rightarrow Q^{*}\times \Sigma\times Q^{*}$ is injective, where $\delta_{1}$ denotes the component of the map $\delta$ sending the arrows of the graph $\Gamma$ to the arrows of $\Sigma$, that is the PRO automaton is \emph{normal} in the sense of \cite{Heindel}.
	\end{itemize}
\begin{definition}(Heindel \cite{Heindel})
A \emph{run} of an automaton $\mathcal A=(A,\Gamma,\delta,I,F)$ is an arrow in $r:q\rightarrow p$ in $\mathcal P(\Gamma)$ and it is \emph{accepting} if its $q\in I$ and $q'\in F$.
The \emph{language} of $\mathcal A$ is the set of the image in $\mathcal F(\Sigma)$ of the accepting runs. With our notation we have
\[{}
\mathcal L(\mathcal A)=\left\{u(\delta(a))\mid a\in\mathcal \mathcal F(\Gamma), \mathtt{In}(a)\in I, \mathtt{Out}(a)\in J \right\}
\]
\end{definition}
This kind of automata and their languages were discovered and studied by Bossut in his PHD \cite{Bossut} without the help of the theory of PROs.
\begin{example}\rm
In the aim to illustrate the construction, we consider, as in \cite{Heindel}, the example of the Bossut's brick wall \cite{Bossut}.\\
Let $Q=\{1,2,3\}$,
$
\Gamma=\left\{\tikdanseq{\chip 00{0.75}{0.75}11{}{\ }
\node(n1) at (0.4,-0.1) {\tiny 1};
\node(n1) at (0.35,1.15) {\tiny 2};},
\tikdanseq{\chip 00{0.75}{0.75}11{}{\ }\node(n1) at (0.4,-0.1) {\tiny 1};
\node(n1) at (0.35,1.15) {\tiny 3};},
\tikdanseq{\chip 00{0.75}{0.75}22{}{\ }
\node(n1) at (0.4,-0.1) {\tiny 1\ \ \  \ 2};
\node(n1) at (0.35,1.15) {\tiny 2\ \ \ \ 3};},
\tikdanseq{\chip 00{0.75}{0.75}22{}{\ }
\node(n1) at (0.4,-0.1) {\tiny 3\ \ \  \ 1};
\node(n1) at (0.35,1.15) {\tiny 2\ \ \ \ 3};},
\tikdanseq{\chip 00{0.75}{0.75}22{}{\ }
\node(n1) at (0.4,-0.1) {\tiny 3\ \ \  \ 2};
\node(n1) at (0.35,1.15) {\tiny 2\ \ \ \ 3};}
\right\}
$, and $I=F=1\{23\}^{*}1\cup\{23\}^{+}$.\\
We have
\[{}
\mathcal L(\mathcal A)=\left\{\begin{array}{c}L_{\epsilon}(m)\\
\updownarrow\\\left(\begin{array}{c}L_{1-\epsilon}(m)\\\updownarrow\\L_{\epsilon}(m)\end{array}\right)^{\updownarrow n}\\\end{array}\mid \epsilon\in\{0,1\}, m\geq 1,n\geq 0\right\}\cup{}
\left\{\left(\begin{array}{c}L_{1-\epsilon}(m)\\\updownarrow\\L_{\epsilon}(m)\end{array}\right)^{\updownarrow n}\mid \epsilon\in\{0,1\}, m\geq 1,n\geq 0\right\}
\]
where the $L_{\epsilon}(m)$ are defined recursively by
\[{}
L_{0}(1)=\tikdanseq{\chip 00{0.75}{0.75}11{}{\ }},\ L_{0}(2m)=L_{0}(2m-1)\leftrightarrow \tikdanseq{\chip 00{0.75}{0.75}11{}{\ }},\ 
L_{0}(2m+1)=L_{0}(2m-1)\leftrightarrow \tikdanseq{\chip 00{0.75}{0.75}22{}{\ }},
\]
and
\[{}
L_{1}(2)=\tikdanseq{\chip 00{0.75}{0.75}22{}{\ }},\ L_{1}(2m+1)=L_{1}(2m)\leftrightarrow \tikdanseq{\chip 00{0.75}{0.75}11{}{\ }},\ 
L_{1}(2(m+1))=L_{0}(2m))\leftrightarrow \tikdanseq{\chip 00{0.75}{0.75}22{}{\ }}.
\]
For instance,  the element $\begin{array}{c} L_{1}(8)\\\updownarrow\\ L_{0}(8)\\\updownarrow\\ L_{1}(8)\\\updownarrow\\ L_{0}(8)\end{array}$ corresponds graphically to the following wall:
\begin{center}
\begin{tikzpicture}
	\Apro 00{0.25}{0.25}11{\ }{black}{a11}
	\Apro {1}{0}{0.5}{0.25}22{\ }{black}{a12}
	\Apro {2}{0}{0.5}{0.25}22{\ }{black}{a13}
	\Apro {3}{0}{0.5}{0.25}22{\ }{black}{a14}
	\Apro 40{0.25}{0.25}11{\ }{black}{a15}
	\Apro {0.5}{0.5}{0.5}{0.25}22{\ }{black}{a21}
	\Apro {1.5}{0.5}{0.5}{0.25}22{\ }{black}{a22}
	\Apro {2.5}{0.5}{0.5}{0.25}22{\ }{black}{a23}
	\Apro {3.5}{0.5}{0.5}{0.25}22{\ }{black}{a24}
	
	\Apro 01{0.25}{0.25}11{\ }{black}{a31}
	\Apro {1}{1}{0.5}{0.25}22{\ }{black}{a32}
	\Apro {2}{1}{0.5}{0.25}22{\ }{black}{a33}
	\Apro {3}{1}{0.5}{0.25}22{\ }{black}{a34}
	\Apro 41{0.25}{0.25}11{\ }{black}{a35}
	\Apro {0.5}{1.5}{0.5}{0.25}22{\ }{black}{a41}
	\Apro {1.5}{1.5}{0.5}{0.25}22{\ }{black}{a42}
	\Apro {2.5}{1.5}{0.5}{0.25}22{\ }{black}{a43}
	\Apro {3.5}{1.5}{0.5}{0.25}22{\ }{black}{a44}
	
	\draw[-](uua111)  to  (dda211);
	\draw[-](uua121)  to  (dda212);
	\draw[-](uua122)  to  (dda221);
	\draw[-](uua131)  to  (dda222);
	\draw[-](uua132)  to  (dda231);
	\draw[-](uua141)  to  (dda232);
	\draw[-](uua142)  to  (dda241);
	\draw[-](uua151)  to  (dda242);
	
	\draw[-](uua211)  to  (dda311);
	\draw[-](uua212)  to  (dda321);
	\draw[-](uua221)  to  (dda322);
	\draw[-](uua222)  to  (dda331);
	\draw[-](uua231)  to  (dda332);
	\draw[-](uua232)  to  (dda341);
	\draw[-](uua241)  to  (dda342);
	\draw[-](uua242)  to  (dda351);
	
	\draw[-](uua311)  to  (dda411);
	\draw[-](uua321)  to  (dda412);
	\draw[-](uua322)  to  (dda421);
	\draw[-](uua331)  to  (dda422);
	\draw[-](uua332)  to  (dda431);
	\draw[-](uua341)  to  (dda432);
	\draw[-](uua342)  to  (dda441);
	\draw[-](uua351)  to  (dda442);
	
\end{tikzpicture}\end{center}
\end{example}
Let us show that such a language can be completely described in terms of multilinear representations. First we consider the free PRO generated by the $u(x)$ for $x\in \Gamma$ and we set $\mathcal X:=\{u(x)\mid x\in\mathcal F\Gamma\}$. For any words in $w=w_{0}\dots w_{k}\in\{1,\dots,N\}$, we define two matrices
\[{}
\mathrm{IN}_{w}(N)=\mathbf E(N,0,k;[],[w_{1},\dots,w_{k}])\in\mathbb B(N,0,k)
\]
and
\[{}
\mathrm{OUT}_{w}(N)=\mathbf E(N,k,0;[w_{1},\dots,w_{k}],[])\in\mathbb B(N,k,0).
\]
For any $\mathfrak x\in \mathcal X$, we define the matrix $\mu_{\Gamma}(\mathfrak x)\in \mathbb B(N,m,n)$ by
$
\displaystyle\mathop{\mu_{\Gamma}(\mathfrak x)}^{I}_{J}=1
$ if there exists $x\in \Gamma$ such that $\mathtt{In}(x)=J$, $\mathtt{Out}(x)=I$, and $u(x)=\mathfrak x$ and $0$ otherwise. The map $\mu$ is extended in a representation of $\mathcal F(\mathcal X)$. As a consequence of Proposition \ref{rep2path}, the language of the automaton is recovered by the formula
\begin{equation}\label{HLA}
\mathcal L(\mathcal A)=\bigcup_{m,n}\bigcup_{u\in I\cup Q^{n}\atop
v\in J\cup Q^{m}}\left\{\mathfrak x\in\mathcal F(\mathcal X)\left| \begin{array}{c}\mathrm{OUT}_{u}\\\updownarrow\\\mu_{\Gamma}(\mathfrak x)\\\updownarrow\\\mathrm{IN}_{v}\end{array}=1\right.\right\}.
\end{equation}
Conversely, to each representation of $\mathcal X$, we can associate a unique signature $\Gamma_{\mu}$ such that for each $x\in\Gamma$, $u(x)\in\mathcal X$ and $\hmat{\mu(u(x))}{\mathtt{Out}(x)}{\mathtt{In}(x)}=1$. 
\begin{definition}(Heindel\cite{Heindel})
	A language is \emph{acceptable} if it is the language of a PRO automaton.
\end{definition}
As a consequence of formula \ref{HLA}, we have:
\begin{theorem}
A language $\mathcal L$ is acceptable if and only if there exists a tuple $(Q,\mu,I,J)$ with $Q=\{1,\dots,N\}$, $I, J\in Q^{*}$, and $\mu$ a representation of $\mathcal X$, such that
\begin{equation}
\mathcal L=\bigcup_{m,n}\bigcup_{u\in I\cup Q^{n}\atop
v\in J\cup Q^{m}}\left\{\mathfrak g\in\mathcal F(\mathcal X)\left| \begin{array}{c}\mathrm{OUT}_{u}\\\updownarrow\\\mu(\mathfrak g)\\\updownarrow\\\mathrm{IN}_{v}\end{array}=1\right.\right\}.
\end{equation}
\end{theorem}
\begin{example}\rm
The representation associated to the Bossut's wall is defined by
\[
\mu\left(\tikdanseq{\Apro 11{0.25}{0.25}11{\ }{green}{a}}\right)=\mathbf E(3,1,1;[2],[1])+\mathbf E(3,1,1;[3],[1])\] and \[
\mu\left(\tikdanseq{\Apro 11{0.5}{0.25}22{\ }{red}{a}}\right)=\mathbf E(3,2,2;[2,3],[1,2])+\mathbf E(3,2,2;[2,3],[3,1])+\mathbf E(3,2,2;[2,3],[3,2])
\]	 (see Figure \ref{bossut}).

	\begin{figure}[h]
\begin{center}
\begin{tikzpicture}
\node[circle,fill=gray](v1) at (0,0){1};
\node[circle,fill=gray](v3) at (1,3){3};
\node[circle,fill=gray](v2) at (-1,3){2};

\Apro 11{0.25}{0.25}11{\ }{green}{a}
\Apro {-1}1{0.25}{0.25}11{\ }{green}{b}
\Apro {-3}1{0.5}{0.25}22{\ }{red}{c}
\Apro {2.5}1{0.5}{0.25}22{\ }{red}{d}
\Apro {-0.25}4{0.5}{0.25}22{\ }{red}{e}
\draw[->,green,thick](uub1) to[out=90,in=270] (v2);
\draw[->,green,thick](v1)  to[out=90,in=260] (ddb1);
\draw[->,green,thick](uua1) to[out=90,in=270] (v3);
\draw[->,green,thick](v1)  to[out=90,in=260] (dda1);
\draw[->,red,thick](uuc1) to[out=90,in=180] (v2);
\draw[->,red,thick](uuc2) to[out=90,in=240] (v3);
\draw[->,red,thick](v1)  to[out=90,in=260] (ddc1);
\draw[->,red,thick](v2)  to[out=240,in=260] (ddc2);

\draw[->,red,thick](uud1) to[out=90,in=300] (v2);
\draw[->,red,thick](uud2) to[out=90,in=0] (v3);
\draw[->,red,thick](v3)  to[out=270,in=260] (ddd1);
\draw[->,red,thick](v1)  to[out=90,in=260] (ddd2);

\draw[->,red,thick](uue1) to[out=90,in=90] (v2);
\draw[->,red,thick](uue2) to[out=90,in=90] (v3);
\draw[->,red,thick](v2)  to[out=0,in=260] (dde2);
\draw[->,red,thick](v3)  to[out=180,in=260] (dde1);

\end{tikzpicture}

\end{center}
\caption{Representation associated to Bossut's wall.\label{bossut}}
\end{figure}
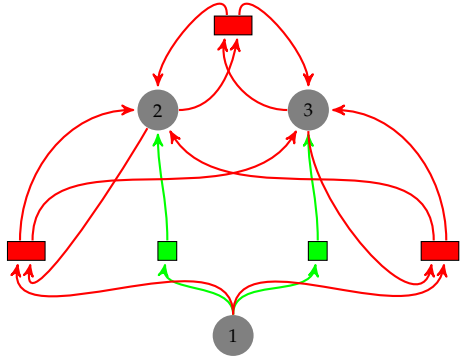
\end{example}
Let $u=u_{1}\dots u_{k}\in \{1,\dots,N\}^{k}$ and  $u'=u'_{1}\dots u'_{k}\in \{1,\dots,N'\}^{k}$, we define
$u\odot v=((u'_{1}-1)M+u_{1})\dots((u'_{k}-1)M+u_{k})$. For instance, if $u=14221\in\{1,\dots,4\}$ and $v=13122\in\{1,\dots,5\}$ then
$u\odot v=2\ 12\ 265$.
\begin{lemma}\label{IodotI'}
	Let $I$ be a regular language over $Q=\{1,\dots,N\}$ and $I'$ 
	a regular language over $Q'=\{1,\dots,N'\}$ then the language $I\odot I':=\{w\odot w'\mid w\in I, w'\in I', |w|=|w'|\}$ is regular over $Q''=\{1,\dots,NN'\}$.
\end{lemma}
\begin{proof}
	Let $A=(Q,S,\{i\},F,\delta)$ and $A'=(Q',S',\{i'\},F',\delta')$ be two complete deterministic automata recognizing respectively 
	$I$ and $I'$. We consider the automaton $A=(Q',S\times S',\{(i,i')\},F\times F',\delta'')$ where
	$\delta''((q,q'),(\beta-1)N+\alpha)=(p,p')$ if and only if $\delta(q,\alpha)=p$ and $\delta'(q',\alpha)=p'$. We remark there exists a path from $(i,i')$ to a final state $(f,f')$ in $Q''$ labeled with $((\beta_{1}-1)N+\alpha_{1})\cdots ((\beta_{k}-1)N+\alpha_{k})=\alpha_{1}\cdots\alpha_{k}\odot \beta_{1}\cdots\beta_{k}$ if and only if there exists a path from $i$ to $f$ in $Q'$ labeled with $\alpha_{1}\cdots\alpha_{k}$ and a path from $i'$ to $f'$ in $Q'$ labeled with $\beta_{1}\cdots\beta_{k}$. We deduce that $I\odot I'$ is regular.
\end{proof}
Furthermore, from the definition, one obtains
\[{}
\mathtt{IN}_{u}(N)\odot\mathtt{IN}_{u'}(N')=\mathtt{IN}_{u\odot u'}(NN') \mbox{ and } 
\mathtt{OUT}_{u}(N)\odot\mathtt{OUT}_{u'}(N')=\mathtt{OUT}_{u\odot u'}(NN').
\]
\begin{theorem} The intersection of two acceptable languages still is acceptable. 
	More precisely, let $\mathcal L$ be the language of $A=(\{1,\dots,N\},\mu,I,J)$ and $\mathcal L'$ be the language of $A'=(\{1,\dots,N'\},\mu',I',J')$ then
	\[\mathcal L\left(\{1,\dots,NN'\},\mu\hat\odot\mu',I\odot I',J\odot J'\right)=\mathcal L\cap\mathcal L'.
	\]
\end{theorem}
\begin{proof}
	Notice first that Lemma \ref{IodotI'} implies that $I\odot I'$ and $J\odot J'$ are regular.\\	We remark also that
	\[{}
	 \begin{array}{c}\mathrm{OUT}_{u\odot u'}\\\updownarrow\\\mu\hat\odot\mu'(\mathfrak x)\\\updownarrow\\\mathrm{IN}_{v\odot v'}\end{array}
=\begin{array}{c}\mathrm{OUT}_{u}\odot\mathrm{OUT}_{u'}\\\updownarrow\\\mu(\mathfrak x)\odot\mu'(\mathfrak x)\\\updownarrow\\\mathrm{IN}_{v}\odot \mathrm{IN}_{v'}\end{array}=\begin{array}{c}\mathrm{OUT}_{u}\\\updownarrow\\\mu(\mathfrak x)\\\updownarrow\\\mathrm{IN}_{v}\end{array}\odot \begin{array}{c}\mathrm{OUT}_{u'}\\\updownarrow\\\mu'(\mathfrak x)\\\updownarrow\\\mathrm{IN}_{v'}\end{array}
=\begin{array}{c}\mathrm{OUT}_{u}\\\updownarrow\\\mu(\mathfrak x)\\\updownarrow\\\mathrm{IN}_{v}\end{array} \begin{array}{c}\mathrm{OUT}_{u'}\\\updownarrow\\\mu'(\mathfrak x)\\\updownarrow\\\mathrm{IN}_{v'}\end{array}.
	\]
	And so 
	\[{}
	\begin{array}{c}\mathrm{OUT}_{u\odot u'}\\\updownarrow\\\mu\hat\odot\mu'(\mathfrak x)\\\updownarrow\\\mathrm{IN}_{v\odot v'}\end{array}=1\mbox{ if and only if }
		\begin{array}{c}\mathrm{OUT}_{u}\\\updownarrow\\\mu(\mathfrak x)\\\updownarrow\\\mathrm{IN}_{v}\end{array}=1\mbox{ and } \begin{array}{c}\mathrm{OUT}_{u'}\\\updownarrow\\\mu'(\mathfrak x)\\\updownarrow\\\mathrm{IN}_{v'}\end{array}=1.
	\]
	This ends the proof.
\end{proof}
Also we have
\begin{theorem}(Bossut et al \cite{BDW})\label{BossutUnion} 
	If $\mathcal L$ and $\mathcal L'$ are acceptable then $\mathcal L\cup\mathcal L'$ is acceptable.
\end{theorem}
This theorem can also be seen as a consequence of Theorem \ref{sumRep}.

\subsection{A universal definition for automata?}
If there were no question mark, this title would be a bit pretentious. Indeed, it does not really make sense to look for a definition of a universal automaton in an absolute way  but rather in relation to a property that we would like to conserve. It is the matrix vision of the automata that we want to keep through  the use of representations and a generalized notion of paths. In this quest for a universal definition, we showed two candidates each with their own strengths and weaknesses. The first one comes from  sections \ref{word}, \ref{tree} and \ref{branching}. In this model, there are no distinction between representations and automata. We are therefore immersed in the heart of the algebraic tools and this gives coherence to this model. For instance, there are no more initial or final states, nor more generally, input and output vectors. All states play a similar role and the automaton produces a number when we give it an element with neither input nor output. Although it is algebraically interesting, this model does not give a general definition of recognizable languages, but a multitude of different theories of recognizability depending on the question we want to ask and on the form of the elements of the PRO. The second candidate is due to Bossut \cite{Bossut,BDW} (while the connexion with PROs was pointed out by Heindel \cite{Heindel}) and was studied in Section \ref{Heindel}. This model may seem a little less elegant than the first one (although it is a matter of point of view) because it forces us to manipulate a triplet instead of only one representation. Indeed, together with the representation, we need two sets $I$ and $J$ to manage the inputs and outputs of the automata. Nevertheless, it is very well adapted to the theory of languages because it allows us to capture a general notion of recognizability (says acceptability in \cite{Heindel}).\\
Notice also that the two models can be extends to manipulate commutative multiplicities. For the first one,  it is straightforward since it suffices to replace  the boolean semiring by the suitable commutative semiring. For the second model, we need besides to replace the regular languages $I$ and $J$ by regular series. Hence, we  define the series of an automaton by
\[{}
S(\{1,\dots,N\},\mu,I,J)=\sum_{m,n}\sum_{\mathfrak x\in\mathcal F(\mathcal X)_{m,n}}\left(\sum_{u\in\{1,\dots,N\}^{m}\atop{}
v\in\{1,\dots,N\}^{n}
} \langle u,I\rangle \langle v,J\rangle\begin{array}{c}
\mathrm{OUT}_{v}\\\updownarrow\\\mu(\mathfrak x)\\\updownarrow\\\mathrm{IN}_{u}
\end{array}\right)\mathfrak x,
\]
where $\langle w,I\rangle$ denotes the coefficient of the word $w$ in the series $I$.\\
\section{Conclusion and perspectives}
Throughout the paper we have investigated several properties of the multilinear representations of  free PROs. In a way, this validates the diagrammatic approaches for multi-linear calculations. As an illustration, we describe in Appendix \ref{QI} how quantum circuits can be seen as representations of a free PRO. The cases of circuits Pros is easier to understand because the combinatorial description is well known. For more general PROs, the combinatorial description is less easy and has been less studied. It seems to be based on isotopic classes of some graphs. However, such representations appear naturally in the literature (see Section \ref{TL} for the example of Temperley-Lieb algebras). Any advances in the understanding of the combinatorial aspect of free Pros should have many applications. For instance, it should be interesting to understand which notion extends the connectivity in the aim to generalize Theorem 
\ref{sumRep}.\\
In the last section, we saw that free PROs procide an ideal framework for defining a general notion of automata embracing most of the classical definitions (word automata, tree automata, branching automata) and based on (multi)-linear algebra. Nevertheless, word automata with non-commutative multiplicities do not fit into that pattern. Indeed, when $\mathbb K$ is not commutative, $\mathbb K(N)$ is not a PRO because the intechange law no longer holds. In the same way, it is not possible to simulate the behavior of transducers with the proposed constructions. All of these remarks suggest that one has to define more general multiplicities having algebraic structure  close than those of PROs, for instance by choosing a convenient structure of enriched PRO.\\
Another direction of research will consist in investigating  generalizations of the Kleene-Schutzenberger theorem for autom for PRO-automata and  characterizing regularseries by introducing  kinds of \emph{rational operations} playing the same role as the catenation and the Kleene star for regular (word) languages. Notice that a Kleene theorem for Bossut automata was proved in \cite{BDW}. In this paper it is proved that a language $\mathcal L$ is accepted by an automaton if and only if it can be described using the symbols together with $5$ \emph{rational} operations: vertical composition, vertical star, horizontal composition, horizontal star, and union. It should be interesting to investigate the hypermatrix counter-parts of these operations and how they can be generalized to series.\\ \\ \\
\noindent{\bf Acknowledgment}
This work was partially supported by  ERDF project MOUSTIC.

\appendix{}
\section{ModPro and enriched categories \label{AppModPro}}
Let $(\mathcal M,\otimes,\alpha,\lambda,\rho,I)$ be a monoidal category. 
A $\mathcal M$-category $\mathtt C$ (also called enriched category) consists of
\begin{itemize}
\item the class $Obj(\mathtt C)$ of the objects of $\mathtt C$,
\item an object $\mathtt C(\aa,\bb)$ of $\mathcal M$ for any pair of objects $(\aa,\bb)$ of $\mathtt C$,
\item for each object $\aa$ in $\mathtt C$, the identity $Id_{\aa}$ is an arrow $I\rightarrow \mathtt C(\aa,\aa)$ in $\mathcal M$,
 where $I$ denotes the identity object in $\mathcal M$,
\item the composition of the arrows in $\mathtt C$ is encoded by the tensor product in $\mathrm M$. More precisely, we have for
each triple $(\aa,\bb,\cc)$ of objects in $\mathtt C$ a composition
 ${\pmb \circ}_{\aa\bb\cc}:\mathtt C(\bb,\cc)\otimes \mathtt C(\aa,\bb)\rightarrow\mathtt C(\aa,\cc)$,
  which is an arrow in $\mathcal M$, such that 
 \begin{itemize}
 \item For any $\aa,\bb,\cc,\dd$ in $\mathtt C$, we have
 $${\pmb\circ}_{\aa\bb\dd}({\pmb \circ}_{\bb\cc\dd}\otimes 1)={\pmb\circ}_{\aa\cc\dd}(1\otimes {\pmb\circ}_{\aa\bb\cc})\alpha:(\mathtt C(\cc,\dd)
 \otimes \mathtt C(\bb,\cc))\otimes \mathtt C(\aa,\bb)\rightarrow  \mathtt C(\aa,\dd),$$
 \item For any pair of objects $(\aa,\bb)$ in $\mathtt C$, we have
$${\pmb\circ}_{\aa\bb\bb}(Id_{\bb}\otimes 1)=\lambda:I\otimes \mathtt C(\aa,\bb)\rightarrow \mathtt C(\aa,\bb),$$
\item For any pair of objects $(\aa,\bb)$ in $\mathtt C$, we have
$${\pmb\circ}_{\aa\aa\bb}(1\otimes Id_{\aa})=\rho:\mathtt C(\aa,\bb)\otimes I\rightarrow \mathtt C(\aa,\bb).$$
 \end{itemize}
 
\end{itemize}
\begin{example}\rm
 Ordinary categories are categories enriched over $\mathrm{Set}$ the category of sets with cartesian product.
\end{example} 
\begin{example}\rm
Consider the strict monoidal category $\mathbf 2$ whose objects are $\{0,1\}$ ($1$ is the identity) with 
a single nonidentity arrow $0\rightarrow 1$ and such that the cartesian product is the ordinary product of integers.
 Each preordered sets $\mathcal P$ can be seen as a category enriched over $\mathbf 2$. 
 Notice that $\mathcal P(a,b)=1$ means that $a\leq b$ and the composition encodes the transitivity. Indeed, if $\mathcal P(a,b)=1$ and $\mathcal P(b,c)=1$ then necessarily we have
 $\mathcal P(a,c)=1$ because there is no arrow $1\rightarrow 0$ in $\mathbf 2$.
\end{example}
The notions of \emph{enriched functors} and \emph{enriched natural transformations} are defined as for categories.
More precisely, if $\mathtt C$ and $\mathtt D$ are two $\mathcal M$-categories, a $\mathcal M$-functor 
$\mathtt F:\mathtt C\rightarrow \mathtt D$  assigns to each object $\aa$ of $\mathtt C$ an object $\mathtt F(\aa)$ in $\mathtt D$ . 
It also associates with each arrow $\mathtt C(\aa,\bb)$ a morphism  $\mathtt F(\aa,\bb):{}
\mathtt C(\aa,\bb)\rightarrow \mathtt D(\mathtt F(\aa),\mathtt F(\bb))$ in $\mathcal M$ such that
$\mathtt F(\aa,\aa)Id_{\aa}=Id_{\mathtt F(\aa)}:I\rightarrow \mathtt D(\mathtt F(\aa),\mathtt F(\aa))$ and
$\mathtt F(\aa,\cc){\pmb \circ}={\pmb \circ}\mathtt F(\bb,\cc)\otimes \mathtt F(\aa,\bb):\mathtt C(\bb,\cc)\otimes \mathtt C(\aa,\bb)\otimes \mathtt D(\mathtt F(\aa),\mathtt F(\bb))$.

Let $\mathtt C$ and $\mathtt D$ be two enriched categories and $\mathtt F, \mathtt G:\mathtt C\rightarrow\mathtt D$ be two enriched functors. 
An \emph{enriched natural transformation}  $\pmb\eta:\mathtt F\rightarrow\mathtt G$ is a family of morphisms 
$\pmb\eta_{\aa}:I\rightarrow \mathtt D(\mathtt F(\aa),\mathtt G(\aa))$ of morphism of $\mathcal M$ defined for each
object $\aa$ in $\mathtt C$ such that for any two objects $\aa$ and $\bb$ of $\mathtt C$ we have
${{\pmb \circ}}_{\mathtt D}(\mathtt G(\aa,\bb)\otimes\pmb\eta_{\aa})\rho^{-1}=
{{\pmb \circ}}_{\mathtt D}(\pmb \eta_{\bb}\otimes \mathtt F(\aa,\bb))\lambda^{-1}:\mathtt C(\aa,\bb)\rightarrow 
\mathtt D(\mathtt F(\aa),\mathtt G(\bb))$.

Recall that a monoidal category  $\mathcal M$ is \emph{symmetric}
 if there exists a natural isomorphism $\mathrm{com}$ such that for any object $a,b\in\mathcal M$, $\mathrm{com}:a\otimes b
 \simeq b\otimes a$ satisfying
 \begin{itemize}
\item{\it The first hexagonal identity}.  $\alpha\ \mathrm{com}\ \alpha=(1\otimes \mathrm{com})\alpha(\mathrm{com}\otimes 1)$,
\item{\it Self inversibility.} $\mathrm{com}\ \mathrm{com}=1$.
 \end{itemize}
 Notice that these two identities imply the second hexagonal identity which reads 
 $\alpha^{{-1}}\ \mathrm{com}\ \alpha^{-1}=(\mathrm{com}\otimes 1)\alpha^{-1}(1\otimes\mathrm{com})$. 
Let $(\mathcal M,\otimes,\alpha,\lambda,\rho,I,\mathrm{com})$ be a symmetric monoidal category 
 and $\mathtt C$ be a $\mathcal M$-category.
 We consider the natural transformation $\phi$ defined by $\phi:=\alpha^{-1}(1\otimes\alpha)(1\otimes(\mathrm{com}\otimes 1))
 (1\otimes\alpha^{-1})\alpha$.
We denote $\mathtt C\times\mathtt D$ the structure consisting of
\begin{itemize}
\item a class $\obj(\mathtt C\times\mathtt D)$ of objects of $\mathtt C\times \mathtt D$ which are the pairs
 $[\aa,\bb]$ of objects,
\item the arrows are  objects $(\mathtt C\times\mathtt D)([\aa_1,  \bb_1],[\aa_2,\bb_2])=\mathtt C(\aa_1,\aa_2)\otimes{}
 \mathtt D(\bb_1,\bb_2)$ 
of $\mathcal M$  such that $\aa_1, \aa_2$   are objects in $\mathtt C$ (resp. $\bb_1, \bb_2$ are objects in $\mathtt D$)
and $\mathtt C(\aa_1,\aa_2)$ (resp.  $\mathtt D(\bb_1,\bb_2)$) are arrows in $\mathtt C$ (resp. $\mathtt D$),
\item For each $O_1=[\aa_1,\bb_1], O_2=[\aa_2,\bb_2], O_3=[\aa_3,\bb_3]$ objects in $\mathtt C\times \mathtt D$, we define the arrow 
${{\pmb \circ}}_{O_1, O_2, O_3}: (\mathtt C\times\mathtt D)[O_2,O_3]\otimes (\mathtt C\times\mathtt D)[O_1,O_2]\longrightarrow 
(\mathtt C\times\mathtt D)[O_1,O_3]$ in $\mathcal M$ 
designating a \emph{composition} and defined by ${{\pmb \circ}}_{O_1, O_2, O_3}=({{\pmb \circ}}_{\aa_1,\aa_2,\aa_3}\otimes{{\pmb \circ}}_{\bb_1,\bb_2,\bb_3})\phi$.
Notice that $\phi$ is the isomorphism $$(\mathtt C(\aa_2,\aa_3)\otimes \mathtt D(\bb_2,\bb_3))\otimes(\mathtt C(\aa_1,\aa_2)
\otimes \mathtt D(\bb_1,\bb_2))
\longrightarrow (\mathtt C(\aa_2,\aa_3)\otimes\mathtt C(\aa_1,\aa_2))\otimes (\mathtt D(\bb_2,\bb_3)\otimes\mathtt D(\bb_1,\bb_2)) .$$

\end{itemize}
\begin{proposition}
If $\mathtt C$ and $\mathtt D$ are $\mathcal M$-categories then
$\mathtt C\times\mathtt D$ is a $\mathcal M$-category
\end{proposition}
\begin{proof}
Let $O_i=[\aa_i,\bb_i]$, $i=1\dots 4$, be four objets of $\mathtt C\times \mathtt D$.
Consider the map
\[
{{\pmb \circ}}_{O_1,O_2,O_4}({{\pmb \circ}}_{O_2,O_3,O_4}\otimes 1):\left((\mathtt C\times\mathtt D)(O_3,O_4)\otimes 
(\mathtt C\times\mathtt D)(O_2,O_3)\right)\otimes (\mathtt C\times\mathtt D)(O_1,O_2)\longrightarrow
(\mathtt C\times\mathtt D)(O_1,O_4). 
\]
We have
\[\begin{array}{rcl}
{{\pmb \circ}}_{O_1,O_2,O_4}({{\pmb \circ}}_{O_2,O_3,O_4}\otimes 1)&=&
({{\pmb \circ}}_{\aa_1\aa_2\aa_4}\otimes{{\pmb \circ}}_{\bb_1\bb_2\bb_4})\phi\left(
({{\pmb \circ}}_{\aa_2\aa_3\aa_4}\otimes{{\pmb \circ}}_{\bb_2\bb_3\bb_4	})\phi\otimes1\right)\\&=&
\left(\left({{\pmb \circ}}_{\aa_1\aa_2\aa_4}({{\pmb \circ}}_{\aa_2\aa_3\aa_4}\otimes 1)\right)\otimes\left(
{{\pmb \circ}}_{\bb_1\bb_2\bb_4}({{\pmb \circ}}_{\bb_2\bb_3\bb_4}\otimes 1)\right)\right)\phi(\phi\otimes 1).
\end{array}
\]
But, since $\mathcal M$ is monoidal one has
\[
{{\pmb \circ}}_{\aa_1\aa_2\aa_4}({{\pmb \circ}}_{\aa_2\aa_3\aa_4}\otimes 1)={{\pmb \circ}}_{\aa_1\aa_3\aa_4}(1\otimes {{\pmb \circ}}_{\aa_1\aa_2\aa_3})
\alpha \mbox{ and }
{{\pmb \circ}}_{\bb_1\bb_2\bb_4}({{\pmb \circ}}_{\bb_2\bb_3\bb_4}\otimes 1)={{\pmb \circ}}_{\bb_1\bb_3\bb_4}(1\otimes {{\pmb \circ}}_{b_1b_2b_3})\alpha.
\]
Noting that 
\[
(\alpha\otimes\alpha)\phi(\phi\otimes 1)=\phi(1\otimes\phi)\alpha,
\]
from
\[
(({{\pmb \circ}}_{\aa_1,\aa_3,\aa_4}(1\otimes {{\pmb \circ}}_{\aa_1,\aa_2,\aa_3}))\otimes ({{\pmb \circ}}_{\bb_1,\bb_3,\bb_4}(1\otimes {{\pmb \circ}}_{\bb_1,\bb_2,\bb_3}))\phi=
({{\pmb \circ}}_{\aa_1,\aa_3,\aa_4}\otimes {{\pmb \circ}}_{\bb_1\bb_3\bb_4})\phi(1\otimes({{\pmb \circ}}_{\aa_1\aa_2\aa_3}\otimes{{\pmb \circ}}_{\bb_1\bb_2\bb_3})),
\]
we deduce
\begin{equation}\label{CxCenri1}{{\pmb \circ}}_{O_1O_2O_4}({{\pmb \circ}}_{O_2O_3O_4}\otimes 1)={{\pmb \circ}}_{O_1O_3O_4}(1\otimes {{\pmb \circ}}_{O_1O_2O_3})\alpha.\end{equation}
Furthermore, since $Id_{O_i}=Id_{\aa_i}\otimes Id_{\bb_i}$, we have
\begin{equation}\label{CxCenri2}
\begin{array}{rcl}
{{\pmb \circ}}_{O_1O_2O_2}(Id_{O_2}\otimes 1)&=&({{\pmb \circ}}_{\aa_1\aa_2\aa_2}\otimes{{\pmb \circ}}_{\bb_1\bb_2\bb_2})\left((Id_{\aa_2}\otimes 1)
\otimes(Id_{\bb_2}\otimes 1)
\right)\phi\alpha^{-1}\lambda^{-1}\\
&=&(\lambda\otimes\lambda)\phi\alpha^{-1}\lambda^{-1}=\lambda
\end{array}
\end{equation}
and
\begin{equation}\label{CxCenri3}
\begin{array}{rcl}
{{\pmb \circ}}_{O_1O_1O_2}(1\otimes Id_{O_1})&=&({{\pmb \circ}}_{\aa_1\aa_1\aa_2}\otimes{{\pmb \circ}}_{\bb_1\bb_1\bb_2})\left((1\otimes Id_{\aa_1})
\otimes(1\otimes
Id_{\bb_1})\right)\phi\rho^{-1}\lambda^{-1}\\
&=&(\rho\otimes\rho)\phi\alpha^{-1}\rho^{-1}=\rho.
\end{array}
\end{equation}
Equations (\ref{CxCenri1}), (\ref{CxCenri2}) and (\ref{CxCenri3}) show that $\mathtt C\times\mathtt D$ is an enriched category.
\end{proof}
An \emph{enriched $\mathcal M$-monoidal category} $\mathtt M$ is a $\mathcal M$-category 
equiped with a bifunctor $\pmb \otimes:\mathtt M\times\mathtt M\rightarrow \mathtt M$, an object I called the unit object
and three natural $\mathcal M$-isomorphisms
\begin{enumerate}
\item the associator $\pmb \alpha$ with components $\pmb\alpha_{\aa,\bb,\cc}: (\aa\pmb \otimes\bb)\pmb \otimes\cc\simeq \aa\pmb
\otimes(\bb\pmb \otimes \cc)$,
\item the left unitor $\pmb\lambda$ with components $\pmb\lambda_{\aa}:I\pmb \otimes\aa\simeq \aa$,
\item the right unitor $\pmb\rho$ with components $\pmb\rho_{\aa}:\aa\pmb \otimes I\simeq \aa$
\end{enumerate}
satisfying 
$(1\pmb \otimes\pmb\alpha)\pmb\alpha(\pmb{\alpha}\pmb \otimes 1)=\pmb\alpha\ \pmb\alpha:((\aa\pmb \otimes\bb)\pmb \otimes\cc)\pmb \otimes\dd\rightarrow
 \aa\pmb \otimes(\bb\pmb \otimes (\cc\pmb \otimes \dd))$ and 
$(1\pmb \otimes\pmb\lambda)\pmb\alpha=\pmb\rho\pmb\otimes 1:(\aa\pmb \otimes I)\pmb \otimes \bb\rightarrow \aa\pmb \otimes\bb$ 
for any objects
 $\aa, \bb, \cc, \dd$ in $\mathtt M$.
A enriched $\mathcal M$-monoidal category is said \emph{strict} if the natural transformations $\pmb \alpha, \pmb \lambda, \pmb \rho$ are identities.
Hence, we can now define the notion of  \emph{enriched PRO} as a strict $\mathcal M$-monoidal category whose objects are the natural numbers and the
tensor product sends $(m,n)$ to $m+n$. 

Let $\mathbb K$ be a commutative semiring. The class of $\mathbb K$-modules forms a category which is  equipped with a tensor product
conferring to it a structure of symmetric monoidal category. For any 
$\mathbb K$-module $M$,
the identity object is $\mathbb K$. The morphism $\lambda_{M}$ (resp. $\rho_{M}$) sends $(a,m)$ (resp. $(m,a)$) to $a\cdot m$.
If $M$, $N$, and $P$ are three $\mathbb K$-modules, the morphism $\alpha_{M,N,P}$ sends $((m,n),p)$ to $(m,(n,p))$.
\begin{definition}
	A $\mathbb K$-ModPro is a PRO enriched in the category of the $\mathbb K$-modules.
\end{definition}
Notice that the compositions  and the tensor product  are morphisms of 
$\mathbb K$-modules.\\
This definition is equivalent to the definition given in Section 
\ref{ModPro}.
\begin{itemize}
	\item The operation $\leftrightarrow$ is assimilated to the bifunctor $\pmb \otimes$. Hence, it sends 
	$(\mathtt p,\mathtt q+\mathtt  r)\in Hom(m,n)\otimes Hom(m',n')$ to 
	$\mathtt p\leftrightarrow (\mathtt q+\mathtt  r)  \in Hom(m+m',n+n')$ and $(\mathtt p,\mathtt q)+ (\mathtt p,\mathtt r)$ to 
	$(\mathtt p\leftrightarrow \mathtt q)+ (\mathtt p\leftrightarrow \mathtt r)$. 
	Since $(\mathtt p,\mathtt q+\mathtt  r)=(\mathtt p,\mathtt q)+(\mathtt p,\mathtt r)$, we deduce  that 
	$\mathtt p\leftrightarrow (\mathtt q+\mathtt r)=
	(\mathtt p\leftrightarrow\mathtt q)+(\mathtt p\leftrightarrow\mathtt r)$.
	In the same way, the right distributivity is deduced from the fact that $\leftrightarrow=\pmb\otimes$.{}
	\item The operation $\updownarrow$ is assimilated to the composition $\pmb \circ$. Hence, it sends 
	$(\mathtt p,\mathtt q+\mathtt  r)\in Hom(m,n)\otimes Hom(n,p)$ to 
	$\begin{array}{c}\mathtt p\\\updownarrow\\ (\mathtt q+\mathtt  r)\end{array}  \in Hom(m,p)$
	 and $(\mathtt p,\mathtt q)+(\mathtt p,\mathtt r)$ to
	 $\begin{array}{c}\mathtt p\\\updownarrow\\\mathtt  q\end{array}+ \begin{array}{c}\mathtt p\\\updownarrow\\\mathtt 
	  r\end{array}$. 
	Since $(\mathtt p,\mathtt q+\mathtt r)=(\mathtt p,\mathtt q)+ (\mathtt p,\mathtt r)$, we deduce  that 
	$
\begin{array}{c}
(\mathtt p_1+\mathtt p_2)\\
\updownarrow\\
\mathtt q
\end{array}=\left(\begin{array}{c}
\mathtt p_1\\
\updownarrow\\
\mathtt q
\end{array}\right)+
\left(\begin{array}{c}
\mathtt p_2\\
\updownarrow\\
\mathtt q
\end{array}\right),\,
\begin{array}{c}
\mathtt q\\
\updownarrow\\
(\mathtt p_1+\mathtt p_2)
\end{array}$.
	In the same way, the down distributivity is deduced from the fact that 
	$\updownarrow=\pmb\circ$.
\end{itemize}

The following table summarize the correspondance between the categorial definition and the constructive definition.
\[
\begin{array}{|c|c|}
\hline
\mbox{Constructive definition}&\mbox{Categorial definition}\\
\hline
\mathtt M_{m,n}&\mathtt M(m,n)\\
\updownarrow&{{\pmb \circ}}\\
\leftrightarrow&\pmb \otimes\\
\hline
\end{array}
\]
\section{Modeling quantum gates\label{QI}}
Representation of free PRos can be used when we draw some multilinear operations as boxes. 
For instance, the composition of quantum gates, in quantum computation theory,
 can be represented by hypermatrices in $ \mathbb C(2)$. Quantum networks are devices consisting of quantum logic gates whose 
 computational steps are synchronized in times (see eg \cite{Ekert}). Graphically a quantum networks is drawn by relying quantum gates. Quantum gates 
 are identified with hypermatrices of $\mathbb C(2)_{n,n}$ acting on qubit systems identified with hypermatrices of 
 $\mathbb C(2)_{0,m}$. In the aim to match with the notation of quantum computation theory, we consider that the indices of the hypermatrices
 are taken in $\{0,1\}$ instead of $\{1,2\}$.
 The most common gate is the Hadamard gate $\boxed{H}$ acting on a single qubit. This gate is identified with the matrix
  $\mathcal H\in 
 \mathbb C(2)_{1,1}$ defined by $\displaystyle \mathop {\mathcal H}^{0}_{0}=\mathop {\mathcal H}^{0}_{1}=
 \mathop {\mathcal H}^{1}_{0}=-\mathop {\mathcal H}^{1}_{1}=\frac1{\sqrt2}$.
 Another well known gate 
 is denoted by $\boxed{c-V}$ 
  acting on two qubit systems as the hypermatrix $\mathcal V$ 
 in $\mathbb C(2)_{2,2}$ defined by 
 $\displaystyle \mathop {\mathcal V}^{ab}_{ab}=1$ 
 for any $(a,b)\neq (1,1)$,
  $\displaystyle \mathop {\mathcal V}^{11}_{11}=I$, 
   and 
  $\displaystyle \mathop {\mathcal V}^{ab}_{cd}=0$ 
  for $(a,b)\neq (c,d)$. Some important gates can be obtained by sticking these elementary gates in quantum networks.
 
  Let us give a few examples. The controlled-not  gate $\boxed{C}$ is one of the
   most popular two-qubit gates because it allows to create entanglement from
  non entangled systems. Graphically, the network is drawn in Figure 
  \ref{Quantum Network}.
  \begin{figure}[h]\[{}
  \begin{tikzpicture}
\Apro 03{0.5}{0.5}11H{white}{H}
\Apro 02{1}{0.5}22{$c-V$}{white}{V}
\Apro 0{1}{1}{0.5}22{$c-V$}{white}{VV}
\Apro 00{0.5}{0.5}11H{white}{HH}
\draw[fill=black] (0.25,4) circle [radius=0.03];
\draw[fill=black] (0.75,4) circle [radius=0.03];
\draw[fill=white] (0.25,-1) circle [radius=0.03];
\draw[fill=white] (0.75,-1) circle [radius=0.03];

\draw (uuH1)--(0.25,4);
\draw (uuV2)--(0.75,4);
\draw (ddH1)--(uuV1);
\draw (uuVV2)--(ddV2);
\draw (uuVV1)--(ddV1);
\draw (ddHH1)--(0.25,-1);
\draw (ddVV2)--(0.75,-1);
\draw (uuHH1)--(ddVV1);
\end{tikzpicture}
  \]
  \caption{An example of Quantum Network.\label{Quantum Network} }
  \end{figure}
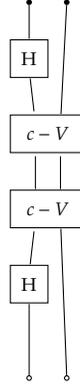
  Indeed, all works as if we consider the free Pro generated by $\boxed{H}$ and $\boxed{c-V}$ together with the representation $\mu${}
  sending $\boxed{H}$ to $\mathcal H$ and $\boxed{c-V}$ to $\mathcal V$. The hypermatrix $\mathcal C=\mu\left(\boxed{C}\right)$ can
   be computed by
  $\displaystyle\mathop {\mathcal C}^{a,b}_{c,d}=\sum_{e_1,e_{2},e_{3},e_{4}}
  \displaystyle\mathop {\mathcal H}^{a}_{e_{1}}\mathop{\mathcal V}^{e_{1},b}_{e_{2},e_{3}}\mathop{\mathcal V}^{e_{2},e_{3}}_{e_{4}d}
  \mathop{\mathcal H}^{e_{4}}_{c}$. A fast calculation gives $\displaystyle\mathop {\mathcal C}^{a,b}_{c,d}=\delta_{b,d}\delta_{c,a+b\mod 2}$.
In this context a (pure) $k$-qubit system $|\phi\rangle=\sum_{i_{1},\dots,i_{k}}\alpha_{i_{1},\dots,i_{k}}|i_{1}\cdots i_{k}\rangle$
is seen as the hypermatrix $\phi\in\mathbb C(2)_{0,k}$
 satisfying $\displaystyle\mathop\phi_{i_{1},\dots,i_{k}}=\alpha_{i_{1},\dots,i_{k}}$.
The action of a quantum network $\boxed N$ on a system $|\phi\rangle$ is just the branching $\begin{array}{c}\phi\\\updownarrow\\\mathcal N\end{array}$, where $\mathcal N$ is the hypermatrix associated to
$N$. For instance, let $|\phi\rangle=|01\rangle+|11\rangle$. The matrix $\begin{array}{c}\phi\\\updownarrow\\\mathcal C\end{array}$ is associated to the quantum system $|\phi'\rangle=|00\rangle+|11\rangle$.{}
Notice that $|\phi'\rangle$ is entangled but not $|\phi\rangle.$
\section{On Temperley-Lieb algebras \label{TL}}

\tikzset{
  pt/.style={insert path={node[scale=1]{.}}},
  dnup/.style={insert path={ [pt] .. controls +(0,1) and +(0,-1) .. +(#1,2) [pt]}},
  dndn/.style={insert path={ [pt] .. controls +(0,0.5) and +(0,0.5) .. +(#1,0) [pt]}},
  upup/.style={insert path={ [pt] .. controls +(0,-0.5) and +(0,-0.5) .. +(#1,0) [pt]}},
}

\def\upup{
\begin{tikzpicture}
	\draw(0,0)[upup=0.5];
\end{tikzpicture}
}

\def\dndn{
\begin{tikzpicture}
	\draw(0,-0.5)[dndn=0.5];
\end{tikzpicture}
}
We consider the graded set $\mathcal X=\mathcal X_{2,0}\cup\mathcal X_{0,2}$ with
$
\mathcal X_{2,0}=\left\{\upup\right\}
$
and
$
\mathcal X_{0,2}=\left\{\dndn\right\}
$.

We define the diagram ModPro $\mathcal D=\mathbb C\langle\mathcal X\rangle/_{\equiv_{D}}$ where $\equiv_{D}$ is the congruence
generated by 
\begin{equation}\label{antinoeud}
\begin{array}{c} \dndn\leftrightarrow |\\
\updownarrow\\
|\leftrightarrow\upup\end{array}\equiv_{D}
\begin{array}{c} |\leftrightarrow \dndn\\
\updownarrow\\
\upup\leftrightarrow|\end{array}\equiv_{D}|.
\end{equation}
Let
\begin{equation}
\mathfrak U_{i}^{(n)}=\overbrace{|\leftrightarrow\cdots\leftrightarrow|}^{i-1\times}\leftrightarrow\begin{array}{c}\upup\\\updownarrow\\\dndn\end{array}
\leftrightarrow \overbrace{|\leftrightarrow\cdots\leftrightarrow|}^{n-i-1\times}\in \mathcal D_{n,n},
\end{equation}
for $1\leq i<n$.
These elements satisfy the commutations $\mathfrak U_{i}^{(n)}\mathfrak U_{j}^{(n)}=\mathfrak U_{j}^{(n)}\mathfrak U_{i}^{(n)}$ 
for $|i-j|>1$ together with reduction of braids
\begin{equation}\begin{array}{c}\mathfrak U^{(n)}_{i}\\\updownarrow\\\mathfrak U^{(n)}_{i+1}\\\updownarrow\\\mathfrak U^{(n)}_{i}\end{array}=\mathfrak U^{(n)}_{i}\end{equation}
and
\begin{equation}
\begin{array}{c}\mathfrak U^{(n)}_{i+1}\\\updownarrow\\\mathfrak U^{(n)}_{i}\\\updownarrow\\\mathfrak U^{(n)}_{i+1}\end{array}=\mathfrak U^{(n)}_{i+1}.
\end{equation}
We denote by $\mathcal T=\mathcal D/_{\equiv_{T}}$, the quotient of $\mathcal D$ by the congruence generated by 
\begin{equation}\label{drel}\begin{array}{c}\dndn\\\updownarrow\\\upup\end{array}\equiv_T d
\end{equation} for some $d\in\mathbb C$. If we denote  $\hat{\mathfrak U}_{i}^{(n-1)}=\phi_{\equiv_{T}}\left(\mathfrak U_{i}^{(n-1)}\right)$,{}
then one has an additional relation 
\begin{equation}\label{quadratique}
\begin{array}{c}\hat{\mathfrak 
U}_{i}^{(n-1)}\\\updownarrow\\\hat{\mathfrak U}_{i}^{(n-1)}\end{array}=d\hat{\mathfrak U}_{i}^{(n-1)}.
\end{equation}
Notice that $(T_{n,n},\updownarrow,+)$ is an algebra and the subsalgebra $\mathcal T\ell_{n}$ generated by the elements 
$\hat{\mathfrak U}_{i}^{(n-1)}$ is isomorphic to the Temperley-Lieb algebra (see e.g. \cite{TempLieb}). As a consequence, any multilinear representation
of $\mathcal T$ of dimension $N$ gives a representation of $\mathcal T\ell_{n}$ of dimension $N^{n}$.\\
Hence to find a representation of the Temperley-Lieb algebra, it suffices to exhibit a representation $\mu$ of $\mathcal F(\mathcal X)$
such that 
\begin{equation}
\mu\left(\begin{array}{c}\dndn\\\updownarrow\\\upup\end{array}\right)=d\mbox{ and }
\mu\left(\begin{array}{c} \dndn\leftrightarrow |\\
\updownarrow\\
|\leftrightarrow\upup\end{array}\right)=\mu\left(\begin{array}{c} |\leftrightarrow \dndn\\
\updownarrow\\
\upup\leftrightarrow|\end{array}\right)=I_{N^{n}\times N^{n}}.
\end{equation}
For instance, one can deduce a representation of dimension $2^{n}$ of 
each 
Temperley-Lieb algebra $\mathcal T\ell_{n}$ from
 the dimension $2$ representation $\mu$ of $\mathcal F(\mathcal X)$ defined by
\[{}
\mu\left(\dndn\right)=\left(\begin{array}{cccc}\ ^{11}&\ ^{12}&\ ^{21}&\ ^{22}\\
2-d&0&d-2&1\end{array}
\right)
\]
and
\[{}
\mu\left(\upup\right)=\begin{array}{c}\ ^{11}\\\ ^{12}\\\ ^{21}\\\ ^{22} \end{array}\left(\begin{array}{c}
\frac1{2-d}\\0\\1\\1
\end{array}
\right).
\]
\def\cycle{\mathrm{Cycle}}
\def\even{\mathrm{ntriv}}
\def\odd{\mathrm{triv}}
Let $\mathfrak d$ be a diagram in $\mathcal D_{n,n}$, we  denote by $\cycle(\mathfrak d)$ the design obtained from $\mathfrak d$ by relying the $i$th input to
the $i$th output for each $1\leq i\leq n$. 
We remark that $\cycle(\mathfrak d)$ is the juxtaposition of connected components. Some components contain at least two generators of $\mathcal X$, we will say that such a component is \emph{non trivial}. The \emph{trivial} components come from a single $|$ in $\mathfrak d$.
We denote by 
$\even(\mathfrak d)$ (resp. $\odd(\mathfrak d)$) the number of non trivial (resp. odd) connected components in $\cycle(\mathfrak d)$ 
number of occurrences of $|$.\\
For instance,
\[{}
  \mbox{ for }\mathfrak d=
  \left.\begin{array}{c}\\ \\ \\ \\ \\ \\ \end{array}\right.{}
\begin{array}{c} \\ 
\begin{tikzpicture}
    \draw (1,-1) [dnup=1];
    \draw (1,1) [upup=0.5];
    \draw (1.5,-1) [dnup=1];
    \draw (2,-1) [dndn=0.5];
    \draw (3,-1) [dnup=0];
    \draw (3.5,-1) [dnup=0];
  \end{tikzpicture}\end{array}\left.\begin{array}{c}\\ \\ \\ \\ \\ 
  \\\end{array} \right. \mbox{, one has }
\cycle(\mathfrak d)
=\begin{array}{c} \\ 
\begin{tikzpicture}
    \draw (1,0) [dnup=1];
    \draw (1,2) [upup=0.5];
    \draw (1.5,0) [dnup=1];
    \draw (2,0) [dndn=0.5];
    \draw (3,0) [dnup=0];
    \draw (3.5,0) [dnup=0];
    \draw[dashed] (1,0)  to [bend left](1,2) ;
    \draw[dashed] (1.5,0)  to [bend left](1.5,2) ;
    \draw[dashed] (2,0)  to [bend left](2,2) ;
    \draw[dashed] (2.5,0)  to [bend left](2.5,2) ;
    \draw[dashed] (3,0)  to [bend left](3,2) ;
    \draw[dashed] (3.5,0)  to [bend left](3.5,2) ;
  \end{tikzpicture}\end{array},
\]
\def\tr{\mathrm{tr}}
$\even(\mathfrak d)=1$, and $\odd(\mathfrak d)=2$.
Let $\tr$ denotes the trace  computed by considering the hypermatrices as $N^{n}\times N^{n}$-matrices.
Remark that $\tr(\mu(\mathfrak d\leftrightarrow \mathfrak d'))=
	\tr(\mu(\mathfrak d)\leftrightarrow\mu(\mathfrak d'))=\tr(\mu(\mathfrak d))\tr(\mu(\mathfrak d'))$ (
	because the trace of a Kronecker product is the product of the trace).
	Hence, we deduce $\tr\left(\mu\left(\overbrace{|\dots|}^{n\times}\right)\right)=N^{n}$. We also have
	\[{}\begin{array}{rcl}
	\tr(\mu(U_{i}))&=&\tr\left(\mu\left(\begin{array}{c}|\leftrightarrow\cdots\leftrightarrow|\leftrightarrow \upup\leftrightarrow|{}
	\leftrightarrow\cdots\leftrightarrow|\\\updownarrow\\ |\leftrightarrow\cdots\leftrightarrow|\leftrightarrow \dndn\leftrightarrow|{}
	\leftrightarrow\cdots\leftrightarrow|\end{array}\right)\right)
		=\tr\left(\mu\left(\begin{array}{c}\upup\\\updownarrow\\  \dndn\end{array}\right)\right){}
		\tr\left(\mu\left(\overbrace{|\cdots|}^{n-2\times}\right)\right)\\
		&=&\tr\left(\mu\left(\begin{array}{c}\dndn\\\updownarrow\\  \upup\end{array}\right)\right)N^{n-2}=N^{n-2}d.
	\end{array}
	\]
	In a similar way, one shows that for $|j-i|>1$ we have
	\[{}
	\tr\left(\mu\left(\begin{array}{c}\mathfrak U_{i}\\\updownarrow\\\mathfrak  U_{j}\end{array}\right)\right)=N^{n-4}d^{2}.
	\]
	Furthermore, for any $i<n$ one has
	\[{}
	\tr(\mu(\mathfrak  U_{i}))=\tr\left(\mu\left(\begin{array}{c}\mathfrak  U_{i}\\\updownarrow\\\mathfrak  U_{i+1}\\\updownarrow\\
	\mathfrak  U_{i}\end{array}
	\right)\right)={}
	\tr\left(\mu\left(\begin{array}{c}\mathfrak  U_{i}\\\updownarrow\\\mathfrak  U_{i}\\\updownarrow\\\mathfrak  U_{i+1}\end{array}
	\right)\right)=d\tr\left(\mu\left(\begin{array}{c}\mathfrak U_{i}\\\updownarrow\\\mathfrak  U_{i+1}\end{array}\right)\right).
	\]
	And then
	\[{}
	\tr\left(\mu\left(\begin{array}{c}\mathfrak U_{i}\\\updownarrow\\\mathfrak  U_{i+1}\end{array}\right)\right)={}
	\tr\left(\mu\left(\begin{array}{c}\mathfrak U_{i+1}\\\updownarrow\\\mathfrak U_{i}\end{array}\right)\right)=\frac1d \tr(\mu(\mathfrak  U_{i}))=N^{n-2}.
	\]
	More generally, one conjectures
\begin{conj}
Let $\mu$ be a representation of  $\mathcal T$ of dimension $N$. Then
\[{}
\tr(\mu(\mathfrak d))=N^{\odd(\mathfrak d)}d^{\even(\mathfrak d)},
\]
for any $\mathfrak d\in\mathcal \mathcal T\ell_{n}$.
\end{conj}

\end{document}